\documentclass[12pt,twoside,english,a4paper,numbers=endperiod,reqno]{amsart}

\usepackage{microtype}
\usepackage{siunitx}
\usepackage{xspace}
\usepackage{amsmath,amsfonts,amsthm,amssymb}
\usepackage{tikz}
\usepackage{tikz-cd}
\usepackage{enumitem}
\usepackage{subcaption}
\usepackage{dsfont}
\usepackage{ytableau}
\usepackage{biblatex}
\usepackage{mathtools}
\usepackage{hyperref}
\usepackage{xurl}
\hypersetup{breaklinks=true,hidelinks}
\usepackage{setspace}
\usepackage{appendix}
\usepackage{dynkin-diagrams}
\usepackage{geometry}
\usepackage{fancyhdr}

\fancypagestyle{normal}{%
    
	\fancyhf{}%
	\fancyhead[CE]{\nouppercase{\leftmark}}
	\fancyhead[CO]{\nouppercase{Multiprojective Seshadri stratifications for Schubert varieties and SMT}}
	\fancyfoot[CE,CO]{\thepage}
}
\fancypagestyle{empty}{%
    
	\fancyhf{}%
	\fancyhead[CE]{}
	\fancyhead[CO]{}
	\fancyfoot[CE,CO]{\thepage}
}
\setlist{itemsep=4pt}

\title{Multiprojective Seshadri stratifications for Schubert varieties and standard monomial theory}
\author{Henrik Müller}
\address{Department Mathematik/Informatik, Universität zu Köln, Weyertal 86-90, 50931 Cologne, Germany}
\email{\href{henrikmueller.math@gmail.com}{henrikmueller.math@gmail.com}}
\date{}

\theoremstyle{plain}
\newtheorem{theorem}{Theorem}[section]
\numberwithin{theorem}{section}
\newtheorem{lemma}[theorem]{Lemma}
\newtheorem{corollary}[theorem]{Corollary}
\newtheorem{proposition}[theorem]{Proposition}
\newtheorem{conjecture}[theorem]{Conjecture}

\theoremstyle{definition}
\newtheorem{definition}[theorem]{Definition}
\newtheorem{example}[theorem]{Example}

\newtheorem{remark}[theorem]{Remark}

\begin{document}

\setstretch{1.07}

\newcommand{\Abk}[1][Abk]{#1.\xspace}
\newcommand{\ie}{\Abk[i.\,e]}
\newcommand{\wwlog}{\Abk[w.\,l.\,o.\,g]}
\newcommand{\wrt}{\Abk[w.\,r.\,t]}
\newcommand{\WWlog}{\Abk[W.\,l.\,o.\,g]}
\newcommand{\loccit}{\textit{loc.\,cit.}}

\newcommand{\B}{\mathbb{B}}
\renewcommand{\F}{\mathbb{F}}
\renewcommand{\K}{\mathbb{K}}
\renewcommand{\A}{\mathbb{A}}
\renewcommand{\N}{\mathbb{N}}
\newcommand{\Z}{\mathbb{Z}}
\newcommand{\PP}{\mathbb{P}}
\newcommand{\Q}{\mathbb{Q}}
\newcommand{\R}{\mathbb{R}}

\newcommand{\set}[1]{\{#1\}}
\newcommand{\Set}[1]{\left\{#1\right\}}

\newcommand{\id}{\mathrm{id}}
\newcommand{\Spec}{\operatorname{Spec}}
\newcommand{\Proj}{\operatorname{Proj}}
\newcommand{\Multiproj}{\operatorname{Multiproj}}
\newcommand{\supp}{\operatorname{supp}}
\newcommand{\wt}{\operatorname{wt}}
\newcommand{\ch}{\operatorname{char}}
\newcommand{\longhookrightarrow}{\lhook\joinrel\longrightarrow}
\newcommand{\DCP}[1]{D_{{#1}}(\underline\lambda, \tau)}
\newcommand{\Wlambda}{W(\underline\lambda, \tau)}
\newcommand{\ulWlambda}{\underline W(\underline\lambda, \tau)}
\newcommand{\ulW}{\underline W}
\newcommand{\nice}[1]{$#1$-standard}
\newcommand{\up}[1]{{#1}^\vartriangle}
\newcommand{\down}[1]{{#1}^\triangledown}
\newcommand{\LST}{\mathbb B(\underline\lambda, \mathcal I)_{\tau, \underline d}}

\newlist{abbrv}{itemize}{1}
\setlist[abbrv,1]{label=,labelwidth=1in,align=parleft,leftmargin=!,noitemsep}

\newcommand{\whatisthis}{article}

\numberwithin{equation}{section}

\setcounter{secnumdepth}{3}


\begin{abstract}
    Using the language of Seshadri stratifications we develop a geometrical interpretation of Lakshmibai-Seshadri-tableaux and their associated standard monomial bases. These tableaux are a generalization of Young-tableaux and De-Concini-tableaux to all Dynkin types. More precisely, we construct filtrations of multihomogeneous coordinate rings of Schubert varieties, such that the subquotients are one-dimensional and indexed by standard tableaux.
\end{abstract}

\maketitle
\thispagestyle{empty}

\section{Introduction}

In the 1940s, Hodge described a basis of the homogeneous coordinate ring of a Grassmann variety $\mathrm{Gr}(d,n)$ via certain products of Pl\"ucker coordinates (\cite{hodge1943some}, \cite{hodge1994methods}). The basis vectors correspond to semistandard Young-tableaux with exactly $d$ rows and entries in $\set{1, \dots, n}$. This is the first example of what is known as a \textit{standard monomial theory} (SMT). However there still exists no clear definition what a standard monomial theory really is, this term rather refers to specific examples, which usually come from the representation theory of semisimple algebraic groups or Lie algebras. Given an algebra generated by a finite set $S$, the set of all monomials in $S$ generate this algebra as a vector space. One tries to extract a basis from this generating set via combinatorial methods. The basis vectors are then called \textit{standard} and every monomial in $S$ not belonging to this basis is called \textit{non-standard}.

In their series of papers (\cite{lakshmibaiGP1}, \cite{lakshmibaiGP2}, \cite{lakshmibaiGP3}, \cite{lakshmibaiGP4}, \cite{lakshmibaiGP5}, ...) Lakshmibai, Musili and Seshadri generalized the work of Hodge to Schubert varieties in classical Dynkin types. They found a standard monomial basis of the multihomogeneous coordinate ring of a Schubert variety $X$ with respect to the embedding
\begin{align*}
    X \hookrightarrow \prod_{i=1}^m \PP(V(\omega_i)),
\end{align*}
where $V(\omega_i)$ are fundamental representations of the algebraic group. This basis is indexed by sequences of Weyl group cosets, which can be lifted to a weakly decreasing sequence in the Weyl group, called \textit{defining chain}. The path model of Littelmann -- more specifically the path model of LS-paths -- developed in \cite{littelmann1994littlewood} and \cite{littelmann1995paths} provided a suitable language for this index set, such that the SMT of Lakshmibai, Musili and Seshadri could be generalized to arbitrary Dynkin types (\cite{littelmann1995plactic}). To each LS-path one associates a function called \textit{path vector}, which Littelmann constructed in~\cite{littelmann1998contracting} using quantum Frobenius splitting. Standard monomials in these path vectors are indexed by sequences of LS-paths which admit a weakly decreasing lift to the Weyl group. This leads to the notion of what we call an \textit{LS-tableau} (see Section~\ref{sec:LS-tableaux}), a generalization of Young-tableaux. 

Since the discovery of this combinatorially defined standard monomial basis, it has attracted much attention and a large amount of citations and applications. As the multihomogeneous coordinate ring is an algebraic-geometric object, it is a natural question whether the SMT can also be derived using geometric methods. This leads to the main theorem of this \whatisthis{}.

\vskip 10pt
\noindent
\textbf{Theorem} (Proposition~\ref{prop:bijection_YT_fan} and Theorem~\ref{thm:standard_monomial_basis}). There exists a quasi-valuation $\mathcal V$ with at most one-dimensional leaves on the multihomogeneous coordinate ring of a Schubert variety, such that the elements in the image of $\mathcal V$ correspond to certain standard LS-tableaux.
\vskip 10pt

Note that there may exist different quasi-valuations on $\K[X]$, which therefore give geometrical interpretations of different SMTs.

The geometric interpretation of LS-tableaux is based on the work of Chiriv{\`i}, Fang and Littelmann in~\cite{seshstrat} and \cite{seshstratandschubvar}. They introduced the notion of a \textit{Seshadri strati\-fi\-cation} on an embedded projective variety $X \subseteq \PP(V)$, which was generalized by the author to projective varieties embedded into a product $\PP(V_1) \times \dots \times \PP(V_m)$ of projective spaces (see \cite{ownarticle}). A Seshadri stratification consists of a family $(X_p)_{p \in A}$ of projective varieties indexed by a finite, graded poset $A$ and a multihomogeneous function $f_p$, called \textit{extremal function}, in the multihomogeneous coordinate ring $\K[X]$ of $X$ for each $p \in A$. Note that the varieties $X_p$ need not be a subvarieties of $X$ itself, except in the case of an ordinary stratification, \ie for $m = 1$. Instead, each stratum $X_p$ is a subvariety of the projection of $X$ into the product $\prod_{i \in I_p} \PP(V_i)$ indexed by a non-empty subset $I_p \subseteq \set{1, \dots, m}$. The collection $\mathcal I = \set{I_p \mid p \in A}$ of these sets is called the \textit{index poset} of the stratification, which is an additional structure not visible for ordinary stratifications. 

Especially for $m > 1$ it is helpful to switch to the affine picture, when thinking about Seshadri stratifications: The affine multicone $\hat X_p$ of a stratum $X_p$ -- which by definition is a closed subvariety of the affine space $\prod_{i \in I_p} V_i$ -- can be linearly embedded into $V_1 \times \dots \times V_m$. In this way, all multicones live in the same ambient space, which is used to specify the conditions the data of varieties, functions and index sets needs to fulfill to form a Seshadri stratification. For example, the grading on $A$ is required to be compatible with the dimensions of the subvarieties, \ie $\hat X_q$ is a divisor in $\hat X_p$, if and only if $q < p$ is a covering relation in $A$.

Seshadri stratifications use a web of subvarieties in contrast to the Newton-Okounkov theoretical approach (\cite{kaveh2012newton}, \cite{lazarsfeld2009convex}), which uses a flag of subvarieties. By taking successive vanishing multiplicities along this web, every Seshadri stratification induces a quasi-valuation $\mathcal V: \K[X] \setminus \set{0} \to \Q^A$, which can be thought of as a filtration of the multihomogeneous coordinate ring $\K[X]$. In general, the quasi-valuation $\mathcal V$ is not quite canonical, as it depends on the choice of a total order $\geq^t$ linearizing the partial order on $A$. The subquotients (called \textit{leaves}) of the filtration on $\K[X]$ are at most one-dimensional and they are indexed by the image $\Gamma$ of $\mathcal V$, which is a union of finitely generated semigroups $\Gamma_{\mathfrak C}$ over all maximal chains $\mathfrak C$ in the poset $A$. Hence $\Gamma$ is called the \textit{fan of monoids} to the stratification. 

For every \textit{normal} Seshadri stratification, \ie each semigroup $\Gamma_{\mathfrak C}$ is saturated, the fan of monoids $\Gamma$ defines a standard monomial theory on the multihomogeneous coordinate ring $\K[X]$. Every element in $\Gamma$ can be uniquely decomposed as a sum of indecomposable elements. When choosing a function $x_{\underline a}$ for each indecomposable element $\underline a \in \Gamma$ then all monomials in these functions generate $\K[X]$ as a vector space and a monomial $x_{\underline a^1} \cdots x_{\underline a^s}$ is standard, if and only if $\underline a^1 + \dots + \underline a^s$ is contained in $\Gamma$. Seshadri stratifications, where this SMT does not depend on the choice of the linearization of the partial order, are called \textit{balanced}. 

In both \cite{seshstratandschubvar} and \cite{seshstratgeometric}, Chiriv{\`i}, Fang and Littelmann constructed a normal and balanced Seshadri stratification on every Schubert variety $X_\tau$, embedded into a projective space over a Demazure module. Hence they also obtain a SMT on the associated homogeneous coordinate ring. Note that the SMT by Lakshmibai, Musili and Seshadri mentioned above, gives rise to a basis of a different coordinate ring, namely the multihomogeneous coordinate ring of $X_\tau$ with respect to the embedding into the product $\prod_{\omega} \PP(V(\omega))$, where $\omega$ runs over certain fundamental weights. This raises the following question.

\vskip 10pt
\noindent
\textbf{Question}. Is there a normal and balanced Seshadri stratification on $X_\tau$ with respect to the multiprojective embedding, such that one obtains the SMT of Lakshmibai, Musili and Seshadri, or more general, the SMT indexed by LS-tableaux?
\vskip 10pt

In this \whatisthis{} we show that such a stratification exists under certain combinatorial conditions. We now give an overview over the different sections.

First, we briefly recapitulate the definition of a (multiprojective) Seshadri stratification and its quasi-valuation in Section~\ref{sec:multiproj_strat} along with the SMT associated to a normal stratification. All subsequent sections are devoted to answering the above question. Note that the author already covered the case of partial flag varieties in Dynkin type \texttt{A} (cf.~\cite{ownarticle}), where there exists a normal and balanced Seshadri stratification, such that the resulting SMT coincides with the usual SMT of monomials in Pl\"ucker coordinates indexed by semistandard Young-tableaux (\cite[Chapter 2]{seshadri2016introduction}). Of course, the stratifications we seek to construct here should both generalize the ordinary stratification on Schubert varieties from~\cite{seshstratandschubvar} and the multiprojective stratification in type $\texttt{A}$ from~\cite{ownarticle}.

We begin by introducing the tableau model of LS-tableaux of type $(\underline\lambda, \mathcal I)$, which we hope to find in the associated fan of monoids to the generalized stratification. These tableaux depend on two choices: First, a sequence $\underline\lambda = (\lambda_1, \dots, \lambda_m)$ of dominant weights which fixes the multiprojective embedding 
\begin{align*}
    X_\tau \hookrightarrow \prod_{i=1}^m \PP(V(\lambda_i)_\tau)
\end{align*}
and second, a subposet $\mathcal I$ of the power set poset $\mathcal P(\set{1, \dots, m}) \setminus \set{\varnothing}$, that is ordered by inclusion and plays the role of the index poset to the stratification. One can think of the elements in $\mathcal I$ as the possible shapes of the columns in the LS-tableaux: To every element $I \in \mathcal I$ we associate a dominant weight $\lambda_I$. Each column $\pi_1, \dots, \pi_s$ of our LS-tableaux is an LS-path to a dominant weight $\lambda_{I_k}$ and these weights need to follow a weakly decreasing sequence $I_1 \supseteq \dots \supseteq I_s$ in $\mathcal I$. 

To obtain the classical Young-tableaux for the group $\mathrm{SL}_n(\K)$ one could choose the poset $\mathcal I$ of all sets $\set{1, \dots, i}$ for $i = 1, \dots, n-1$. The shapes of the columns in a Young-tableau correspond to their length, so the set $\set{1, \dots, i}$ represents a column of length $i$. Semi\-standard Young-tableaux are generalized in the following way: An LS-tableau is called \textit{$\tau$-standard}, if one can lift the Weyl group cosets of its columns to a \textit{defining chain}, \ie a weakly decreasing sequence in the Weyl group. In the Appendix~\ref{sec:tableaux} we explain in more detail how Young-tableaux and also the tableaux of De Concini (\cite{deconcini}) can be seen as special cases of LS-tableaux.

To the fixed choice of $\underline\lambda$ and $\mathcal I$ we also associate a graded poset $\DCP{}$ (cf. Section~\ref{sec:dcp}), called \textit{defining chain poset}, which serves as the underlying poset for the desired Seshadri stratification. This poset is constructed from the idea that every defining chain of a $\tau$-standard LS-tableau should be contained in a chain of $\DCP{}$. However, only certain index posets $\mathcal I$ induce a well-defined stratification. First, every pair of non-comparable elements needs to satisfy the condition (\ref{eq:need_for_s2}), which assures the existence of specific covering relations. Second, the poset is required to be \nice{\tau}. These are exactly the index posets, where the $\tau$-standardness of an associated LS-tableau can be verified locally, by comparing consecutive columns (which is known as \textit{weak standardness}). 

\vskip 10pt
\noindent
\textbf{Theorem} (Theorem~\ref{thm:stratification}). If the poset $\mathcal I$ is \nice{\tau} and it satisfies the condition~(\ref{eq:need_for_s2}), then there exists a multiprojective Seshadri stratification on $X_\tau$ with underlying poset $\DCP{}$ and index poset $\mathcal I$.
\vskip 10pt

Fortunately, there always exists at least one \nice{\tau} index poset $\mathcal I$, namely the full power set $\mathcal P(\set{1, \dots, m}) \setminus \set{\varnothing}$, but this is a rather unwieldy choice for computations. The author was not able to find a combinatorial characterization of \nice{\tau} posets in full generality. When $\tau$ is the unique maximal element in $W/W_Q$, then $\tau$-standardness is characterized by the existence of certain paths in the Dynkin diagram of $G$ (see Theorem~\ref{thm:nice_I}). If the Dynkin diagram is a line (\ie in the types \texttt{A}, \texttt{B}, \texttt{C}, \texttt{F} and \texttt{G}), one can always choose a totally ordered poset $\mathcal I$ and the associated model of LS-tableaux is similar to classical Young-tableaux. We give an example for $\tau$-standard posets for the partial flag varieties in all Dynkin types (Section~\ref{subsec:poset_examples}).

In Sections~\ref{sec:LS-fan} to \ref{sec:smt}, we compute the fan of monoids $\Gamma$ for the previously constructed stratifications. The elements in $\Gamma$ correspond exactly to the $\tau$-standard LS-tableaux of type $(\underline\lambda, \mathcal I)$. To each of these tableaux we associate a monomial in the path vectors defined by Littelmann. We prove that the stratification is normal and balanced and that these monomials in the path vectors form the resulting SMT. As expected, the standard monomial basis coincides with the basis constructed in \cite{littelmann1998contracting}.

This \whatisthis{} is based on the last two chapters of the author's Ph.\,D. thesis \cite{thesis}. 

\section{Multiprojective Seshadri stratifications}
\label{sec:multiproj_strat}

We begin by reviewing the notion of a multiprojective Seshadri stratification introduced in~\cite{ownarticle}, where the proofs of all statements in the following paragraph can be found.

\subsection{Definitions and properties}

Throughout this section we fix an algebraically closed field $\K$ and a \textit{multiprojective} variety $X$, \ie a (Zariski-)closed subset $X$ in a product $\PP(V_1) \times \dots \times \PP(V_m)$ of projective spaces, where $V_1, \dots, V_m$ are finite-dimensional vector spaces over $\K$.

The multicone $\hat X$ of $X$ is a closed subvariety of the affine space $V = V_1 \times \dots \times V_m$. Let $R = \K[X]$ denote the multihomogeneous coordinate ring of $X$, \ie the coordinate ring $\K[\hat X]$ of the multicone. We write $[k]$ for the set of all integers between $1$ and $k \in \N$. Each subset $I \subseteq [m]$ comes with the two natural projections
\begin{align}
    \label{eq:def_projection_maps}
    \pi_I: \prod_{i \in [m]} \PP(V_i) \twoheadrightarrow \prod_{i \in I} \PP(V_i) \quad \text{and} \quad \hat \pi_I: \prod_{i \in [m]} V_i \twoheadrightarrow \prod_{i \in I} V_i
\end{align}
as well as the multiprojective variety $X_I = \pi_I(X)$. Note that the multicone $\hat X_I$ of $X_I$ coincides with the image of $\hat X$ under the map $\hat \pi_I$. The surjection $\hat X \twoheadrightarrow \hat X_I$ induces an embedding of the multihomogeneous coordinate ring $\K[X_I]$ onto a graded subalgebra of $R$, namely the direct sum of all homogeneous components $R_{\underline d} \subseteq R$ for tuples $\underline d = (d_1, \dots, d_m) \in \N_0^m$ where $d_j = 0$ for all $j \notin I$.

Analogous to the definition of a Seshadri stratification in~\cite{seshstrat}, we fix a finite set $A$, a collection $\set{X_p \mid p \in A}$ of irreducible projective varieties, which are smooth in codimension one, and a collection of functions $\set{f_p \in R \mid p \in A}$ called \textbf{extremal functions}. The main difference to the original definition is that $X_p$ no longer needs to be a subvariety or even a subset of $X$. Instead we fix a third collection $\set{I_p \subseteq [m] \mid p \in A}$ of non-empty subsets of $[m]$ and require that $X_p$ is a closed subvariety of $X_{I_p} = \pi_{I_p}(X)$. If we view the affine space $\prod_{i \in I_p} V_i$ as a closed subvariety of $V$ via the linear embedding $\prod_{i \in I_p} V_i \hookrightarrow V$, then $\hat X_p$ can be seen as a closed subvariety of $\hat X$. This allows us to equip the set $A$ with the partial order $\leq$, such that $q \leq p$ if and only if $\hat X_q \subseteq \hat X_p$. The function $f_p$ needs to be non-constant, multihomogeneous and included in the subring $\K[X_{I_p}] \subseteq R$.

\begin{definition}[Multiprojective Seshadri stratification]
    \label{def:multiproj_seshadri_strat}
    These three collections of varieties, extremal functions and index sets are called a \textbf{(multiprojective) Seshadri stratification}, if there exists an element $p_{\text{max}} \in A$ with $I_{p_{\text{max}}} = [m]$ and $X_{p_{\text{max}}} = X$ and the following three conditions are fulfilled:
    \begin{enumerate}[label=(S\arabic{enumi})]
        \item \label{itm:seshadri_strat_a} If $q < p$ is a covering relation, then $\hat X_q \subseteq \hat X_p$ is a codimension one subvariety (where both are seen as subvarieties of $V$);
        \item \label{itm:seshadri_strat_b} The function $f_q$ vanishes on $\hat X_p$, if $q \nleq p$;
        \item \label{itm:seshadri_strat_c} For each $p \in A$ holds the set-theoretic equality
        \begin{equation}
        \label{eq:s3}
        \set{x \in \hat X \mid f_p(x) = 0} \cap \hat X_p = \set{0} \cup \bigcup_{p \;\text{covers}\; q} \hat X_q.
        \end{equation}
    \end{enumerate}
\end{definition}

Notice, that for $m = 1$ the notion of a multiprojective Seshadri stratification coincides with notion of a Seshadri stratification introduced in \cite{seshstrat}. In this case all strata $X_p$ for $p \in A$ are closed subvarieties of $X$, since $I_p = \set{1}$. 

The affine multicone $\hat X \subseteq V$ of $X$ is the affine cone of a projective variety $\widetilde X \subseteq \PP(V)$. Hence every multiprojective Seshadri stratification on $X \subseteq \prod_{i=1}^m \PP(V_i)$ can also be seen as a Seshadri stratification on $\widetilde X$. Therefore one can informally say that every result in \loccit{}, where the grading on $R$ is not involved, does also hold in the multiprojective case. As a first example: The poset $A$ is a graded poset of length $\dim \widetilde X = \dim \hat X - 1$, that is to say all maximal chains have length $\dim \widetilde X$. The rank of an element $p \in A$ is given by $r(p) = \dim \hat X_p - 1$.

\begin{proposition}
    Every multiprojective variety $X \subseteq \PP(V_1) \times \dots \times \PP(V_m)$ admits a Seshadri stratification.
\end{proposition}

In contrast to the Seshadri stratifications introduced in~\cite{seshstrat}, their multiprojective generali\-zations have an additional underlying structure, namely the poset
\begin{align}
    \label{eq:def_mathcal_I}
    \mathcal I = \{ I_p \subseteq [m] \mid p \in A \},
\end{align}
which is ordered by inclusion. We call it the \textbf{index poset}.

\begin{lemma}
    \label{lem:properties_A_mathcal_I}
    The map $A \to \mathcal I$, $p \mapsto I_p$ is monotone and has the following properties:
    \begin{enumerate}[label=(\alph{enumi})]
        \item \label{itm:properties_A_mathcal_I_a} Let $q < p$ be a covering relation in $A$. Then $I_p \setminus I_q$ contains at most one element. In the case $I_q \neq I_p$ it holds $\pi_{I_q}(X_p) = X_q$.
        \item \label{itm:properties_A_mathcal_I_b} If $p \in A$ is a minimal element, then $I_p$ is a one-element set.
    \end{enumerate}
    In particular, $\mathcal I$ is a graded poset of length $m-1$.
\end{lemma}

For multiprojective stratifications we have the following new kind of covering relations in $A$ which do not appear for $m = 1$.

\begin{lemma}
    \label{lem:projective_covering_relation}
    Let $q < p$ be a covering relation in $A$ with $I_p \setminus I_q = \set{i}$. The vanishing multiplicity of a multihomogeneous function $g \in \K[\hat X_p] \setminus \set{0}$ along the prime divisor $\hat X_q \subseteq \hat X_p$ is equal to the $i$-th component of $\deg g \in \N_0^m$. In particular, the $i$-th component of $\deg f_p$ is non-zero.
\end{lemma}

\subsection{The quasi-valuation and the fan of monoids}

We summarize some constructions and results from \cite{seshstrat} about Seshadri stratifications. It is strongly recommended to read the original papers, as we cannot do justice to their results on just a few pages and this section mainly serves as a reminder for all the notation introduced for Seshadri stratifications.

We fix the following notation: \label{txt:def_support}If $K$ is any field of characteristic zero and $S$ is a finite set, then we write $K^S$ for the vector space over $K$ with basis $\set{e_s \mid s \in S}$ indexed by $S$. Let $\N_0^S$ be the monoid generated by these basis elements and $\Z^S \subseteq K^S$ be the smallest group containing $\N_0^S$. For each element $x = \sum_{s \in S} x_s e_s \in K^S$ with coefficients $x_s \in K$ the set
\begin{align*}
    \supp x = \set{s \in S \mid x_s \neq 0}
\end{align*}
is called the \textit{support} of $x$.

By definition, the multicone $\hat X_q$ is a prime divisor of $\hat X_p$ for every covering relation $q < p$ in $A$. If one extends the poset $A$ by a unique minimal element $p_{-1}$ with associated index set $I_{p_{-1}} = \varnothing$, then the multicone $\hat X_{p_{-1}} = \set{0}$ is a prime divisor of $\hat X_p \cong \A^1$ for each minimal element $p \in A$. To each covering relation $p > q$ in the extended poset $\hat A = A \cup \set{p_{-1}}$ we have an associated valuation, namely the discrete valuation
\begin{align*}
    \nu_{p, q}: \K(\hat X_p) \setminus \set{0} \to \Z,
\end{align*}
sending a non-zero, rational function $g$ to its vanishing multiplicity at the prime divisor $\hat X_q \subseteq \hat X_p$. Its value 
\begin{align*}
    b_{p,q} = \nu_{p,q}(f_p\big|_{\hat X_p}) \in \N
\end{align*}
at the extremal function $f_p$ is called the \textit{bond} of the covering relation $q < p$. If $p$ is minimal in $A$, then $b_{p, p_{-1}}$ coincides with the total degree $\vert \mkern-1mu \deg f_p \mkern1mu \vert$, which is the sum of all entries in the degree $\deg f_p \in \N_0^m$. 

Every Seshadri stratification gives rise to a collection of valuations on $R$, one for each maximal chain $\mathfrak C$ in $A$. Let $p_r > \dots > p_0$ be the elements of $\mathfrak C$. To a regular function $g \in R \setminus \set{0}$ one associates a sequence $g_{\mathfrak C} = (g_r, \dots, g_0)$ of rational functions inductively via $g_r \coloneqq g$ and 
\begin{align*}
    g_{i-1} = \frac{ g_i^{\mkern-1mu b_{p_i, p_{i-1}}} }{ f_{p_i}^{\mkern2mu \nu_{p_i, p_{i-1}}(g_i)} } \bigg|_{\hat X_{p_{i-1}}} \in \K(\hat X_{p_{i-1}}).
\end{align*}
for $i = r, \dots, 1$. Further one defines the element
\begin{align*}
    \mathcal V_{\mathfrak C}(g) = \sum_{j=0}^r \frac{\nu_{p_j, p_{j-1}}(g_j)}{\prod_{k=j}^r b_{p_k, p_{k-1}}} \, e_{p_j} \in \Q^{\mathfrak C}.
\end{align*}
By this definition, each extremal function $f_p$ for $p \in \mathfrak C$ is mapped to the vector $\mathcal V_{\mathfrak C}(f_p) = e_p$. We equip the abelian group $\Q^{\mathfrak C}$ with the lexicographic order induced by the total order on the maximal chain $\mathfrak C$, \ie for all elements $\underline a = \sum_{i=0}^r a_i e_{p_i}, \underline b = \sum_{i=0}^r b_i e_{p_i}$ in $\Q^{\mathfrak C}$ it holds
\begin{align*}
    \underline a \geq \underline b \quad \Longleftrightarrow \quad \text{$\underline a = \underline b$ or $a_i > b_i$ for the maximal index $i \in \set{0, \dots, r}$ with $a_i \neq b_i$}.
\end{align*}
Then the map $\mathcal V_{\mathfrak C}: R \setminus \set{0} \to \Q^{\mathfrak C}$ is a valuation. Chiriv{\`i}, Fang and Littelmann also gave another, equivalent definition in \cite{seshstrat}, which we do not use here, as it is less suited for computations.

\begin{remark}
    \label{rem:induced_strat}
    Every element $p \in A$ induces a Seshadri stratification on the multiprojective variety $X_p \subseteq \prod_{i \in I_p} \PP(V_i)$ via the poset $A_p = \set{q \in A \mid q \leq p}$, where we take the same strata, extremal functions and index sets as in the stratification on $X$. By its definition, the valuation $\mathcal V_{\mathfrak C}$ is compatible in the following sense with the valuation $\mathcal V_{\mathfrak C_p}$ of the induced stratification along the maximal chain $\mathfrak C_p = \mathfrak C \cap A_p$: For every $g \in R \setminus \set{0}$, that does not vanish identically on $\hat X_p$, the valuation $\mathcal V_{\mathfrak C_p}(g\big|_{\hat X_p}) \in \Q^{\mathfrak C_p}$ coincides with $\mathcal V_{\mathfrak C}(g)$, when extended by zeros to an element of $\Q^{\mathfrak C}$.
\end{remark}

The collection of all valuations $\mathcal V_{\mathfrak C}$ define a quasi-valuation $\mathcal V$, which respects the structure of the whole poset $A$, not just of one maximal chain. A \textbf{quasi-valuation} is defined similar to valuation, only the condition $\mathcal V(gh) = \mathcal V(g) + \mathcal V(h)$ for all $g, h \in R$ with $gh \neq 0$ is replaced by the inequality $\mathcal V(gh) \geq \mathcal V(g) + \mathcal V(h)$. To obtain this quasi-valuation one needs to extend $\mathcal V_{\mathfrak C}$ to a valuation $R \setminus \set{0} \to \Q^{\mathfrak C} \hookrightarrow \Q^A$, such that all valuations take values in the same abelian group. In order to make sense of this, we need a total order on $\Q^A$ such that each linear inclusion $\Q^{\mathfrak C} \hookrightarrow \Q^A$ is monotone. In general, there is no natural candidate for this total order. For this reason, one needs to choose and fix a total order $\geq^t$ on $A$ linearizing the partial order, \ie for each elements $p, q \in A$ the relation $p \geq q$ implies $p \geq^t q$. This total order induces the lexicographic order on $\Q^A$ and each map $\mathcal V_{\mathfrak C}: R \setminus \set{0} \to \Q^A$ is a valuation. One obtains the quasi-valuation $\mathcal V$ by taking their minimum with respect to this total order on $\Q^A$:
\begin{align*}
    \mathcal V: R \setminus \set{0} \to \Q^A, \quad g \mapsto \operatorname{min} \set{\mathcal V_{\mathfrak C}(g) \mid \text{$\mathfrak C$ maximal chain in $A$}}.
\end{align*}
Hence the quasi-valuation depends on the choice of this total order $\geq^t$ on $A$.

There is also the following inductive way of describing the quasi-valuation $\mathcal V$. Let $p$ be any element in $A$, $g \in \K(\hat X_p)$ be a non-zero rational function. We write $\mathcal V_p$ for the quasi-valuation on the induced Seshadri stratification on $X_p$ with underlying poset $A_p = \set{q \in A \mid q \leq p}$. Then it holds
\begin{align}
    \label{eq:inductive_def_quasi_val}
    \mathcal V_p(g) = \frac{\nu_{p, q}(g)}{b_{p,q}} e_p + \frac{1}{b_{p,q}} \mathcal V_q \big( \frac{g^{b_{p,q}}}{f_p^{\nu_{p, q}(g)}} \big\vert_{\hat X_q} \big),
\end{align}
where $q$ is the unique minimal element covered by $p$ with respect to the total order $\geq^t$, such that it holds
\begin{align*}
    \frac{\nu_{p, q}(g)}{b_{p,q}} = \operatorname{min} \Set{ \frac{\nu_{p, q'}(g)}{b_{p,q'}} \ \middle\vert \ \text{$q' \in A$ covered by $p$}}.
\end{align*}

The quasi-valuation $\mathcal V$ has the following important properties, which we use many times throughout this \whatisthis{} without mention (see \cite[Section 8]{seshstrat}). 
\begin{itemize}
    \item The values of $\mathcal V$ have non-negative entries, \ie the quasi-valuation $\mathcal V(g)$ of every function $g \in R \setminus \set{0}$ is contained in the non-negative orthant $\Q_{\geq 0}^A$.
    \item One can characterize combinatorially for which maximal chains $\mathfrak C$ the quasi-valuation attains its minimum. For each $g \in R \setminus \set{0}$ it holds $\mathcal V_{\mathfrak C}(g) = \mathcal V(g)$, if and only if the support $\supp \mathcal V(g) \subseteq A$ lies in $\mathfrak C$. As a consequence: If $g, h \in R$ are non-zero and there exists a maximal chain $\mathfrak C$ containing both $\supp \mathcal{V}(g)$ and $\supp \mathcal{V}(h)$, then the quasi-valuation is additive, \ie we have $\mathcal{V}(gh) = \mathcal{V}(g) + \mathcal{V}(h)$. 
    \item Every extremal function $f_p$ for $p \in A$ has the quasi-valuation $\mathcal V(f_p) = e_p$, so the support is given by $\supp \mathcal V(f_p) = \set{p}$. In particular: If $p_1, \dots, p_s \in A$ are contained in a chain in $A$ and $n_1, \dots, n_s \in \N_0$, then it follows
    \begin{align*}
        \mathcal V(f_{p_1}^{n_1} \cdots f_{p_s}^{n_s}) = \sum_{i=1}^s n_i e_{p_i}.
    \end{align*}
\end{itemize}

The image of the quasi-valuation is denoted by $\Gamma = \set{\mathcal V(g) \in \Q^A \mid g \in R \setminus \set{0}}$. For each (not necessarily maximal) chain $C$ in $A$ the subset
\begin{align*}
    \Gamma_{C} = \set{\underline a \in \Gamma \mid \supp \underline a \subseteq C}
\end{align*}
is a finitely generated monoid. The set $\Gamma$ is called the \textit{fan of monoids} of the Seshadri stratification, since it is the union of all monoids $\Gamma_{C}$ and the cones in $\R^A$ generated by these monoids form a fan.

The quasi-valuation $\mathcal V: R \setminus \set{0} \to \Q^m$ induces a filtration on $R$ by the subrings
\begin{align*}
    R_{\geq \underline a} = \set{g \in R \setminus \set{0} \mid \mathcal V(g) \geq \underline a} \cup \set{0}
\end{align*}
for $\underline a \in \Gamma$. Since $\mathcal V(g)$ only has non-negative entries for all $g \in R \setminus \set{0}$, these subrings are ideals in $R$. The quotient of $R_{\geq \underline a}$ by the ideal $R_{> \underline a} = \set{g \in R \setminus \set{0} \mid \mathcal V(g) > \underline a} \cup \set{0}$ is one-dimensional for every $\underline a \in \Gamma$. They are called the \textit{leaves} of the quasi-valuation $\mathcal V$. All leaves are at most one-dimensional. Hence choosing a regular function $g_{\underline a} \in R$ with $\mathcal V(g_{\underline a}) = \underline a$ for each $\underline a \in \Gamma$ yields a basis of $R$ as a vector space.

The concepts of normal and balanced Seshadri stratifications were introduced in \cite[Sections 13, 15]{seshstrat}. They can also be used in the multiprojective case. 

\begin{definition}
    A multiprojective Seshadri stratification is called 
    \begin{enumerate}[label=(\alph{enumi})]
        \item \textbf{normal}, if $\Gamma_{\mathfrak C}$ is saturated for every maximal chain $\mathfrak C$, \ie it is equal to the intersection of the lattice $\mathcal L^{\mathfrak C}$ generated by $\Gamma_{\mathfrak C}$ with the positive orthant $\Q^{\mathfrak C}_{\geq 0}$;
        \item \textbf{balanced}, when the fan of monoids $\Gamma$ is independent of the choice of the total order $\geq^t$.
    \end{enumerate}
\end{definition}

Every normal Seshadri stratification defines a standard monomial theory on $R$ in the sense of the next proposition. When the stratification is balanced as well, then the normality and its associated standard monomial theory do not depend on the choice if the total order $\geq^t$.

An element $\underline a \in \Gamma$ is called \textit{decomposable}, if it is $0$ or it can be written in the form $\underline a = \underline a^1 + \underline a^2$ for two elements $\underline a^1, \underline a^2 \in \Gamma \setminus \set{0}$ with $\min \supp \underline a^1 \geq \max \supp \underline a^2$. Otherwise $\underline a$ is  called \textit{indecomposable}. Note that the minima and maxima exist, since the support of each element in $\Gamma$ is totally ordered. Let $\mathbb G$ be the set of all indecomposable elements in $\Gamma$. For each $\underline a \in \mathbb G$ we fix a regular function $x_{\underline a} \in R \setminus \set{0}$ with $\mathcal V(x_{\underline a}) = \underline a$ and let $\mathbb G_R = \set{x_{\underline a} \mid \underline a \in \mathbb G}$ be the set of these functions.

We assume that the stratification is normal. In this case every element $\underline a \in \Gamma$ has a unique decomposition into a sum $\underline a = \underline a^1 + \dots + \underline a^s$ of indecomposable elements $\underline a^k \in \Gamma$, such that $\min \supp \underline a^k \geq \max \supp \underline a^{k+1}$ holds for all $k = 1, \dots, s-1$. With the choice of the set $\mathbb G_R$ one can therefore associate a regular function to every element $\underline a \in \Gamma$ via
\begin{align*}
    x_{\underline a} \coloneqq x_{\underline a^1} \cdots x_{\underline a^s} \in R.
\end{align*}
A monomial in the functions in $\mathbb G_R$ is called \textit{standard}, if it is of the form $x_{\underline a}$ for some element $\underline a \in \Gamma$.

\begin{proposition}[{\cite[Proposition 15.6]{seshstrat}}]
    \label{prop:smt}
    If the stratification is normal and $\mathbb G_R$ and $x_{\underline a}$ are chosen as above, then the following statements hold:
    \begin{enumerate}[label=(\alph{enumi})]
        \item The set $\mathbb G_R$ generates $R$ as a $\K$-algebra.
        \item The set of all standard monomials in $\mathbb G_R$ is a basis of $R$ as a vector space.
        \item If $\underline a = \underline a^1 + \dots + \underline a^s$ is the unique decomposition of $\underline a$ into indecomposables, then $x_{\underline a} \coloneqq x_{\underline a^1} \cdots x_{\underline a^s}$ is a standard monomial with $\mathcal V(x_{\underline a}) = \underline a$.
        \item For each non-standard monomial $x_{\underline a^1} \cdots x_{\underline a^s}$ in $\mathbb G_R$ there exists a straightening relation
        \begin{align*}
            x_{\underline a^1} \cdots x_{\underline a^s} = \sum_{\underline b \in \Gamma} u_{\underline b} \mkern2mu x_{\underline b}
        \end{align*}
        expressing it as a linear combination of standard monomials, where $u_{\underline b} \neq 0$ only if $\underline b \geq^t \underline a^1 + \dots + \underline a^s$.
    \end{enumerate}
\end{proposition}

The above standard monomial theory is compatible with the induced stratification (see Remark~\ref{rem:induced_strat}) on $X_p$ for $p \in A$. This was shown in \cite[Theorem 15.12]{seshstrat}. By adapting this result to the multiprojective setting, we obtain the following corollary.

\begin{corollary}
    \label{cor:smt_induced_strat}
    If the Seshadri stratification on $X$ is normal and balanced, then the following statements are fulfilled for each $p \in A$:
    \begin{enumerate}[label=(\alph{enumi})]
        \item The induced stratification on $X_p$ is also normal and balanced.
        \item The fan of monoids of this stratification is equal to $\Gamma_p = \set{\underline a \in \Gamma \mid \max \supp \underline a \leq p}$.
        \item The set of indecomposable elements in $\Gamma_p$ is given by $\mathbb G_p = \mathbb G \cap \Gamma_p$ and the restriction of a function $x_{\underline a}$ with $\underline a \in \mathbb G_p$ fulfills $\mathcal V_p(x_{\underline a}\big\vert_{\hat X_p}) = \underline a$, where $\mathcal V_p$ denotes the quasi-valuation of the induced stratification.
        \item A standard monomial $x_{\underline a}$, $a \in \Gamma$, vanishes identically on $\hat X_p$, if and only if $\max \supp \underline a \leq p$. In this case, the monomial is called \textbf{\boldmath{}standard on $X_p$}.
        \item The restrictions of the standard monomials $x_{\underline a}$, $a \in \Gamma$, which are standard on $X_p$, form a basis of the multihomogeneous coordinate ring $\K[X_p]$ \wrt{} the embedding
        \begin{align*}
            X_p \longhookrightarrow \prod_{i \in I_p} \PP(V_i).
        \end{align*}
    \end{enumerate}
\end{corollary}

\ytableausetup{smalltableaux}
\begin{figure}
\centering
\begin{subfigure}{.5\textwidth}
    \begin{center}
        \begin{tikzpicture}[scale=0.65]
        \node (p3) at (0,0) {$\begin{ytableau} 3 \end{ytableau}$};
        \node (p2) at (0,-1.5)  {$\begin{ytableau} 2 \end{ytableau}$};
        \node (p1) at (-1.5,-3) {$\begin{ytableau} 1 \end{ytableau}$};
        \node (p23) at (1.5,-3) {$\begin{ytableau} 2 \\ 3 \end{ytableau}$};
        \node (p13) at (0,-4.5) {$\begin{ytableau} 1 \\ 3 \end{ytableau}$};
        \node (p12) at (0,-6.5) {$\begin{ytableau} 1 \\ 2 \end{ytableau}$};
        
        \draw [ thick, shorten <=-2pt, shorten >=-2pt, -stealth] (p3) -- (p2);
        \draw [thick, shorten <=-2pt, shorten >=-2pt, -stealth] (p2) -- (p1);
        \draw [thick, shorten <=-2pt, shorten >=-2pt, -stealth] (p2) -- (p23);
        \draw [thick, shorten <=-2pt, shorten >=-2pt, -stealth] (p1) -- (p13);
        \draw [thick, shorten <=-2pt, shorten >=-2pt, -stealth] (p23) -- (p13);
        \draw [thick, shorten <=-2pt, shorten >=-2pt, -stealth] (p13) -- (p12);
        \end{tikzpicture}
    \end{center}
    \caption{Type $\texttt{A}_2$}
    \label{fig:hasse_a_2}
\end{subfigure}%
\begin{subfigure}{.5\textwidth}
    \begin{center}
        \begin{tikzpicture}[scale=0.70]
        \node (p4) at (0,0) {$\begin{ytableau} 4 \end{ytableau}$};
        \node (p3) at (1.5,0) {$\begin{ytableau} 3 \end{ytableau}$};
        \node (p2) at (3,0)  {$\begin{ytableau} 2 \end{ytableau}$};
        \node (p1) at (4.5,0) {$\begin{ytableau} 1 \end{ytableau}$};
        \node (p34) at (1.5,-1.75) {$\begin{ytableau} 3 \\ 4 \end{ytableau}$};
        \node (p24) at (3,-1.75) {$\begin{ytableau} 2 \\ 4 \end{ytableau}$};
        \node (p14) at (4.5,-1.75) {$\begin{ytableau} 1 \\ 4 \end{ytableau}$};
        \node (p23) at (3,-3.8) {$\begin{ytableau} 2 \\ 3 \end{ytableau}$};
        \node (p13) at (4.5,-3.8) {$\begin{ytableau} 1 \\ 3 \end{ytableau}$};
        \node (p12) at (6,-3.8) {$\begin{ytableau} 1 \\ 2 \end{ytableau}$};
        \node (p234) at (3,-6.1) {$\begin{ytableau} 2 \\ 3 \\ 4 \end{ytableau}$};
        \node (p134) at (4.5,-6.1) {$\begin{ytableau} 1 \\ 3 \\ 4 \end{ytableau}$};
        \node (p124) at (6,-6.1) {$\begin{ytableau} 1 \\ 2 \\ 4 \end{ytableau}$};
        \node (p123) at (7.5,-6.1) {$\begin{ytableau} 1 \\ 2 \\ 3 \end{ytableau}$};
        
        \draw [thick, shorten <=-2pt, shorten >=-2pt, -stealth] (p4) -- (p3);
        \draw [thick, shorten <=-2pt, shorten >=-2pt, -stealth] (p3) -- (p2);
        \draw [thick, shorten <=-2pt, shorten >=-2pt, -stealth] (p2) -- (p1);
        \draw [thick, shorten <=-2pt, shorten >=-2pt, -stealth] (p3) -- (p34);
        \draw [thick, shorten <=-2pt, shorten >=-2pt, -stealth] (p2) -- (p24);
        \draw [thick, shorten <=-2pt, shorten >=-2pt, -stealth] (p1) -- (p14);
        \draw [thick, shorten <=-2pt, shorten >=-2pt, -stealth] (p34) -- (p24);
        \draw [thick, shorten <=-2pt, shorten >=-2pt, -stealth] (p24) -- (p23);
        \draw [thick, shorten <=-2pt, shorten >=-2pt, -stealth] (p24) -- (p14);
        \draw [thick, shorten <=-2pt, shorten >=-2pt, -stealth] (p23) -- (p13);
        \draw [thick, shorten <=-2pt, shorten >=-2pt, -stealth] (p14) -- (p13);
        \draw [thick, shorten <=-2pt, shorten >=-2pt, -stealth] (p13) -- (p12);
        \draw [thick, shorten <=-2pt, shorten >=-2pt, -stealth] (p23) -- (p234);
        \draw [thick, shorten <=-2pt, shorten >=-2pt, -stealth] (p13) -- (p134);
        \draw [thick, shorten <=-2pt, shorten >=-2pt, -stealth] (p12) -- (p124);
        \draw [thick, shorten <=-2pt, shorten >=-2pt, -stealth] (p234) -- (p134);
        \draw [thick, shorten <=-2pt, shorten >=-2pt, -stealth] (p134) -- (p124);
        \draw [thick, shorten <=-2pt, shorten >=-2pt, -stealth] (p124) -- (p123);
        \end{tikzpicture}
    \end{center}
    \caption{Type $\texttt{A}_3$}
    \label{fig:hasse_a_3}
\end{subfigure}
\caption{Hasse-diagrams of $\ulW$ for $Q = B$}
\label{fig:hasse_ulWlambda_type_A}
\end{figure}
\ytableausetup{nosmalltableaux}

\begin{example}
    \label{ex:strat_type_A}
    In~\cite{ownarticle} the author constructed a multiprojective stratification on flag varieties in Dynkin type $\texttt{A}$, which forms the first example of the stratifications, we seek to construct in this \whatisthis{}. Let $G$ denote the simple algebraic group $\mathrm{SL}_n(\K)$, where $\K$ is algebraically closed of characteristic zero. Each time we specifically work in type $\texttt{A}$ we use the notation from the Appendix~\ref{subsec:weyl_type_A}, e.\,g. the maximal parabolic subgroups $P_i$, the fundamental weights $\omega_i$ or the one-line notation for Weyl group elements.
    
    Each parabolic subgroup $Q$ containing the Borel subgroup can be uniquely written as the intersection $Q = P_{k_1} \cap \dots \cap P_{k_m}$ of maximal parabolics with $1 \leq k_1 < \dots < k_m \leq n-1$. The corresponding partial flag variety $G/Q$ admits the multiprojective embedding
    \begin{align*}
        G/Q \longhookrightarrow \prod_{i=1}^m G/P_{k_i} \longhookrightarrow \prod_{i=1}^m \PP(V(\omega_{k_i})).
    \end{align*}
    As the underlying poset of the Seshadri stratification we use the disjoint union $\underline W = \coprod_{i=1}^m W/W_i \times \set{i}$, where $W_i \subseteq W$ denotes the stabilizer of $\omega_{k_i}$. Each element in $\underline W$ can be seen as a Young-tableau, that only consists of one column. The partial order can be defined via row-wise comparison of these tableaux (see~\cite{ownarticle}). The two posets for the full flag variety $G/B$ in the cases $n = 3$ and $n = 4$ are shown in Figure~\ref{fig:hasse_ulWlambda_type_A}.
    
    Let $I_{(\theta, i)} = \set{i, \dots, m} \subseteq [m]$ be the index set of an element $(\theta, i) \in \underline W$. The projection of $G/Q$ via $\pi_{\set{i, \dots, m}}$ can be interpreted as the flag variety $G/Q_i$ to the parabolic subgroup $Q_i = P_{k_i} \cap \dots \cap P_{k_m}$, since the following diagram commutes.
    \begin{equation*}
        \begin{tikzcd}
        G/Q \arrow[d, two heads] \arrow[r, hook] & \prod_{j=1}^m G/P_{k_j} \arrow[r, hook] \arrow[d, two heads] & \prod_{j=1}^m \PP(V(\omega_{k_j})) \arrow[d, two heads] \\
        G/Q_i \arrow[r, hook] & \prod_{j=i}^m G/P_{k_j} \arrow[r, hook] & \prod_{j=i}^m \PP(V(\omega_{k_j}))
        \end{tikzcd}
    \end{equation*}
    To each element $(\theta, i) \in \underline W$ we associate:
    \begin{itemize}
        \item The Schubert variety $X_{\max_{Q_i}(\theta)} \subseteq G/Q_i$ as the stratum $X_{(\theta, i)}$, where $\max_{Q_i}(\theta)$ is the unique maximal lift of $\theta$ to $W/W_{Q_i}$;
        \item the extremal function $f_{(\theta, i)} = p_{(\theta, i)}$, which is the Pl\"ucker coordinate to $\theta$ (up to a non-zero scalar), \ie the projection $V(\omega_{k_i}) \to V(\omega_{k_i})_{\theta} \cong \K$ onto the weight space of weight $\theta(\omega_{k_i})$.
    \end{itemize}
    This data defines a normal and balanced Seshadri stratification and the set of all indecomposable elements in its fan of monoids is given by
    \begin{align*}
        \mathbb G = \set{e_{(\theta, i)} \mid (\theta, i) \in \underline W}.
    \end{align*}
    Hence one can choose $\mathbb G_R$ to be the set of all extremal functions $f_{(\theta, i)}$. 
    The elements in the fan of monoids correspond to semistandard Young-tableaux, where only columns of the lengths $k_1, \dots, k_m$ may appear. More precisely, the coefficient of the vector $e_{(\theta, i)}$ in an element $\underline a \in \Gamma$ is equal to the number of times the column corresponding to $(\theta, i)$ appears in the tableau associated to $\underline a$. The unique decomposition of $\underline a \in \Gamma$ is therefore given by the columns of its Young-tableau and we obtain the usual standard monomial theory of Hodge-Young consisting of monomials in Pl\"ucker coordinates indexed by semistandard Young-tableaux (cf. \cite[Chapter 2]{seshadri2016introduction}).
\end{example}

\section{Background and notation}
\label{sec:choices}

From now on we fix a connected, simply-connected, simple algebraic group $G$ over an algebraically closed field $\K$ of characteristic zero as well as a maximal torus $T \subseteq G$ and a Borel subgroup $B \subseteq G$ containing $T$. Let $\Delta$ be the set of all simple roots corresponding to the choice of $B$. The associated weight lattice shall be denoted by $\Lambda$ and the monoid of dominant weights by $\Lambda^+$. Let $W$ be the Weyl group and $W_\lambda \subseteq W$ be the stabilizer of a weight $\lambda \in \Lambda$.

We fix a Schubert variety $X_\tau$ to some Weyl group coset $\tau \in W/W_Q$, where $Q \subseteq G$ is a parabolic subgroup containing $B$. The flag variety $G/Q$ can be embedded into a projective space by choosing a dominant weight $\lambda$, such that $\langle \lambda, \alpha^\vee \rangle = 0$, if and only if the simple reflection $s_\alpha$ is contained in $W_Q$. Equivalently, the stabilizer of $W_\lambda \subseteq W$ coincides with the subgroup $W_Q$. Let $v_\lambda$ be any highest weight vector in the irreducible representation $V(\lambda)$ of $G$. Then the parabolic subgroup $Q$ is the stabilizer of the highest weight space $\K v_\lambda$ and one obtains a closed embedding
\begin{align*}
    G/Q \hookrightarrow \PP(V(\lambda)), \quad gQ \mapsto [g \cdot v_\lambda].
\end{align*}
For each element $\sigma \in W/W_Q$ the weight space in $V(\lambda)$ of weight $\sigma(\lambda)$ is one-dimensional. Up to a non-zero scalar, one can therefore associate a unique weight vector $v_{\sigma(\lambda)} \in V(\lambda)$ of weight $\sigma(\lambda)$. The linear span of the orbit $B \cdot v_{\tau(\lambda)}$ is known as the \textit{Demazure module} associated to $\lambda$ and $\tau$, which we denote by $V(\lambda)_{\tau}$. As the Schubert variety $X_\tau$ can be written as the closure of the $B$-orbit $B \cdot [v_{\tau(\lambda)}] \subseteq \PP(V(\lambda))$, one can embed $X_\tau$ as a closed subvariety of $\PP(V(\lambda)_\tau)$. 

It is assumed that the reader is familiar with Littelmann path models, in particular with the path model of Lakshmibai-Seshadri-paths (LS-paths). It was originally introduced by Littelmann in~\cite{littelmann1994littlewood} and then further developed in \cite{littelmann1995plactic} and \cite{littelmann1995paths}. We can also recommend the appendix of~\cite{seshstratandschubvar} as an introduction to LS-paths, which is adapted to the language of Seshadri stratifications.

The primary goal of this \whatisthis{} is to construct multiprojective Seshadri stratifications on Schubert varieties, such that their fans of monoids correspond to tableau models of LS-paths (see Section~\ref{sec:LS-tableaux}). Unfortunately, not every tableau model has an associated stratification, as the condition~(\ref{itm:seshadri_strat_b}) may not be fulfilled. We discuss the obstacles in Section~\ref{sec:dcp}, which will lead to the notion of a \nice{\tau} index poset. If we do obtain a well-defined stratification, however, then it is normal and balanced and the resulting standard monomial theory (from~Proposition~\ref{prop:smt}) provides a geometric interpretation of the SMT established in~\cite[Section 6]{littelmann1998contracting}.

Our stratifications are direct generalizations of the work of Chiriv{\`i}, Fang and Littelmann in \cite{seshstratandschubvar}, where they proved that every Schubert variety $X_\tau \subseteq \PP(V(\lambda)_\tau)$ admits a Seshadri stratification via its Schubert subvarieties. Their construction made use of the strong connection between LS-paths and the vanishing multiplicity of \textit{extremal weight vectors} on Schubert varieties (see Remark\,3.4 and Prop.\,A.6 in~\loccit{}). An extremal weight vector is a non-zero function $\ell_\sigma$ of weight $-\sigma(\lambda)$ in the dual representation $V(\lambda)^*$, where $\sigma$ is some Weyl group element. Extremal weight vectors of a given weight are unique up to multiplication by a non-zero scalar, as their weight lies in the Weyl group orbit of the highest weight in $V(\lambda)^*$. 

The underlying poset $A = \set{\sigma \in W/W_Q \mid \sigma \leq \tau}$ of the stratification on $X_\tau \subseteq \PP(V(\lambda)_\tau)$ is induced by the Weyl group and the stratum to $\sigma \in A$ is the Schubert variety $X_\sigma \subseteq G/Q$ associated to $\sigma$. The extremal functions are given by all extremal weight vectors $\ell_\sigma \in V(\lambda)^*$ for $\sigma \in A$, restricted to $X_\tau \subseteq \PP(V(\lambda)_\tau)$. It was proved in~\cite{seshstratandschubvar} that this data forms a normal and balanced Seshadri stratification and that one can interpret the elements of degree $d \in \N_0$ in the associated fan of monoids $\Gamma$ via the Littelmann path model $\B(d\lambda)$ of LS-paths of shape $d\lambda$.

This stratification on $X_\tau$ of course depends on the choice of the dominant weight $\lambda$. However, one can also consider a decomposition $\lambda = \lambda_1 + \dots + \lambda_m$ into a sum of dominant weights, as this gives rise to the closed embedding
\begin{align}
    \label{eq:embedding_G_Q}
    G/Q \hookrightarrow \prod_{i=1}^m \PP(V(\lambda_i)), \quad gQ \mapsto ([g \cdot v_{\lambda_1}], \dots, [g \cdot v_{\lambda_m}]),
\end{align}
where $v_{\lambda_i}$ is a highest weight vector in $V(\lambda_i)$. The most well-known example is the Pl\"ucker embedding of a partial flag variety in type \texttt{A} into a product of fundamental representations, which we used in Example~\ref{ex:strat_type_A}. We therefore seek to generalize both the stratification from this example and the stratification on $X_\tau \subseteq \PP(V(\lambda)_\tau)$ to multiprojectively embedded Schubert varieties in arbitrary Dynkin types. The key for this intent is to consider other -- possibly non totally-ordered -- index posets $\mathcal I$. Therefore we choose the following objects for our construction:
\begin{itemize}
    \item A dominant weight $\lambda \in \Lambda^+$ and a sequence $\underline\lambda = (\lambda_1, \dots, \lambda_m)$ of dominant weights, that sum up to $\lambda$,
    \item a Schubert variety $X_\tau \subseteq G/Q$ for an element $\tau \in W/W_Q$, where $Q = B W_{\lambda} B$ is the parabolic subgroup associated to the stabilizer $W_\lambda \subseteq W$,
    \item and a subposet $\mathcal I$ of the power set poset $\mathcal P(\set{1, \dots, m}) \setminus \set{\varnothing}$, such that $\mathcal I$ is a graded poset of length $m-1$ and it holds
    \begin{align}
    \label{eq:need_for_s2}
    \underline J \subseteq I \ \, \Rightarrow \ \, J \subseteq I \quad \forall J, I \in \mathcal I.
    \end{align}
    \label{txt:def_underline_I}The subset $\underline J \subseteq J$ is defined exactly as in the section on LS-type stratifications in~\cite{ownarticle}: If $J$ is a minimal element in $\mathcal I$, then $\underline J \coloneqq J$, otherwise $\underline J$ is defined as the union of all $J \setminus K$, where $K \subsetneq J$ is a covering relation in $\mathcal I$.
\end{itemize}
The combinatorial requirement~(\ref{eq:need_for_s2}) on the poset $\mathcal I$ is necessary for the condition~\ref{itm:seshadri_strat_b} on a Seshadri stratification (see the proof of Theorem~\ref{thm:stratification}). We want to remark that there are two important cases, where this requirement is automatically satisfied, namely when $\mathcal I$ is totally ordered or equal to the full poset $\mathcal P(\set{1, \dots, m}) \setminus \set{\varnothing}$.

Regarding Weyl groups, we use the notation in Appendix~\ref{sec:weyl_groups}, namely the projection maps $\pi_Q$, the lifting maps $\min_Q$ and $\max_Q$ as well as the notion of $Q$-minimal and $Q$-maximal elements. For every $i \in [m]$ we define the parabolic subgroup
\begin{align*}
    P_i = B W_{\lambda_i} B
\end{align*}
and the projection $\tau_i = \pi_{P_i}(\tau)$ of $\tau$ to $W/W_{P_i}$. As the Schubert variety $X_\tau$ is the closure of the $B$-orbit through $w_\tau Q \in G/Q$ for a representative $w_\tau \in N_G(T)$ of $\tau$, the map~(\ref{eq:embedding_G_Q}) induces an embedding of $X_{\tau}$ into a product of projective spaces over Demazure modules:
\begin{align*}
    X_\tau \hookrightarrow \prod_{i=1}^m X_{\tau_i} \hookrightarrow \prod_{i=1}^m \PP(V(\lambda_i)_{\tau_i}).
\end{align*}
Here $X_{\tau_i}$ denotes the Schubert variety in $G/P_i$ to the element $\tau_i \in W/W_{P_i}$. This is the embedding we use for the multiprojective stratification on $X_\tau$. To obtain the construction from Example~\ref{ex:strat_type_A}, one can choose the sequence $\underline\lambda = (\omega_{k_1}, \dots, \omega_{k_m})$ of dominant weights and the index poset $\mathcal I = \Set{ \set{i, \dots, m} \mid i \in [m]}$. Note that the parabolic subgroups $P_i$ defined above usually do not coincide with the parabolic subgroups $P_i$ from Section~\ref{subsec:weyl_type_A}.

We need to fix some notation for the following sections. For a tuple $\underline d \in \N_0^m$ we define 
\begin{align*}
    \underline d \cdot \underline \lambda \coloneqq d_1 \lambda_1 + \dots + d_m \lambda_m \in \Lambda^+.
\end{align*}
To each index set $I \in \mathcal I$ we associate
\begin{itemize} 
    \item the degree $e_I = \sum_{i \in \underline I} e_i \in \N_0^m$,
    \item the dominant weight $\lambda_I = e_I \cdot \underline\lambda \in \Lambda^+$
    \item and the parabolic subgroup $P_I = B W_{\lambda_I} B = \cap_{i \in \underline I} P_i$.
\end{itemize}
It may not be intuitive to index these objects by $I$ instead of $\underline I$, but helps to simplify the notation. The parabolics $P_I$ take the role of the maximal parabolic subgroups $P_{k_1}, \dots, P_{k_m}$ from the Example~\ref{ex:strat_type_A} and the tuple $e_I$ is the multidegree of all the extremal functions for strata associated to the index set $I$. 

\label{txt:def_Q_tau}Let $Q_\tau$ be the unique parabolic subgroup containing $Q$, that is maximal with the property that $\tau$ is $Q_\tau$-maximal. This parabolic subgroup exists: The element $\tau$ is $Q'$-maximal for a parabolic subgroup $Q'$, if and only if $\ell(\tau s) < \ell(\tau)$ holds for all simple reflections $s \in W_{Q'}$ (see Corollary 2.4.5 in \cite{bjorner2006combinatorics}). Therefore $Q_\tau$ is equal to the subgroup which is generated by all parabolic subgroups $Q'$ containing $Q$, such that $\tau$ is $Q'$-maximal.

We can now define the parabolic subgroups
\begin{align}
    \label{eq:def_Q_I_and_Q^I}
    Q_I = \bigcap_{J \in \mathcal I \atop J \subseteq I} P_J \quad \text{and} \quad Q^I = Q_\tau \cap \bigcap_{J \in \mathcal I \atop J \supseteq I} P_J.
\end{align}
The groups $Q_I$ generalize the parabolic subgroups $Q_i$ from Example~\ref{ex:strat_type_A}. The relevance of the other groups $Q^I$ is not immediately visible in this example. They are only briefly mentioned in~\cite[Remark 11.14]{ownarticle}. However, the groups $Q^I$ play a very important role behind the scenes, as they are crucial for identifying, if the stratification is well-defined (see Theorem~\ref{thm:nice_I}).

\begin{lemma}
    \label{lem:properties_P_Q}
    The following properties hold for all $I \in \mathcal I$:
    \begin{enumerate}[label=(\alph{enumi})]
        \item $Q_I = \cap_{i \in I} \, P_i$ and $W_{Q_I}$ is the stabilizer of $\sum_{i \in I} \lambda_i$;
        \item if $J \subsetneq I$ is a covering relation in $\mathcal I$, then $P_I \cap Q_J = Q_I$;
        \item $Q_I \cap Q^I = Q_\tau$.
    \end{enumerate}
\end{lemma}
\begin{proof}
    $ $
    \begin{enumerate}[label=(\alph{enumi})]
        \item The definition of the index sets $\underline J$ for $J \in \mathcal I$ implies
        \begin{align*}
            \bigcup_{J \in \mathcal I \atop J \subseteq I} \underline J = I.
        \end{align*}
        In particular, we have $Q_I = \cap_{i \in I} P_i$. The Weyl subgroup $W_{Q_I}$ thus is the intersection of the stabilizers $W_{\lambda_i}$ over all $i \in I$, which is equal to the stabilizer of the weight $\sum_{i \in I} \lambda_i$.
        \item If $J \subsetneq I$ is a covering relation, the Weyl subgroup $W_{P_I \cap Q_J}$ is the stabilizer of $\sum_{i \in \underline I \cup J} \lambda_i = \sum_{i \in I} \lambda_i$ and it therefore coincides with $W_{Q_I}$. Hence $P_I \cap Q_J = Q_I$. 
        \item The equality $W_{Q_I \cap Q^I} = W_\tau$ follows from the fact that the subgroup $W_{Q_I \cap Q^I} \subseteq W$ is the intersection of $W_{Q_\tau}$ with the stabilizers $W_{\lambda_i}$ over all indices in the set
        \begin{align*}
            \bigcup_{J \subseteq I} \underline J \ \cup \ \bigcup_{J \supseteq I} \underline J = [m]. \tag*{\qedhere}
        \end{align*}
    \end{enumerate}
\end{proof}

\section{Lakshmibai-Seshadri-tableaux}
\label{sec:LS-tableaux}

To generalize the stratification from Example~\ref{ex:strat_type_A} we first need a suitable candidate for the underlying poset. It should be motivated by a combinatorial model which parametrizes basis of Demazure modules. Such a model was developed by Lakshmibai, Musili and Seshadri (\cite{lakshmibaiGP4}, \cite{lakshmibaiGP5}, \cite{seshadri2016introduction}) via certain sequences of Weyl group cosets, that admit a so called \textit{defining chain}. A few years later, Littelmann generalized their tableaux to arbitrary Dynkin types using his path model of LS-paths (see~\cite{littelmann1995plactic}). However, we use a slightly different notation than in \loccit{}: Instead of concatenations we consider tuples of LS-paths, and we call them LS-tableaux instead of LS-monomials.

Recall that an \textit{LS-path} $\pi$ of shape $\nu \in \Lambda^+$ is an element
\begin{align*}
    \pi = (\sigma_p > \dots > \sigma_1; 0, d_p, \dots, d_1 = 1),
\end{align*}
where $\sigma_p > \dots > \sigma_1$ is a chain in $W/W_\nu$ and $0 < d_p < \dots < d_1 = 1$ is a sequence of rational numbers, such that there exists a $(d_i, \nu)$-chain in $W/W_\nu$ from $\sigma_i$ to $\sigma_{i-1}$ for each $i = 2, \dots, p$. By definition, this is a chain $\sigma_i = \kappa_t > \dots > \kappa_0 = \sigma_{i-1}$ of covering relations in $W/W_\nu$ with the following integrality property: For every $j = 1, \dots, t$ the number $d_i \langle \kappa_j(\nu), \beta_j^\vee \rangle$ is an integer, where $\beta_j$ is the unique positive root of $G$ with $s_{\beta_j} \min_B(\kappa_{j-1}) = \min_B(\kappa_j)$. The Weyl group coset $\sigma_p$ is called the \textit{initial direction} of $\pi$ and is denoted by $i(\pi)$. 

The set $\mathbb B(\nu)$ of all LS-paths of shape $\nu$ can be interpreted in terms of the Littelmann path model (see \cite{littelmann1994littlewood} or \cite[Appendix A]{seshstratandschubvar}). The corresponding path model $\mathbb B(\pi_\nu)$ is generated by the straight-line path $\pi_\nu: [0,1] \to \Lambda \otimes_\Z \R$, $t \mapsto t \nu$.

We fix a sequence $\underline\mu = (\mu_1, \dots, \mu_s)$ of dominant weights with sum $\mu = \mu_1 + \dots + \mu_s$. 

\begin{definition}
    A \textbf{Lakshmibai-Seshadri-tableau} (short: LS-tableau) of shape $\underline\mu$ is a sequence $\underline\pi = (\pi_1, \dots, \pi_s)$ of LS-paths $\pi_i \in \mathbb B(\mu_i)$, called the \textbf{columns} of $\underline\pi$. Let $\sigma_{p_k}^{(k)} > \dots > \sigma_1^{(k)}$ be the chain of cosets in $W/W_{\mu_k}$ for the LS-path $\pi_k$, $k \in [s]$. For a fixed element $\tau \in W/W_\mu$ the LS-tableau $\underline\pi$ is called
    \begin{enumerate}[label=(\alph{enumi})]
        \item \textbf{\boldmath{}$\tau$-standard}, if there exists a weakly decreasing sequence
        \begin{align*}
            \overline\sigma_{p_1}^{(1)} \geq \dots \geq \overline\sigma_1^{(1)} \geq \dots \geq \overline\sigma_{p_k}^{(s)} \geq \dots \geq \overline\sigma_1^{(s)}
        \end{align*}
        in $W/W_\mu$, such that $\overline\sigma_{j}^{(i)} W_{\mu_i} = \sigma_{j}^{(i)} \in W/W_{\mu_i}$ holds for all $i = 1, \dots, s$ and $j = 1, \dots, p_i$. Such a sequence is called a \textbf{defining chain}.
        \item \textbf{weakly \boldmath{}$\tau$-standard}, if the LS-tableau $(\pi_k, \pi_{k+1})$ of shape $(\mu_k, \mu_{k+1})$ is $\tau$-standard for each $k = 1, \dots, s-1$.
    \end{enumerate}
\end{definition}

Note that defining chains are not unique, there can exist different defining chains for a given LS-tableau. As long as there is at least one defining chain, the tableau is $\tau$-standard. For every parabolic subgroup $Q'$ of $G$ with $W_{Q'} \subseteq W_\mu$, defining chains can also be lifted via the maps $\min_{Q'}$ and $\max_{Q'}$ to weakly decreasing sequences in $W/W_{Q'}$ consisting of lifts of the columns. As the defining chain in $W/W_\mu$ is bounded by $\tau$, its lifts to $W/W_{Q'}$ are bounded by $\max_{Q'}(\tau)$. Conversely, assume we have a weakly decreasing sequence in $W/W_{Q'}$ consisting of lifts of the columns and bounded by $\max_{Q'}(\tau)$. Such a chain clearly projects to a defining chain in $W/W_\mu$ via $W/W_{Q'} \twoheadrightarrow W/W_\mu$. Hence $W_\mu$ is the largest subgroup of $W$, where a defining chain is well-defined, as $W_\mu = W_{\mu_1} \cap \dots \cap W_{\mu_s}$.

When $\tau$ is equal to the unique maximal element $w_0 W_\mu \in W/W_\mu$, we often omit $\tau$ and just talk about (weakly) standard LS-tableaux.

\begin{example}
    Consider the group $G = \mathrm{SL}_4(\K)$ with the notation from the Appendix~\ref{subsec:weyl_type_A}. As all fundamental representations in type \texttt{A} are minuscule, LS-paths of shape $\omega_i$ correspond to Weyl group cosets in $W/W_{P_i}$. The tuple $\underline\pi = (13, 124, 3)$ is an LS-tableau of shape $(\omega_2, \omega_3, \omega_1)$. The stabilizer of $\mu = \omega_2 + \omega_3 + \omega_1$ is trivial, hence $W/W_\mu \cong W$. This tableau $\underline\pi$ is not standard: The element $3124$ is the unique minimal lift of $3 \in W/W_{P_1}$. By Deodhar's Lemma~\ref{lem:deodhar} we have unique minimal lift of $124$ that is greater or equal to $3124$, namely $4123$. But the unique maximal lift $3142$ of $13$ is not greater or equal to $4123$, hence there exists no defining chain for $\underline\pi$. However, $\underline\pi$ is weakly standard, since the tableaux $(13, 124)$ and $(124, 3)$ have the defining chains $1324 \geq 1243$ and $4123 \geq 3124$ respectively.
\end{example}

Let $\pi_{\mu_i}: [0,1] \to \Lambda \otimes_\Z \R$, $t \mapsto t \mu_i$ be the straight-line path to $\mu_i$ and $\mathbb B(\pi_{\mu_1} \ast \dots \ast \pi_{\mu_s})$ be the path model induced by the concatenation $\pi = \pi_{\mu_1} \ast \dots \ast \pi_{\mu_s}$, \ie it is the smallest set of piecewise linear paths which contains $\pi$ and is stable under the root operators. This path model is the connected component of the concatenation $\mathbb B(\pi_{\mu_1}) \ast \dots \ast \mathbb B(\pi_{\mu_s})$ of the LS-path models. As the path $\pi = \pi_{\mu_1} \ast \cdots \ast \pi_{\mu_s}$ and the straight-line path $\pi_\mu: [0,1] \to \Lambda \otimes_\Z \R$, $t \mapsto t \mu$ both have the same end point $\pi(1) = \mu = \pi_\mu(1)$ and their images stay in the dominant Weyl chamber, there exists a unique isomorphism of crystal graphs $\phi: \mathbb B(\pi_{\mu_1} \ast \cdots \ast \pi_{\mu_s}) \to \mathbb B(\pi_\mu)$ with $\phi(\pi) = \pi_\mu$ (see \cite{littelmann1995paths}). Using this isomorphism, Littelmann proved the following connection between $\tau$-standard LS-tableaux and path models.

\begin{proposition}[{\cite[Theorems 7, 8]{littelmann1995plactic}}]
    \label{prop:tau_standard_vs_path_model}
    An LS-tableau $\underline\pi = (\pi_1, \dots, \pi_s)$ of shape $\underline\mu$ is $\tau$-standard, if and only if the path $\pi \coloneqq \pi_{1} \ast \cdots \ast \pi_{s}$ is contained in the connected component $\mathbb B(\pi_{\mu_1} \ast \dots \ast \pi_{\mu_s}) \subseteq \mathbb B(\mu_1) \ast \dots \ast \mathbb B(\mu_s)$ and the initial direction $i(\phi(\pi))$ of the LS-path $\phi(\pi) \in \mathbb B(\mu)$ is smaller or equal to $\tau$.
\end{proposition}

It was also proved in \cite{littelmann1995plactic} that LS-tableaux give rise to a character formula for the Demazure modules.

\begin{theorem}[{\cite[Corollary 4]{littelmann1995plactic}}]
    Let $\B(\underline\mu)_\tau$ denote the set of all $\tau$-standard LS-tableaux of shape $\underline\mu$. Then the character of the Demazure module $V(\mu)_\tau$ is given by
    \begin{align}
        \label{eq:comb_demazure_character_formula_multi}
        \ch V(\mu)_\tau = \sum_{\underline\pi \in \B(\underline\mu)_\tau} e^{\underline\pi(1)},
    \end{align}
    where $\underline\pi(1)$ denotes the end point $(\pi_1 \ast \dots \ast \pi_s)(1)$ of the concatenation of all paths in the LS-tableau $\underline\pi = (\pi_1, \dots, \pi_s)$.
\end{theorem}

In the Appendix~\ref{sec:tableaux} we explain how LS-tableaux can be seen as a generalization of classical Young-tableaux and of the Young diagrams of admissable pairs, which were defined by Lakshmibai, Musili and Seshadri (see~\cite{lakshmibaiGP4}, \cite{lakshmibaiGP5}).

The Young-tableaux appearing in the fan of monoids to the stratification in Example~\ref{ex:strat_type_A} have a specific shape determined by the index poset $\mathcal I$: The number of boxes in any column is an element of the set $\set{k_1, \dots, k_m}$. Every element $I \in \mathcal I$ is of the form $I = \set{i, \dots, m}$ for some $i \in [m]$ and the parabolic subgroup $P_I$ coincides with $P_{k_i}$. Hence each column is associated to an element of $\mathcal I$, as columns of length $k_i$ can be viewed as an LS-path in $\mathbb B(\lambda_I)$ for $I = \set{i, \dots, m}$. The partial order on $\mathcal I$ defines the order in which the columns may appear. They follow a weakly increasing chain in $\mathcal I$, so that the lengths of the columns weakly increase from left to right.

We now define the type of tableaux, we hope to retrieve via the Seshadri stratification. Analogous to the above, they should be defined by the sequence of dominant weights $\underline\lambda$ and the index poset $\mathcal I$, which we fixed in Section~\ref{sec:choices}. Every column is associated to a parabolic subgroups $P_I$ for $I \in \mathcal I$. Again, the columns need to follow a chain in $\mathcal I$, but the index poset is not required to be totally ordered. As LS-paths are usually denoted by decreasing Weyl group cosets $\sigma_p > \dots > \sigma_1$ it is convenient to work with decreasing chains in $\mathcal I$ instead of the increasing chains, which the columns of a Young-tableau follow.

\begin{definition}
    \label{def:generalized_yt}
    An \textbf{LS-tableau of type \boldmath$(\underline\lambda, \mathcal I)$} is an LS-tableau $\underline\pi$ of shape $(\lambda_{I_1}, \dots, \lambda_{I_s})$, where $I_1 \supseteq \dots \supseteq I_s$ is a (possibly empty) weakly decreasing sequence in $\mathcal I$. We call the tuple $\deg \underline\pi = e_{I_1} + \dots + e_{I_s} \in \N_0^m$ the \textbf{degree} of $\underline\pi$.
\end{definition}

\begin{example}
    We choose the index poset $\mathcal I = \set{\set{i, \dots, m} \mid i \in [m]}$ and the sequence $\underline\lambda = (\omega_1, \dots, \omega_{n-1})$ of all fundamental weights for the group $\mathrm{SL}_n(\K)$ (see Section~\ref{subsec:weyl_type_A} for all type \texttt{A} notation). By reversing the order of their columns, the LS-tableaux of type $(\underline\lambda, \mathcal I)$ correspond bijectively to all Young-tableaux with less than $n$ rows and entries in $[n]$ (cf. Appendix~\ref{sec:tableaux}). Such an LS-tableau is standard, if and only if the corresponding Young-tableaux is semistandard.

    The index poset $\mathcal I = \set{[i] \mid i \in [m]}$ leads to similar tableaux, which we call \textit{Anti-Young-tableaux} (again, see Appendix~\ref{sec:tableaux}). We prove later in Corollary~\ref{cor:totally_ordered_nice_I} that Anti-Young-tableaux can also be obtained via Seshadri stratifications.
\end{example}

\begin{remark}
    \label{rem:shape_to_degree}
    For for each $\underline d = (d_1, \dots, d_m) \in \N_0^m$ there exists a weakly decreasing sequence $I_1 \supseteq \dots \supseteq I_s$ in $\mathcal I$, such that the LS-tableaux of shape $(\lambda_{I_1}, \dots, \lambda_{I_s})$ have degree $\underline d$: This is clearly true for $m = 1$. If $m \geq 2$, we choose an index $i \in \underline I$ for $I = [m]$, where $d_i$ is minimal. As the $i$-th entry of $\underline d - d_i e_I$ is zero, we can find a weakly decreasing sequence $I_1 \supseteq \dots \supseteq I_s$ with $\sum_{k=1}^s e_{I_k} = \underline d - d_i e_I$ by induction. If we append $[m]$ exactly $d_i$ times to the start of this sequence, we thus get a sequence for $\underline d$. 
    
    It is not obvious that this sequence $I_1 \supseteq \dots \supseteq I_s$ is uniquely determined by $\underline d$. This follows later via Corollary~\ref{cor:unique_shape}. Note that the LS-tableaux for each fixed sequence to a degree $\underline d \in \N_0^m$ give rise to a character formula for the Demazure module $V(\underline d \cdot \underline\lambda)_\tau$.
\end{remark}

\begin{definition}
    $ $
    We define the following posets and monotone maps:
    \begin{enumerate}[label=(\alph{enumi})]
        \item Let $\Wlambda$ be the direct product of the posets $\set{\sigma \in W/W_Q \mid \sigma \leq \tau}$ and $\mathcal I$, \ie the order relation is given by
        \begin{align*}
        (\theta, I) \geq (\phi, J) \quad :\Longleftrightarrow \quad I \supseteq J \quad \text{and} \quad \theta \geq \phi
        \end{align*}
        for all $(\theta, I), (\phi, J) \in \Wlambda$.
        \item Let $\ulWlambda = \coprod_{I \in \mathcal I} \, \set{\theta \in W/W_{P_I} \mid \theta \leq \pi_{P_I}(\tau)} \times \set{I}$ be the poset, which partial order is given by the transitive hull of the following relation:
        \begin{equation}
        \label{eq:quotient_relation}
        (\theta, I) \geq (\phi, J) \quad :\Longleftrightarrow \quad I \supseteq J \ \ \text{and} \ \ \mathrm{max}_Q(\theta) \geq \mathrm{min}_Q(\phi)
        \end{equation}
        for all $(\theta, I), (\phi, J) \in \ulWlambda$. We also denote the order of $\ulWlambda$ by $\geq$.
        \item Furthermore, we define the map $\pi_{P_{\mathcal I}}: \Wlambda \to \ulWlambda$, $(\theta, I) \mapsto (\pi_{P_I}(\theta), I)$, which is clearly monotone.
        \item Let $\underline\theta$ be a chain in $\ulWlambda$ of elements $(\theta_\ell, I_\ell) > \dots > (\theta_0, I_0)$. We say that $\underline\theta$ is \textbf{\boldmath $\tau$-standard}, if it has a \textbf{defining chain}, that is to say a chain $(\overline\theta_\ell, I_\ell) > \dots > (\overline\theta_0, I_0)$ in $\Wlambda$ with $\pi_{P_{\mathcal I}}(\overline\theta_k, I_k) = (\theta_k, I_k)$ for all $k = 0, \dots, \ell$.
    \end{enumerate}
\end{definition}

\begin{lemma}
    \label{lem:relation_ulWlambda}
    For $(\theta, I), (\phi, J) \in \ulWlambda$ with $J \subseteq I$ the condition $\mathrm{max}_Q(\theta) \geq \mathrm{min}_Q(\phi)$ in (\ref{eq:quotient_relation}) is equivalent to each of the following:
    \begin{enumerate}[label=(\alph{enumi})]
        \item $\pi_{P_J} \circ \max_{Q_I}(\theta) \geq \phi$;
        \item \label{itm:relation_ulWlambda_b} there exists a parabolic subgroup $Q \subseteq Q' \subseteq P_I \cap P_J$ and lifts $\overline\theta$ and $\overline\phi$ in $W/W_{Q'}$ of $\theta$ and $\phi$ respectively, such that $\overline\theta \geq \overline\phi$ in $W/W_{Q'}$.
    \end{enumerate}
\end{lemma}
\begin{proof}
    The inclusion $J \subseteq I$ implies $P_J \subseteq Q_I$. Projecting the condition $\max_Q(\theta) \geq \min_Q(\phi)$ to $W/W_{P_J}$ yields $\pi_{P_J} \circ \max_{Q_I}(\theta) \geq \phi$ and this inequality in $W/W_{P_J}$ lifts back to $\max_Q(\theta) \geq \max_Q \circ \pi_{P_J} \circ \max_{Q_I}(\theta) \geq \max_Q(\phi) \geq \min_Q(\phi)$. Clearly, $\max_Q(\theta) \geq \min_Q(\phi)$ implies condition~\ref{itm:relation_ulWlambda_b}. Conversely, the relation $\overline\theta \geq \overline\phi$ of lifts in $W/W_{Q'}$ gives rise to the inequality $\max_{Q'}(\theta) \geq \theta \geq \phi \geq \min_{Q'}(\phi)$, which in turn lifts to $\max_Q(\theta) \geq \min_Q(\phi)$.
\end{proof}

Let $\underline\pi = (\pi_1, \dots, \pi_s)$ be an LS-tableau of type $(\underline\lambda, \mathcal I)$ and let $\sigma_{p_k}^{(k)} > \dots > \sigma_1^{(k)}$ denote the sequence of elements in $W/W_{P_{I_k}}$ of the LS-path $\pi_k \in \mathbb B(\lambda_{I_k})$ for each $k \in [s]$. Then weak $\tau$-standardness can be described using the poset $\ulWlambda$:
\begin{align}
    \label{eq:tableau_chain}
    \underline\pi \ \text{is} \ &\text{weakly} \ \tau\text{-standard} \quad \Longleftrightarrow \quad \nonumber\\
    &(\sigma_{p_1}^{(1)}, I_1) \geq \dots \geq (\sigma_1^{(1)} \mkern-2mu , I_1) \geq \dots \geq (\sigma_{p_s}^{(s)}, I_s) \geq \dots \geq (\sigma_1^{(s)} \mkern-2mu , I_s) \mkern10mu \text{in $\ulWlambda$}
\end{align}
The LS-tableau $\underline\pi$ is $\tau$-standard, if and only if it is weakly $\tau$-standard and the chain one obtains from (\ref{eq:tableau_chain}) by erasing all duplicates is $\tau$-standard.

\ytableausetup{smalltableaux}
\begin{figure}
    \begin{center}
        \begin{tikzpicture}[scale=0.7]
        \node (p4) at (-1.6,-7.5) {$(4, [1])$};
        \node (p3) at (0,-9) {$(3, [1])$};
        \node (p2) at (0,-10.5)  {$(2, [1])$};
        \node (p1) at (0,-12) {$(1, [1])$};
        \node (p34) at (0,1.5) {$(34, [3])$};
        \node (p24) at (0,0) {$(24, [3])$};
        \node (p14) at (1.6,-1.5) {$(14, [3])$};
        \node (p23) at (-1.6,-1.5) {$(23, [3])$};
        \node (p13) at (3.2,-3.0) {$(13, [3])$};
        \node (p12) at (1.6,-4.5) {$(12, [3])$};
        \node (p234) at (-3.2,-3.0) {$(234, [2])$};
        \node (p134) at (-1.6,-4.5) {$(134, [2])$};
        \node (p124) at (0,-6) {$(124, [2])$};
        \node (p123) at (1.6,-7.5) {$(123, [2])$};
        
        \draw [thick, shorten <=-2pt, shorten >=-2pt, -stealth] (p34) -- (p24);
        \draw [thick, shorten <=-2pt, shorten >=-2pt, -stealth] (p24) -- (p14);
        \draw [thick, shorten <=-2pt, shorten >=-2pt, -stealth] (p14) -- (p13);
        \draw [thick, shorten <=-2pt, shorten >=-2pt, -stealth] (p24) -- (p23);
        \draw [thick, shorten <=-2pt, shorten >=-2pt, -stealth] (p23) -- (p13);
        \draw [thick, shorten <=-2pt, shorten >=-2pt, -stealth] (p13) -- (p12);
        \draw [thick, shorten <=-2pt, shorten >=-2pt, -stealth] (p234) -- (p134);
        \draw [thick, shorten <=-2pt, shorten >=-2pt, -stealth] (p134) -- (p124);
        \draw [thick, shorten <=-2pt, shorten >=-2pt, -stealth] (p124) -- (p123);
        \draw [thick, shorten <=-2pt, shorten >=-2pt, -stealth] (p4) -- (p3);
        \draw [thick, shorten <=-2pt, shorten >=-2pt, -stealth] (p3) -- (p2);
        \draw [thick, shorten <=-2pt, shorten >=-2pt, -stealth] (p2) -- (p1);
        \draw [thick, shorten <=-2pt, shorten >=-2pt, -stealth] (p23) -- (p234);
        \draw [thick, shorten <=-2pt, shorten >=-2pt, -stealth] (p13) -- (p134);
        \draw [thick, shorten <=-2pt, shorten >=-2pt, -stealth] (p12) -- (p124);
        \draw [thick, shorten <=-2pt, shorten >=-2pt, -stealth] (p124) -- (p4);
        \draw [thick, shorten <=-2pt, shorten >=-2pt, -stealth] (p123) -- (p3);
        \end{tikzpicture}
    \end{center}
    \caption{Hasse-diagram of $\ulWlambda$ in type $\texttt{A}_3$.}
    \label{fig:hasse_quotient_weird_a_3}
\end{figure}
\ytableausetup{nosmalltableaux}

The poset $\ulWlambda$ generalizes the poset $\ulW$ from Example~\ref{ex:strat_type_A}. The relation (\ref{eq:quotient_relation}) is reflexive and antisymmetric, but not transitive in general. As an example, consider the sequence $\underline\lambda = (\omega_1, \omega_3, \omega_2)$ of fundamental weights in Dynkin type $\texttt{A}_3$, the unique maximal element $\tau = w_0$ in $W$ and the index poset $\mathcal I = \set{[1], [2], [3]}$. Then we have $(12, [3]) \geq (124, [2])$ and $(124, [2]) \geq (4, [1])$ but $(12, [3]) \ngeq (4, [1])$. Again, we used the one-line notation from Appendix~\ref{subsec:weyl_type_A} for elements in $W/W_{P_i}$. The complete poset $\ulWlambda$ is shown in Figure~\ref{fig:hasse_quotient_weird_a_3}. Note that it cannot be the underlying poset $A$ of a multiprojective Seshadri stratification on $X = G/B$, since the length $\ell = 9$ of the poset does not coincide with $\dim \hat X - 1 = 8$.

\section{The defining chain poset}
\label{sec:dcp}

Recall that we fixed a sequence $\underline\lambda$ of dominant weights, an index poset $\mathcal I$ satisfying a certain property and a coset $\tau \in W/W_Q$ in Section~\ref{sec:choices}. We also make use of all notation introduced in this section without mention.

We now construct a poset $\DCP{}$, which serves as the underlying poset $A$ for the multiprojective stratification on $X_\tau$. This construction heavily relies on Theorem~\ref{thm:unique_def_chain}, but before we can state and prove it, we need a few more results about defining chains. Note that the poset $\DCP{}$ is built specifically for the LS-tableaux of type $(\underline\lambda, \mathcal I)$, so that we can hopefully obtain a geometric interpretation of these tableaux via the fan of monoids of the stratification.

\begin{lemma}
    Every $\tau$-standard chain $\underline\theta: (\theta_\ell, I_\ell) > \dots > (\theta_0, I_0)$ in $\ulWlambda$ has a unique maximal and a unique minimal defining chain
    \begin{align*}
        \overline\theta^{\mathrm{max}} \!\! : \; (\overline\theta^{\mathrm{max}}_\ell \mkern-2mu , I_\ell) > \dots > (\overline\theta^{\mathrm{max}}_0 \mkern-2mu , I_0) \quad \text{and} \quad \overline\theta^{\mathrm{min}} \!\! : \; (\overline\theta^{\mathrm{min}}_\ell \mkern-2mu , I_\ell) > \dots > (\overline\theta^{\mathrm{min}}_0 \mkern-2mu , I_0),
    \end{align*}
    \ie for every defining chain $\overline\theta: (\overline\theta_\ell, I_\ell) > \dots > (\overline\theta_0, I_0)$ of $\underline\theta$ it holds $\overline\theta^{\mathrm{max}}_k \geq \overline\theta_k \geq \overline\theta^{\mathrm{min}}_k$ for all $k = 0, \dots, \ell$.
\end{lemma}
\begin{proof}
    We only proof the statements about the unique maximal defining chain, the other statement follows analogously. Since $\underline\theta$ is $\tau$-standard, there exists a defining chain $(\overline\theta_\ell, I_\ell) > \dots > (\overline\theta_0, I_0)$. In particular, we have $\tau \geq \overline\theta_\ell$, so via Deodhar's Lemma~\ref{lem:deodhar} we can choose a unique maximal lift $\overline\theta_\ell^{\mathrm{max}} \in W/W_Q$ of $\theta_\ell$ that is smaller or equal to $\tau$. Then $\overline\theta_\ell^{\mathrm{max}} \geq \overline\theta_\ell$. For all $k = \ell - 1, \dots, 1$ we now iteratively choose a lift $\overline\theta_k^{\mathrm{max}}$, such that $\overline\theta_k^{\mathrm{max}} \geq \overline\theta_k$. Since we have $\overline\theta_{k+1}^{\mathrm{max}} \geq \overline\theta_{k+1} \geq \overline\theta_k$, there exists a unique maximal lift $\overline\theta_k^{\mathrm{max}} \in W/W_Q$ of $\theta_k$ with $\overline\theta_{k+1}^{\mathrm{max}} \geq \overline\theta_k^{\mathrm{max}}$ and this lift fulfills $\overline\theta_k^{\mathrm{max}} \geq \overline\theta_k$. By construction, we thus obtain the unique maximal defining chain of $\underline\theta$.
\end{proof}

\begin{lemma}
    \label{lem:operations_on_def_chains}
    Let $\underline\theta: (\theta_\ell, I_\ell) > \dots > (\theta_0, I_0)$ be a $\tau$-standard sequence in $\ulWlambda$ and $(\overline\theta_\ell, I_\ell) > \dots > (\overline\theta_0, I_0)$ be a defining chain for $\underline\theta$. For each $k \in \set{0, \dots, \ell}$ we define the parabolic subgroups
    \begin{align*}
    Q^k = Q_\tau \cap \bigcap_{r=k}^\ell P_{I_r} \quad \text{and} \quad Q_k = \bigcap_{r=0}^k P_{I_r}
    \end{align*}
    as well as the following elements:
    \begin{align*}
    \overline\theta_k^\vartriangle = \operatorname{max}_Q \circ \pi_{Q^{k}}(\overline\theta_k) \quad \text{and} \quad \overline\theta_k^\triangledown = \operatorname{min}_Q \circ \pi_{Q_{k}}(\overline\theta_k).
    \end{align*}
    Then $(\overline\theta_\ell^\vartriangle, I_\ell) > \dots > (\overline\theta_0^\vartriangle, I_0)$ and $(\overline\theta_\ell^\triangledown, I_\ell) > \dots > (\overline\theta_0^\triangledown, I_0)$ are also defining chains for $\underline\theta$ satisfying $\overline\theta_k^\vartriangle \geq \overline\theta_k \geq \overline\theta_k^\triangledown$ for each $k \in \set{0, \dots, \ell}$. In particular, the lift of $(\theta_k, I_k)$ in the unique maximal/minimal defining chain of $\underline\theta$ is $Q^{k}$-maximal/$Q_{k}$-minimal respectively.
\end{lemma}
\begin{proof}
    Again, we only prove the statements about the chain $(\overline\theta_\ell^\vartriangle, I_\ell) > \dots > (\overline\theta_0^\vartriangle, I_0)$. Since $Q^{k} \subseteq P_{I_k}$, the element $\overline\theta_k^\vartriangle$ is still a lift of $\theta_k$ in $W/W_Q$ and by definition we have $\overline\theta_k^\vartriangle \geq \overline\theta_k$. The relation $\tau \geq \overline\theta_\ell$ together with the fact, that $\tau$ is $Q^\ell$-maximal, implies $\tau = \max_Q \circ \pi_{Q^{\ell}}(\tau) \geq \max_Q \circ \pi_{Q^{\ell}}(\overline\theta_\ell) = \overline\theta_\ell^\vartriangle$. By monotony of the maps $\max_Q$ and $\pi_{Q_{I_k}}$ and the inclusion $Q^{k-1} \subseteq Q^k$ we get 
    \begin{align*}
        \overline\theta_k^\vartriangle = \operatorname{max}_Q \circ \pi_{Q^{I_k}}(\overline\theta_k) \geq \operatorname{max}_Q \circ \pi_{Q^{I_k}}(\overline\theta_{k-1}) \geq \operatorname{max}_Q \circ \pi_{Q^{I_{k-1}}}(\overline\theta_{k-1}) = \overline\theta_{k-1}^\vartriangle.
    \end{align*}
    Therefore $(\overline\theta_\ell^\vartriangle, I_\ell) > \dots > (\overline\theta_0^\vartriangle, I_0)$ is a defining chain for $\underline\theta$.
\end{proof}

Note that the parabolic subgroup $Q_k$ in Lemma~\ref{lem:operations_on_def_chains} coincides with the group $Q_{I_k}$ for every $k = 0, \dots, \ell$. The analogous statement does not hold for $Q^k$. In general, we just have the inclusion $Q^{I_k} \subseteq Q^k$.

\begin{theorem}
    \label{thm:unique_def_chain}
    Every maximal $\tau$-standard chain $\underline\theta: (\theta_\ell, I_\ell) > \dots > (\theta_0, I_0)$ in $\ulWlambda$ has a unique defining chain $(\overline\theta_\ell, I_\ell) > \dots > (\overline\theta_0, I_0)$ and this chain is a maximal chain in $\Wlambda$. Additionally, the element $\overline\theta_k \in W/W_Q$ is $Q_k$-minimal and $Q^k$-maximal (using the parabolic subgroups from Lemma~\ref{lem:operations_on_def_chains}).
\end{theorem}
\begin{proof}
    Let $(\overline\theta_\ell, I_\ell) > \dots > (\overline\theta_0, I_0)$ be the unique maximal defining chain of $\underline\theta$. The largest element $(\theta_\ell, I_\ell)$ in $\underline\theta$ is equal to $(\tau, [m])$, otherwise $(\pi_{P_{[m]}}(\tau), [m]) > (\theta_\ell, I_\ell) > \dots > (\theta_0, I_0)$ would be a longer $\tau$-standard chain.
    
    By Lemma~\ref{lem:operations_on_def_chains} every lift $\overline\theta_{k}$ is $Q^{k}$-maximal. We first prove by descending induction over $k = \ell, \dots, 0$, that $\overline\theta_k$ is $Q_{k}$-minimal as well. The element $\overline\theta_\ell = \tau$ is $Q_\ell$-minimal, since $Q_\ell = Q$. Now suppose that $\overline\theta_k$ is $Q_{k}$-minimal for some $k < \ell$. We show, that $\overline\theta_{k-1}$ is $Q_{k-1}$-minimal. In order to keep indices to a minimum, we write $I = I_k$ and $J = I_{k-1}$. We need to differentiate between two cases: $I = J$ and $I \neq J$.
    
    First, suppose that $I = J$. Let $B = [\overline\theta_k, \min_Q \circ \pi_{Q_I}(\overline\theta_{k-1})]$ be the Bruhat interval of all $\sigma \in W/W_Q$ with $\overline\theta_k \geq \sigma \geq \min_Q \circ \pi_{Q_I}(\overline\theta_{k-1})$. The image of $B$ via $\pi_{P_I}$ is exactly $\set{\theta_k, \theta_{k-1}}$. Otherwise there exists an element $\phi \in \pi_{P_I}(B)$ and a lift $\overline\phi$ of $\phi$ in $W/W_Q$ such that $\theta_k > \phi > \theta_{k-1}$ and $\overline\theta_k > \overline\phi > \min_Q \circ \pi_{Q_I}(\overline\theta_{k-1})$. By inserting $(\phi, I)$ between $(\theta_k, I)$ and $(\theta_{k-1}, I)$ we get a longer chain in $\ulWlambda$, which is still $\tau$-standard, since we can use Lemma~\ref{lem:operations_on_def_chains} to construct the following defining chain:
    \begin{align*}
    (\overline\theta_\ell, I_\ell) > \dots > (\overline\theta_k, I) > (\overline\phi, I) > (\operatorname{min}_Q \circ \pi_{Q_I}(\overline\theta_{k-1}), I) \geq (\overline\theta_{k-1}^\triangledown, I_{k-1}) > \dots > (\overline\theta_0^\triangledown, I_0).
    \end{align*}
    The image of $B$ under the projection $\pi_{Q_I}$ is equal to the Bruhat interval $[\pi_{Q_I}(\overline\theta_k), \pi_{Q_I}(\overline\theta_{k-1})]$ in $W/W_{Q_I}$, because both $\overline\theta_k$ and $\min_Q \circ \pi_{Q_I}(\overline\theta_{k-1})$ are $Q_I$-minimal. The element $\pi_{Q_I}(\overline\theta_{k-1})$ is the unique maximal lift of $\theta_{k-1}$ in $W/W_{Q_I}$, which is less or equal to $\pi_{Q_I}(\overline\theta_k)$. Otherwise there would exist a lift $\psi \in W/W_{Q_I}$ of $\theta_{k-1}$ such that $\pi_{Q_I}(\overline\theta_k) \geq \psi > \pi_{Q_I}(\overline\theta_{k-1})$. Taking the $Q$-maximum yields:
    \begin{align*}
    \overline\theta_k = \operatorname{max}_Q \circ \pi_{Q_I}(\overline\theta_k) \geq \operatorname{max}_Q(\psi) > \operatorname{max}_Q \circ \pi_{Q_I}(\overline\theta_{k-1}) \geq \overline\theta_{k-1}.
    \end{align*}
    But this is a contradiction to the construction of $\overline\theta_{k-1}$ as it is the unique maximal lift of $\theta_{k-1}$ in $W/W_Q$ such that $\overline\theta_{k-1} \leq \overline\theta_k$.
    
    Combining our observations, wee see that the only element in $\pi_{Q_I}(A)$, which does not project to $\theta_k$ is $\pi_{Q_I}(\overline\theta_{k-1})$. Using Lemma~\ref{lem:bruhat_interval} on $\pi_{Q_I}(\overline\theta_k) > \pi_{Q_I}(\overline\theta_{k-1})$ it now follows, that this is a covering relation in $W/W_{Q_I}$. It lifts to the covering relation $\overline\theta_k > \min_Q \circ \pi_{Q_I}(\overline\theta_{k-1})$ in $W/W_Q$ and since $\overline\theta_{k-1}$ lies in between them, it is $Q_I$-minimal.
    
    It remains the case $I \neq J$. Let $K \subsetneq I$ be a covering relation in $\mathcal I$ such that $J \subseteq K$. Then we have $J = K$ and $\overline\theta_{k-1} = \overline\theta_k$, since $\underline\theta$ is maximal $\tau$-standard and the inequalities $(\theta_k, I) > (\pi_{P_K}(\overline\theta_k), K) \geq (\theta_{k-1}, J)$ in $\ulWlambda$ can be lifted to $(\overline\theta_k, I) > (\overline\theta_k, K) \geq (\overline\theta_{k-1}, J)$.

    We now prove that $\overline\theta_k$ is $Q_J$-minimal, \ie $\min_Q \circ \pi_{Q_J}(\overline\theta_k) = \overline\theta_k$. Notice that both sides are $Q_I$-minimal. The left hand side is even $Q_J$-minimal and $\overline\theta_k$ is $Q_I$-minimal by induction. It therefore suffices to prove $\pi_{Q_I} \circ \min_Q \circ \pi_{Q_J}(\overline\theta_k) = \pi_{Q_I}(\overline\theta_k)$. As $Q_I = Q_J \cap P_I$ we can use the following argument, which will appear a few more times throughout this \whatisthis{}: The natural map $W/W_{Q_I} \to W/W_{Q_J} \times W/W_{P_I}$ is an isomorphism of posets onto its image . Showing an equality in $W/W_{Q_I}$ can thus be reduced to showing the projected equalities in both $W/W_{Q_J}$ and $W/W_{P_I}$.
    
    The images of the two elements $\min_Q \circ \pi_{Q_J}(\overline\theta_k)$ and $\overline\theta_k$ are equal in $W/W_{P_I}$. Otherwise we could extend the chain $\underline\theta$ to the longer $\tau$-standard chain
    \begin{align*}
    (\theta_\ell, I_\ell) > \dots > (\theta_k, I) > (\pi_{P_I} \circ \operatorname{min}_Q \circ \pi_{Q_J}(\overline\theta_k), I) > (\theta_{k-1}, J) > \dots > (\theta_0, I_0)
    \end{align*}
    as it has the following defining chain:
    \begin{align*}
    (\overline\theta_\ell, I_\ell) > \dots > (\overline\theta_k, I) > (\operatorname{min}_Q \circ \pi_{Q_J}(\overline\theta_k), I) > (\overline\theta_{k-1}^\triangledown, J) > \dots > (\overline\theta_0^\triangledown, I_0).
    \end{align*}
    The images of the two elements $\min_Q \circ \pi_{Q_J}(\overline\theta_k)$ and $\overline\theta_k$ are also equal in $W/W_{Q_J}$, which completes proving the $Q_J$-minimality of $\overline\theta_k$.
    
    We still need to show that $\underline\theta$ has a unique defining chain and compute its length. We know that there is a unique minimal defining chain $(\overline\theta_\ell^\mathrm{min}, I_\ell) > \dots > (\overline\theta_0^\mathrm{min}, I_0)$ and a unique maximal defining chain $(\overline\theta_\ell^\mathrm{max}, I_\ell) > \dots > (\overline\theta_0^\mathrm{max}, I_0)$ for $\underline\theta$. It is easy to see, that $\underline\theta$ ends at the element $\theta_0 = \id W_{P_{I_0}}$. Its lift in the maximal defining chain is $Q_0$-minimal, hence $\overline\theta_0^\mathrm{max} = \id W_Q = \overline\theta_0^\mathrm{min}$. We can now work ourselves inductively through the two defining chains, showing $\overline\theta_k^\mathrm{max} = \overline\theta_k^\mathrm{min}$ for $k = 1, \dots, \ell$. It always holds $\overline\theta_{k+1}^\mathrm{max} \geq \overline\theta_{k+1}^\mathrm{min} \geq \overline\theta_{k}^\mathrm{min} = \overline\theta_{k}^\mathrm{max}$. If $I_{k+1} = I_k$, then $\overline\theta_{k+1}^\mathrm{max} > \overline\theta_k^\mathrm{max}$ is a covering relation and $\overline\theta_{k+1}^\mathrm{min} \neq \overline\theta_{k}^\mathrm{min}$. For $I_{k+1} \neq I_k$ we have $\overline\theta_{k+1}^\mathrm{max} = \overline\theta_k^\mathrm{max}$. In both cases it follows $\overline\theta_{k+1}^\mathrm{max} = \overline\theta_{k+1}^\mathrm{min}$.
    
    As the minimal and maximal defining chain coincide, there is exactly one defining chain for $\underline\theta$. Its first element is $\tau$ and its last element is $\id W_Q$. In between we only have covering relations and $m-1$ equalities representing the change of the subset $W/W_{P_I} \subseteq \ulWlambda$. So the chain $\underline\theta$ is of length $r(\tau) + m - 1$, where $r(\tau)$ denotes the rank of $\tau$ in $W/W_Q$, hence $\underline\theta$ is a maximal chain in $\Wlambda$.
\end{proof}

\begin{definition}
    The \textbf{defining chain poset} $\DCP{} \subseteq \Wlambda$ consists of all elements $(\theta, I) \in \Wlambda$, which are contained in the unique defining chain of a maximal $\tau$-standard chain in $\ulWlambda$. The order relation $\succeq$ on $\DCP{}$ is given by
    \begin{align*}
    (\theta, I) \succeq (\phi, J) \quad \Longleftrightarrow \quad &\text{$(\theta, I) \geq (\phi, J)$ in $\Wlambda$ and there exists a maximal } \\
    &\text{$\tau$-standard chain in $\ulWlambda$, such that $(\theta, I)$ and $(\phi, J)$} \\
    &\text{are contained in its unique defining chain}.
    \end{align*}
\end{definition}

\begin{remark}
    \label{rem:relation_dcp}
    For two elements $(\theta, I), (\phi, J) \in \DCP{}$ the relation $(\theta, I) \succeq (\phi, J)$ is equivalent to the existence of a \textbf{\boldmath{}$P_{\mathcal I}$-chain} from $(\theta, I)$ to $(\phi, J)$, by which we mean a chain of covering relations in $\Wlambda$ from $(\theta, I)$ to $(\phi, J)$, which projects to a chain of the same length via the map $\pi_{P_{\mathcal I}}$. In particular, the relation $\succeq$ is reflexive, antisymmetric and transitive.
\end{remark}

\begin{example}
    \label{ex:small_dcp}
    As a first example consider the group $G = \mathrm{SL}_3(\K)$, $\tau = 312$, $\underline\lambda = (\omega_2, \omega_1)$ and $\mathcal I = \set{[1], [2]}$. The poset $\ulWlambda$ looks as follows in this case:
    \begin{center}
        \begin{tikzpicture}
        \node (3) at (0,0) {$(3, [2])$};
        \node (2) at (2,0) {$(2, [2])$};
        \node (1) at (4,0) {$(1, [2])$};
        \node (13) at (6,0) {$(13, [1])$};
        \node (12) at (8,0) {$(12, [1])$};
        
        \draw [thick, shorten <=-2pt, shorten >=-2pt, -stealth] (3) -- (2);
        \draw [thick, shorten <=-2pt, shorten >=-2pt, -stealth] (2) -- (1);
        \draw [thick, shorten <=-2pt, shorten >=-2pt, -stealth] (1) -- (13);
        \draw [thick, shorten <=-2pt, shorten >=-2pt, -stealth] (13) -- (12);
    \end{tikzpicture}
    \end{center}
    The maximal $\tau$-standard chains have length $4$ by Theorem~\ref{thm:unique_def_chain}, so they contain all but one element of $\ulWlambda$. This element can only be $(2, [2])$ or $(13, [1])$. In total, we get the defining chain poset from Figure~\ref{fig:small_dcp}.
\end{example}

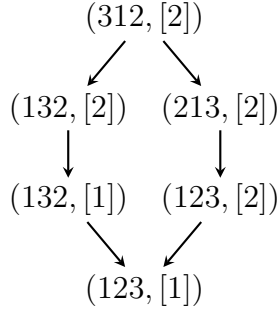
\begin{figure}
    \begin{center}
        \begin{tikzpicture}
        \node (3) at (0,0) {$(312, [2])$};
        \node (2) at (1,-1.2) {$(213, [2])$};
        \node (123) at (1,-2.4) {$(123, [2])$};
        \node (132) at (-1,-1.2) {$(132, [2])$};
        \node (13) at (-1,-2.4) {$(132, [1])$};
        \node (12) at (0,-3.6) {$(123, [1])$};
        
        \draw [thick, shorten <=-2pt, shorten >=-2pt, -stealth] (3) -- (2);
        \draw [thick, shorten <=-2pt, shorten >=-2pt, -stealth] (2) -- (123);
        \draw [thick, shorten <=-2pt, shorten >=-2pt, -stealth] (123) -- (12);
        \draw [thick, shorten <=-2pt, shorten >=-2pt, -stealth] (3) -- (132);
        \draw [thick, shorten <=-2pt, shorten >=-2pt, -stealth] (132) -- (13);
        \draw [thick, shorten <=-2pt, shorten >=-2pt, -stealth] (13) -- (12);
        \end{tikzpicture}
    \end{center}
    \caption{Defining chain poset from Example~\ref{ex:small_dcp}}
    \label{fig:small_dcp}
\end{figure}

\begin{example}
    If the sequence $\underline\lambda = (\lambda)$ only consists of one element, then we have the index poset $\mathcal I = \set{[1]}$, the parabolic subgroup $Q = P_{[1]}$ and the defining chain poset
    \begin{align*}
        \DCP{} \cong \Wlambda \cong \ulWlambda \cong \set{\theta \in W/W_Q \mid \theta \leq \tau}.
    \end{align*}
    Hence we obtain the underlying poset of the Seshadri stratification on $X_\tau \subseteq \PP(V(\lambda)_\tau)$ constructed in~\cite{seshstratandschubvar}, which we mentioned in Section~\ref{sec:choices}.
\end{example}

We now examine the covering relations in $\DCP{}$. By definition, $\DCP{}$ is a graded poset and of the same length as $\Wlambda$, so every covering relation $(\theta, I) \succ (\theta, J)$ in $\DCP{}$ is also a covering relation in $\Wlambda$. Therefore $(\theta, I)$ covers $(\phi, J)$ in $\DCP{}$ if and only if these elements are of one of the following two forms:
\begin{itemize}
    \item $J = I$, $\theta > \phi$ is a covering relation in $W/W_Q$ and $\pi_{P_I}(\theta) > \pi_{P_I}(\phi)$;
    \item $J \subsetneq I$ is a covering relation in $\mathcal I$ and $\theta = \phi$ in $W/W_Q$.
\end{itemize}

\begin{remark}
    \label{rem:restriction_of_DCP}
    The defining chain poset is compatible with restriction: For every $(\sigma, I) \in \DCP{}$ the subposet
    \begin{align*}
        \DCP{}_{\preceq (\sigma, I)} = \set{(\theta, J) \in \DCP{} \mid (\theta, J) \preceq (\sigma, I)}
    \end{align*}
    is also a defining chain poset in the following sense. Let $m'$ be the number of elements in $I$ and let $\kappa: [m'] \to I$ be a bijection. We define the sequence $\underline\lambda' = (\lambda_{\kappa(1)}, \dots, \lambda_{\kappa(m')})$ and the index poset $\mathcal I' = \set{ \kappa^{-1}(J) \mid J \in \mathcal I, J \subseteq I }$. Then the map
    \begin{align*}
        \DCP{}_{\preceq (\sigma, I)} \to D(\underline\lambda', \pi_{Q_I}(\sigma)), \quad (\theta, J) \mapsto ( \pi_{Q_I}(\theta), \kappa^{-1}(J) )
    \end{align*}
    is well-defined and monotone, where $D(\underline\lambda', \pi_{Q_I}(\sigma))$ is the defining chain poset with respect to the index poset $\mathcal I'$. 
    Since $\theta$ is $Q_I$-minimal for each $(\theta, J) \in \DCP{}_{\preceq (\sigma, I)}$, the map is injective. It is also surjective and its inverse map is monotone, because every maximal $\pi_{Q_I}(\sigma)$-standard chain can be extended to a maximal $\tau$-standard chain by using a maximal chain from $(\tau, [m])$ to $(\sigma, I)$ in $\DCP{}$.
\end{remark}

\begin{lemma}
    \label{lem:dcp_elements}
    The following are equivalent for every $(\theta, I) \in \Wlambda$:
    \begin{enumerate}[label=(\roman{enumi})]
        \item \label{itm:dcp_elements_a} $(\theta, I) \in \DCP{}$;
        \item \label{itm:dcp_elements_b} $\theta$ is $Q_I$-minimal and there exists a $P_{\mathcal I}$-chain from $(\tau, [m])$ to $(\theta, I)$;
        \item \label{itm:dcp_elements_c} $\theta$ is $Q_I$-minimal and there exists a $P_{\mathcal I}$-chain from an element $(\phi, J) \succeq (\theta, I)$ in $\DCP{}$ to $(\theta, I)$. 
    \end{enumerate}
\end{lemma}
\begin{proof}
    The implication \ref{itm:dcp_elements_b} $\Rightarrow$ \ref{itm:dcp_elements_c} is obvious and \ref{itm:dcp_elements_a} $\Rightarrow$ \ref{itm:dcp_elements_b} follows from Theorem~\ref{thm:unique_def_chain}, since $(\tau, [m])$ is contained in every unique defining chain of a maximal $\tau$-standard chain in $\ulWlambda$. Now suppose, that $\theta$ is $Q_I$-minimal and there exists an element $(\phi, J) \in \DCP{}$ and a  $P_{\mathcal I}$-chain from $(\phi, J)$ to $(\theta, I)$. We choose a maximal chain $I_1 \subsetneq \dots \subsetneq I_s = I$ in $\mathcal I$ from a minimal element $I_1 \in \mathcal I$ to $I$. Since the element $\theta^\triangledown \coloneqq \min_Q \circ \pi_{Q_{I_{s-1}}}(\theta)$ is less or equal to $\theta$, there exists a chain $\theta = \theta_r > \dots > \theta_0 = \theta^\triangledown$ of covering relations in $W/W_Q$. Both $\theta$ and $\theta^\triangledown$ are $Q_I$-minimal and  $\pi_{Q_{I_{s-1}}}(\theta^\triangledown) = \pi_{Q_{I_{s-1}}}(\theta)$. Since $Q_I = Q_{I_{s-1}} \cap P_I$, this implies $\pi_{P_I}(\theta_k) > \pi_{P_I}(\theta_{k-1})$ for all $1 \leq k \leq r$. Hence we get the $P_{\mathcal I}$-chain $(\theta_r, I) > \dots > (\theta_0, I) > (\theta_0, J)$ in $\Wlambda$. Analogously, we can continue this procedure by constructing $P_{\mathcal I}$-chains from $(\min_Q \circ \pi_{Q_{I_k}}(\theta), I_k)$ to $(\min_Q \circ \pi_{Q_{I_{k-1}}}(\theta), I_{k-1})$ for all $k = s-1, \dots, 2$. The element $\min_Q \circ \pi_{Q_{I_1}}(\theta)$ is minimal w.\,r.\,t. $Q_{I_{1}} = P_{I_{1}}$, so there is a $P_{\mathcal I}$-chain in $\Wlambda$ from this element to $(\id W_Q, I_1)$. There also exists a $P_{\mathcal I}$-chain from $(\tau, [m])$ to $(\phi, J)$, as $(\phi, J) \in \DCP{}$. In total, we can now glue the chain from $(\tau, [m])$ to $(\phi, J)$ with the chain from $(\phi, J)$ to $(\theta, I)$ and all of our constructed chains, to obtain a $P_{\mathcal I}$-chain $\overline\theta$, which also is a maximal chain in $\Wlambda$. Its projection to $\ulWlambda$ is a maximal $\tau$-standard chain and $\overline\theta$ is its unique defining chain. Therefore, we have $(\theta, I) \in \DCP{}$.
\end{proof}

\begin{corollary}
    \label{cor:dcp_go_down}
    For all $J \subseteq I$ in $\mathcal I$ and $(\theta, I) \in \DCP{}$ the element $(\min_Q \circ \pi_{Q_J}(\theta), J)$ lies in $\DCP{}$ and is less or equal to $(\theta, I)$.
\end{corollary}
\begin{proof}
    Follows from the proof of Lemma~\ref{lem:dcp_elements}.
\end{proof}

Lemma~\ref{lem:dcp_elements} also gives an inductive procedure to compute the defining chain poset. For every $r = r(\tau) + m - 1, \dots, 0$ we construct the set $D_r$ of all elements in $\DCP{}$ of rank $r$, starting with the largest rank, where we clearly have $D_r = \set{(\tau, [m])}$. If $D_r$ is known for some $r > 0$, then $D_{r-1}$ is the union of the sets
\begin{align*}
D_{r-1}(\theta, I) = \, &\set{(\theta, J) \mid \text{$\theta$ is $Q_J$-minimal and $I$ covers $J$}} \, \cup \\
&\set{(\phi, I) \mid \text{$\phi$ is $Q_I$-minimal, $\theta$ covers $\phi$ in $W/W_Q$ and $\pi_{P_I}(\theta) > \pi_{P_I}(\phi)$}}
\end{align*}
over all $(\theta, I) \in D_r$. Using this procedure, we compute another example of a defining chain poset, drawn in Figure~\ref{fig:dcp_example}.

The time complexity of this inductive procedure, however, scales linearly with the number of covering relations in $\Wlambda$, which can get out of hand quickly. Fortunately, the computation can be significantly accelerated, when $\tau = w_0 W_Q$ is the unique maximal element in $W/W_Q$. In this case the defining chain poset can be computed directly.

\begin{figure}
    \begin{center}
        \begin{tikzpicture}[scale=0.8]
        \node (p43) at (0,0) {$(4321, [3])$};
        \node (p42) at (0,-1.2) {$(4231, [3])$};
        \node (p41) at (2.7,-2.4) {$(4132, [3])$};
        \node (p32) at (0,-2.4) {$(3241, [3])$};
        \node (p31) at (2.7,-3.6) {$(3142, [3])$};
        \node (p21) at (2.7,-4.8) {$(2143, [3])$};
        \node (p423) at (-2.7,-2.4) {$(4231, [2])$};
        \node (p413) at (-2.7,-3.6) {$(4132, [2])$};
        \node (p412) at (-2.7,-4.8) {$(4123, [2])$};
        \node (p312) at (0,-6.0) {$(3124, [2])$};
        \node (p324) at (0,-3.6) {$(3241, [2])$};
        \node (p314) at (0,-4.8) {$(3142, [2])$};
        \node (p214) at (2.7,-6.0) {$(2143, [2])$};
        \node (p213) at (2.7,-7.2) {$(2134, [2])$};
        \node (p4) at (-2.7,-6.0) {$(4123, [1])$};
        \node (p3) at (0,-7.2) {$(3124, [1])$};
        \node (p2) at (0,-8.4) {$(2134, [1])$};
        \node (p1) at (0,-9.6) {$(1234, [1])$};
        
        \draw [thick, shorten <=-2pt, shorten >=-2pt, -stealth] (p43) -- (p42);
        \draw [thick, shorten <=-2pt, shorten >=-2pt, -stealth] (p42) -- (p41);
        \draw [thick, shorten <=-2pt, shorten >=-2pt, -stealth] (p42) -- (p32);
        \draw [thick, shorten <=-2pt, shorten >=-2pt, -stealth] (p41) -- (p31);
        \draw [thick, shorten <=-2pt, shorten >=-2pt, -stealth] (p32) -- (p31);
        \draw [thick, shorten <=-2pt, shorten >=-2pt, -stealth] (p31) -- (p21);
        \draw [thick, shorten <=-2pt, shorten >=-2pt, -stealth] (p42) -- (p423);
        \draw [thick, shorten <=-2pt, shorten >=-2pt, -stealth] (p423) -- (p413);
        \draw [thick, shorten <=-2pt, shorten >=-2pt, -stealth] (p413) -- (p412);
        \draw [thick, shorten <=-2pt, shorten >=-2pt, -stealth] (p412) -- (p312);
        \draw [thick, shorten <=-2pt, shorten >=-2pt, -stealth] (p32) -- (p324);
        \draw [thick, shorten <=-2pt, shorten >=-2pt, -stealth] (p324) -- (p314);
        \draw [thick, shorten <=-2pt, shorten >=-2pt, -stealth] (p314) -- (p214);
        \draw [thick, shorten <=-2pt, shorten >=-2pt, -stealth] (p214) -- (p213);
        \draw [thick, shorten <=-2pt, shorten >=-2pt, -stealth] (p412) -- (p4);
        \draw [thick, shorten <=-2pt, shorten >=-2pt, -stealth] (p4) -- (p3);
        \draw [thick, shorten <=-2pt, shorten >=-2pt, -stealth] (p3) -- (p2);
        \draw [thick, shorten <=-2pt, shorten >=-2pt, -stealth] (p2) -- (p1);
        \draw [thick, shorten <=-2pt, shorten >=-2pt, -stealth] (p21) -- (p214);
        \draw [thick, shorten <=-2pt, shorten >=-2pt, -stealth] (p31) -- (p314);
        \draw [thick, shorten <=-2pt, shorten >=-2pt, -stealth] (p314) -- (p312);
        \draw [thick, shorten <=-2pt, shorten >=-2pt, -stealth] (p312) -- (p3);
        \draw [thick, shorten <=-2pt, shorten >=-2pt, -stealth] (p213) -- (p2);
        \draw [thick, shorten <=-2pt, shorten >=-2pt, -stealth] (p41) -- (p413);
        \end{tikzpicture}
    \end{center}
    \caption{$D((\omega_1, \omega_3, \omega_2), w_0)$ in type $\texttt{A}_3$ for $\mathcal I = \set{[1], [2], [3]}$}
    \label{fig:dcp_example}
\end{figure}
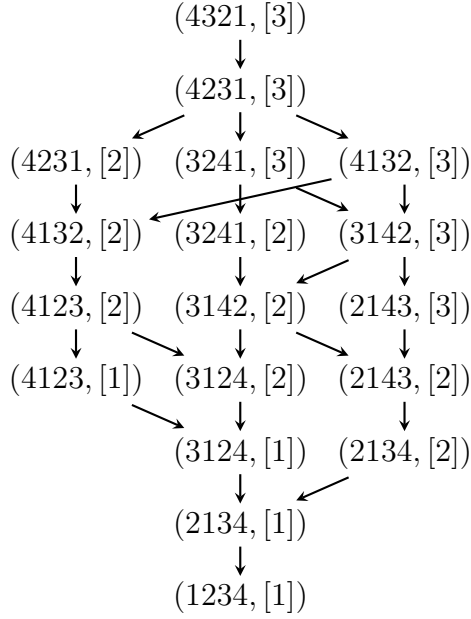

\begin{proposition}
    \label{prop:min_max_crit}
    Suppose $\tau = w_0 W_Q$ is the unique maximal element in $W/W_Q$. Then an element $(\theta, I) \in \Wlambda$ lies in $\DCP{}$, if and only if $\theta$ is $Q_I$-minimal and there exists a chain $I = I_r \subsetneq \dots \subsetneq I_m = [m]$ of covering relations in $\mathcal I$, such that $\theta$ is maximal w.\,r.\,t. the parabolic subgroup $Q^r = \bigcap_{j=r}^m P_{I_j}$ from Lemma~\ref{lem:operations_on_def_chains} (since $Q_\tau = Q$).
\end{proposition}
\begin{proof}
    Let $\theta \in W/W_Q$ be $Q_I$-minimal and  $I = I_r \subsetneq \dots \subsetneq I_m = [m]$ be a chain of covering relations in $\mathcal I$ with the above property. We define the parabolic subgroups $Q^k = \bigcap_{j=k} P_{I_j}$ for $k = r, \dots, m$. If $k < m$ we have $\max_Q \circ \pi_{Q^{k+1}}(\theta) \geq \max_Q \circ \pi_{Q^k}(\theta)$ in $W/W_Q$ and by a proof, which is completely analogous to parts of the proof of Lemma~\ref{lem:dcp_elements}, we can construct a $P_{\mathcal I}$-chain between $(\max_Q \circ \pi_{Q^{k+1}}(\theta), I_{k+1})$ and $(\max_Q \circ \pi_{Q^k}(\theta), I_k)$. Since $Q^m = P_{[m]}$ there also is a $P_{\mathcal I}$-chain between $(\tau, [m])$ and $(\max_Q \circ \pi_{Q^m}(\theta), [m])$. Hence $(\theta, I)$ lies in $\DCP{}$ by Lemma~\ref{lem:dcp_elements}. The other implication follows from Theorem~\ref{thm:unique_def_chain}.
\end{proof}

\begin{remark}
    If $(\theta, I)$ is an element of $\DCP{}$, then the $Q^r$-maximality of $\theta$ even holds for every chain of covering relations in $\mathcal I$. This follows later from Corollary~\ref{cor:nice_I}. Hence the choice of the this chain does not affect the statement of Proposition~\ref{prop:min_max_crit}: If one such chain exists, then $\theta$ is $Q^r$-maximal for any chain.
\end{remark}

Although the defining chain poset can be defined in this full generality, it is not always a reasonable candidate for the underlying poset $A$ of a Seshadri stratification on $X_\tau$. The extremal function of an element $(\theta, I) \in \DCP{}$, we wish to use, is a generalization of the Pl\"ucker coordinates in type \texttt{A} and it only depends on the image $\rho(\theta, I) \in \ulWlambda$, where $\rho$ denotes the (monotone) composition $\DCP{} \hookrightarrow \Wlambda \twoheadrightarrow \ulWlambda$. In combination with the condition~\ref{itm:seshadri_strat_b} on a Seshadri stratification, this forces us to only consider those defining chain posets, where $\rho$ is injective, such that no two elements in $\DCP{}$ have the same extremal function.

\begin{definition}
    We say the poset $\mathcal I$ is \textbf{\boldmath \nice{\tau}}, if the monotone map 
    \begin{align*}
        \rho: \DCP{} \to \ulWlambda, \quad (\theta, I) \mapsto (\pi_{P_I}(\theta), I)
    \end{align*}
    is an isomorphism of posets.
\end{definition}

The map $\rho: \DCP{} \twoheadrightarrow \ulWlambda$ is automatically an isomorphism, if it is injective. Indeed, if $(\theta, I) > (\phi, J)$ in $\ulWlambda$, then this is a $\tau$-standard chain, which we can therefore extend to a maximal $\tau$-standard chain. Its unique defining chain contains the preimages of $(\theta, I)$, and $(\phi, J)$ under $\rho$, because of the bijectivity of $\rho$ (it is always surjective), hence these preimages are comparable in $\DCP{}$.

There always exists at least one \nice{\tau} poset $\mathcal I$, namely $\mathcal I = \mathcal P(\set{1, \dots, m}) \setminus \set{\varnothing}$. Here the map $\rho$ is injective, since $P_I = Q_I$ holds for every $I \in \mathcal I$.

\begin{proposition}
    \label{prop:nice_I_vs_weakly_standard}
    The poset $\mathcal I$ is \nice{\tau}, if and only if every weakly $\tau$-standard LS-tableau of type $(\underline\lambda, \mathcal I)$ is $\tau$-standard. In this case the relation~(\ref{eq:quotient_relation}) is transitive.
\end{proposition}
\begin{proof}
    The notions of weakly $\tau$-standard and $\tau$-standard LS-tableaux coincide, if and only if every chain in $\ulWlambda$ is $\tau$-standard. This follows from the equivalence~(\ref{eq:tableau_chain}) and the fact that each element $(\theta, I) \in \ulWlambda$ defines an LS-path in $\mathbb{B}(\lambda_I)$, namely the straight-line path from the origin to $\theta(\lambda_I) \in \Lambda^+$.
    
    If $\mathcal I$ is \nice{\tau} and $\underline\theta: (\theta_\ell, I_\ell) > \dots > (\theta_0, I_0)$ is a chain in $\ulWlambda$, then its unique preimage via $\rho$ is a defining chain for $\underline\theta$ as $\DCP{} \cong \ulWlambda$. Additionally the relation~(\ref{eq:quotient_relation}) is transitive by Lemma~\ref{lem:relation_ulWlambda}, since chains in $\ulWlambda$ can be lifted to $W/W_Q$ via $\rho$. Conversely, if $\mathcal I$ is not \nice{\tau}, then there are two different preimages $(\overline\theta, I), (\overline{\theta'}, I) \in \DCP{}$ of an element $(\theta, I) \in \ulWlambda$, \wwlog the rank of $(\overline\theta, I)$ in $\DCP{}$ is less or equal to the rank of $(\overline{\theta'}, I)$. We choose chains of covering relations 
    \begin{align*}
    (\tau, [m]) = (\sigma_r, I_r) \succ \dots \succ (\sigma_{j+1}, I_{j+1}) \succ (\overline\theta, I)& \quad \text{and}\\
    (\overline{\theta'}, I) \succ (\sigma_{j-1}, I_{j-1}) \succ \dots \succ (\sigma_0, I_0)&
    \end{align*}
    in $\DCP{}$, where $(\sigma_0, I_0)$ is a minimal element. By projecting both chains to $\ulWlambda$ and gluing them together at their shared element, we get a chain $\underline\theta$ in $\ulWlambda$ containing $(\theta, I)$. Its length is equal to the length of $\DCP{}$. In the case where $(\overline\theta, I)$ and $(\overline{\theta'}, I)$ have different ranks in $\DCP{}$, the chain $\underline\theta$ certainly is too long to be $\tau$-standard. 
    
    If the ranks are equal, suppose that $\underline\theta$ is $\tau$-standard. Then there exists an unique defining chain by Theorem~\ref{thm:unique_def_chain}. The beginning of this defining chain must agree with $(\sigma_r, I_r) \succ \dots \succ (\sigma_{j+1}, I_{j+1})$ and its end agrees with $(\sigma_{j-1}, I_{j-1}) \succ \dots \succ (\sigma_0, I_0)$. The element in between would be a lift of $(\theta, I)$ via $\rho$. It is equal to $(\overline\theta, I)$, as $\overline\theta$ is the unique maximal lift of $\theta$, which is less or equal to $\sigma_{j+1}$. Analogously this lift is equal to $(\overline{\theta'}, I)$, which is impossible.
\end{proof}

\begin{example}
    Even in type \texttt{A} there are elements $\tau \in W/W_Q$, where no totally ordered, \nice{\tau} poset $\mathcal I$ exists. One of the easiest examples is $\tau = 3412$ for $G = \mathrm{SL}_4(\K)$ and $\underline\lambda = (\omega_1, \omega_2, \omega_3)$. The reason for this is the following: When $\mathcal I$ is totally ordered, then $P_{[3]}$ is equal to a maximal parabolic subgroup $P_i$ for $i \in [3]$. We can write $\tau$ in the form $\tau^{P_i} \tau_{P_i}$ for $\tau^{P_i} \in W^{P_i}$ and $\tau_{P_i} \in W_{P_i}$ (see Appendix~\ref{sec:weyl_groups}). Then the element $(\sigma \tau_{P_i}, [3])$ lies in the defining chain poset for every $\sigma \in W^{P_i}$ with $\sigma \leq \tau^{P_i}$. But for each $i \in [3]$ there is a covering relation $\sigma' < \tau$, such that $\pi_{P_i}(\sigma') < \pi_{P_i}(\tau)$ is not a covering relation: $1432 < 3412$ for $i = 1,2$ and $3214 < 3412$ for $i = 3$. Hence $(\pi_{P_i}(\sigma'), [3])$ has multiple preimages under $\rho$.

    Let $\mathcal I$ be a \nice{\tau} index poset for $\tau = 3412$. We show that there are only two possible choices for $\mathcal I$. Suppose that $P_{[3]} = P_1 \cap P_2$ or $P_{[3]} = P_2 \cap P_3$. In the first case the defining chain poset would contain the chains $(3412, [3]) \succ (1432, [3])$ and $(3412, [3]) \succ (2413, [3]) \succ (1423, [3])$. But $1432 = 1423$ in $W/W_{P_1 \cap P_2}$, so $\mathcal I$ is not \nice{\tau}. Similarly, in the second case we have the chains $(3412, [3]) \succ (1432, [3]) \succ (1342, [3])$ and $(3412, [3]) \succ (3142, [3])$ with $1342 = 3142$ in $W/W_{P_2 \cap P_3}$.

    Therefore the parabolic $P_{[3]}$ is either equal to $P_1 \cap P_3$ or to $B$. If $P_{[3]} = B$, then the requirement~(\ref{eq:need_for_s2}) implies that $\mathcal I$ is equal to the poset $\mathcal I = \mathcal P(\set{1,2,3}) \setminus \set{\varnothing}$. If $P_{[3]} = P_1 \cap P_3$, the index poset $\mathcal I$ contains $I = \set{1,2}$ and $J = \set{2,3}$. Then the following elements lie in the defining chain poset: $(3412, I)$, $(3214, [3])$, $(3214, I)$, $(3412, J)$, $(1432, [3])$ and $(1432, J)$. Since $3412 = 3214$ in $W/W_{P_1}$, the set $\underline I$ cannot be equal to $\set{1}$. As $I \nsubseteq J$, we have $2 \in \underline I$ by~(\ref{eq:need_for_s2}). Therefore $\underline I = I$. Analogously, one can show $\underline J = J$. Hence $\mathcal I$ is given by $\mathcal I = \set{\set{1}, \set{2}, \set{3}, \set{1,2}, \set{2,3}, [3]}$.
\end{example}

For a general element $\tau$ it is difficult to tell, which posets $\mathcal I$ are \nice{\tau}. However, if $\tau = w_0 W_Q$ is the unique maximal element in $W/W_Q$, then the injectivity of $\rho$ translates into the absence of certain paths in the Dynkin diagram of $G$. To state this criterion, we define the set $\Delta_{Q'} = \set{\alpha \in \Delta \mid s_\alpha \in W_{Q'}}$ of simple roots for every parabolic subgroup $Q' \subseteq G$.

\begin{theorem}
    \label{thm:nice_I}
    For $\tau = w_0 W_Q$ the poset $\mathcal I$ is \nice{\tau}, if and only if one of the following equivalent conditions holds for each $I \in \mathcal I$ and every chain $I = I_r \subsetneq \dots \subsetneq I_m = [m]$ of covering relations in $\mathcal I$. Here $Q^r = \bigcap_{j=r}^m P_{I_j}$.
    \begin{enumerate}[label=(\roman{enumi})]
        \item \label{itm:nice_I_a} The element $(\id W_{P_I}, I)$ has exactly one preimage via $\rho: \DCP{} \twoheadrightarrow \ulWlambda$;
        \item \label{itm:nice_I_b} $\min_Q \circ \max_{Q_I}(\id W_{P_I}) = \max_Q \circ \min_{Q^r}(\id W_{P_I})$;
        \item \label{itm:nice_I_c} $W_{P_I} \cap W^{Q_I} \subseteq W_{Q^r} \cap W^Q$.
        \item \label{itm:nice_I_d} The two parabolic subgroups $Q_I$ and $Q^r$ generate $P_I$ and every path in the Dynkin diagram of $G$ (not visiting the same vertex twice) connecting a vertex of $\Delta_{Q_I} \setminus \Delta_{Q^r}$ with a vertex of $\Delta_{Q^r} \setminus \Delta_{Q_I}$ contains a vertex not in $\Delta_{P_I}$.
    \end{enumerate}
\end{theorem}
\begin{proof}
    The poset $\mathcal I$ is \nice{\tau}, if and only if statement~\ref{itm:nice_I_a} is fulfilled for every $I \in \mathcal I$. We remark, that this equivalence also holds for every $\tau \neq w_0 W_Q$. Indeed if $\rho$ is not bijective, then there exist two different lifts $(\overline\theta, I), (\overline{\theta'}, I) \in \DCP{}$ of an element $(\theta, I) \in \ulWlambda$. We write $\min_B(\overline\theta) = \theta^{P_I} \overline\theta_{P_I}$ and $\min_B(\overline\theta) = \theta^{P_I} \overline{\theta'}_{P_I}$ for $\theta^{P_I} \in W^{P_I}$ and $\overline\theta_{P_I}, \overline{\theta'}_{P_I} \in W_{P_I}$. Since there are $P_{\mathcal I}$-chains from $(\overline\theta, I)$ to $(\overline\theta_{P_I} W_Q, I)$ and from $(\overline\theta, I)$ to $(\overline\theta_{P_I} W_Q, I)$, both $(\overline\theta_{P_I} W_Q, I)$ and $(\overline{\theta'}_{P_I} W_Q, I)$ are two different lifts of $(\id W_{P_I}, I)$ in $\DCP{}$ by Lemma~\ref{lem:dcp_elements}.
    
    Next we show the implication \ref{itm:nice_I_a} $\Rightarrow$ \ref{itm:nice_I_b}, by proving that both $\sigma = \min_Q \circ \max_{Q_I}(\theta)$ and $\sigma' = \max_Q \circ \min_{Q^r}(\theta)$ are $Q_I$-minimal and $Q^r$-maximal lifts of $(\theta, I)$, where we set $\theta = \id W_{P_I}$. It then follows $\sigma = \sigma'$ from Proposition~\ref{prop:min_max_crit}. We show this equality for any $\theta \in W/W_{P_I}$, because then we obtain the next corollary as a consequence. 
    
    The element $\sigma$ is $Q_I$-minimal by definition and maps to $(\max_{Q_I}(\theta), \pi_{Q^r}(\sigma))$ via the map $\iota: W/W_Q \to W/W_{Q_I} \times W/W_{Q^r}$. On the other hand, $\max_Q \circ \pi_{Q^r}(\sigma)$ maps to $(\overline\theta, \pi_{Q^r}(\sigma))$ for some lift $\overline\theta \in W/W_{Q_I}$ of $\theta$. We clearly have $\max_{Q_I}(\theta) \geq \overline\theta$. As $Q_I \cap Q^r = Q_\tau = Q$, the map $\iota$ is an embedding and we can conclude $\sigma \geq \max_Q \circ \pi_{Q^r}(\sigma)$. In particular, $\sigma$ is $Q^r$-maximal. Analogously, one can show the $Q_I$-minimality of $\sigma'$.
    
    Part~\ref{itm:nice_I_c} follows from \ref{itm:nice_I_b}: Since $W_{P_I} \cap W^{Q_I} \subseteq W^Q$, we only need to prove the inclusion $W_{P_I} \cap W^{Q_I} \subseteq W_{Q^r}$. Every element $\phi \in W_{P_I} \cap W^{Q_I}$ is smaller or equal to $\sigma \coloneqq \min_B \circ \max_{Q_I}(\id W_{P_I})$ since both are $Q_I$-minimal and $\phi W_{Q_I} \leq \max_{Q_I}(\id W_{P_I}) = \sigma W_{Q_I}$. By statement~\ref{itm:nice_I_b} we now have
    \begin{align*}
        \pi_{Q^r}(\phi) \leq \pi_{Q^r}(\sigma) = \pi_{Q^r} \circ \operatorname{min}_B \circ \operatorname{max}_Q \circ \operatorname{min}_{Q^r}(\id W_{P_I}) = \operatorname{min}_{Q^r}(\id W_{P_I}) = \id W_{Q^r}
    \end{align*}
    and this inequality is equivalent to $\phi \in W_{Q^r}$.
    
    We close the first circle of implications via \ref{itm:nice_I_c} $\Rightarrow$ \ref{itm:nice_I_a}: Let $(\sigma, I) \in \DCP{}$ be any preimage of $(\id W_{P_I}, I)$ via $\rho$. Then by Proposition~\ref{prop:min_max_crit} there exists a chain $I = I_r \subsetneq \dots \subsetneq I_m = [m]$ of covering relations in $\mathcal I$, such that $\sigma$ is $Q_I$-minimal and $Q^r$-maximal (w.\,r.\,t. this chain). By assumption we have $W_{P_I} \cap W^{Q_I} \subseteq W_{Q^r} \cap W^Q$. Notice, that the other inclusion $W_{Q^r} \cap W^Q \subseteq W_{P_I} \cap W^{Q_I}$ is always fulfilled, even if $\mathcal I$ is not \nice{\tau}: Here $W_{Q^r} \cap W^Q \subseteq W_{P_I}$ follows from $Q^r \subseteq P_I$ and the $Q_I$-minimality of every element in $Q_{Q^r} \cap W^Q$ can again be shown via the embedding $W/W_Q \to W/W_{Q_I} \times W/W_{Q^r}$. 
    
    By the bijection in~(\ref{eq:product_map_weyl}) the set $W^Q$ can be decomposed into the product
    \begin{align*}
    W^Q = W^{Q^r} \cdot (W_{Q^r} \cap W^Q) = W^{P_I} \cdot (W_{P_I} \cap W^{Q^r}) \cdot (W_{Q^r} \cap W^Q).
    \end{align*}
    The element $\min_B(\sigma)$ is contained in $W_{P_I} \cap W^{Q_I}$ and can therefore be (uniquely!) written in the form $\min_B(\sigma) = \id \cdot \theta \cdot \phi$ for $\theta \in W_{P_I} \cap W^{Q^r}$ and $\phi \in W_{Q^r} \cap W^Q$. As $\min_B(\sigma)$ is $Q^r$-maximal, $\phi$ is equal to the unique maximal element in the poset $W_{Q^r} \cap W^Q = W_{P_I} \cap W^{Q_I}$. But since $\min_B(\sigma)$ is also an element of $W^{Q_I} = W^{P_I} \cdot (W_{P_I} \cap W^{Q_I})$, it follows $\theta = \id$. Therefore $\sigma$ is uniquely determined: It is the maximal element of $W_{P_I} \cap W^{Q_I}$ and this is independent of the choice of the chain $I = I_r \subsetneq \dots \subsetneq I_m = [m]$.
    
    Next we show \ref{itm:nice_I_c} $\Rightarrow$ \ref{itm:nice_I_d}. For each simple root $\alpha \in \Delta_{P_I}$ not contained in $\Delta_{Q_I}$ or $\Delta_{Q^r}$, $s_\alpha$ is an element of $W_{P_I} \cap W^{Q_I}$, but does not lie in $W_{Q^r}$. Hence $Q_I$ and $Q^r$ generate $P_I$.
    
    Now suppose that there exists a path $\alpha_1 \to \dots \to \alpha_k$ in the Dynkin diagram of $G$, such that $\alpha_1, \dots, \alpha_k \in \Delta_{P_I}$, $\alpha_1 \in \Delta_{Q_I} \setminus \Delta_{Q^r}$ and $\alpha_k \in \Delta_{Q^r} \setminus \Delta_{Q_I}$. Then $\sigma = s_{\alpha_1} \cdots s_{\alpha_k}$ is an element of $W_{P_I}$ and we claim, that $\sigma$ also lies in $W^{Q_I}$. When we have shown this $Q_I$-minimality, it then follows $\sigma \in W_{P_I} \cap W^{Q_I}$. But $\sigma \notin W_{Q^r}$, since $\alpha_1 \notin \Delta_{Q^r}$.

    The $Q_I$-minimality of $\sigma$ is equivalent to $\sigma s_{\beta} > \sigma$ for all $\beta \in \Delta_{Q_I}$. First, let $\beta \in \Delta_{Q_I}$ be a simple root, which in not contained in our chosen path. In this case $s_{\alpha_1} \cdots s_{\alpha_k} s_\beta$ is in reduced decomposition by Lemma~\ref{lem:red_decomp}, so $\sigma s_\beta > \sigma$. Now let $\beta = \alpha_i \in \Delta_{Q_I}$ for an index $1 \leq i \leq k$. Suppose that $s_{\alpha_1} \cdots s_{\alpha_k} s_\beta$ is not in reduced decomposition. Then there exists an index $j$, such that $\sigma s_\beta = s_{\alpha_1} \cdots \hat s_{\alpha_j} \cdots s_{\alpha_k}$ ($s_{\alpha_j}$ is omitted). Thus $\sigma = s_{\alpha_1} \cdots \hat s_{\alpha_j} \cdots s_{\alpha_k} s_\beta$. The simple reflections occuring in a reduced decomposition are always uniquely determined (\textit{not} counting with multiplicity), hence we have $\beta = \alpha_j$ and $s_{\alpha_1} \cdots s_{\alpha_k} = s_{\alpha_1} \cdots \hat s_{\alpha_j} \cdots s_{\alpha_k} s_{\alpha_j}$. Hence $s_{\alpha_j}$ commutes with $s_{\alpha_{j+1}} \cdots s_{\alpha_k}$. The simple reflection $s_{\alpha_j}$ commutes with all $s_{\alpha_k}$ for $k > j+1$, since Dynkin diagrams contain no cycles and we have a path from $\alpha_j$ to $\alpha_k$. Therefore $s_{\alpha_j}$ and $s_{\alpha_{j+1}}$ must commute, but this contradicts to edge between $\alpha_j$ and $\alpha_{j+1}$ in the Dynkin diagram. The decomposition $\sigma s_\beta = s_{\alpha_1} \cdots s_{\alpha_k} s_\beta$ thus is reduced and $\sigma s_\beta > \sigma$.
    
    Finally, \ref{itm:nice_I_d} implies \ref{itm:nice_I_c}. Suppose that $\sigma$ is an element in $W_{P_I} \cap W^{Q_I}$, but not in $W_{Q^r} \cap W^Q$. Since $W^{Q_I} \subseteq W^Q$, we thus have $\sigma \notin W_{Q^r}$. To each reduced decomposition $\sigma = s_{\alpha_1} \cdots s_{\alpha_k}$ we now associate a pair $(p,q)$ of natural numbers in the following way. Since $Q_I$ and $Q^r$ generate $P_I$, all simple roots $\alpha_1, \dots, \alpha_k$ lie in $\Delta_{Q_I}$ or $\Delta_{Q^r}$. As $\sigma \notin W_{Q^r}$, there exists a maximal index $1 \leq p \leq k$ with $\alpha_p \in \Delta_{Q_I} \setminus \Delta_{Q^r}$. The simple root $\alpha_k$ is not contained in $\Delta_{Q_I}$ by the $Q_I$-minimality of $\sigma$. Hence there exists a minimal number $q \in \set{p, \dots, k}$ with $\alpha_q \in \Delta_{Q^r} \setminus \Delta_{Q_I}$. We can assume, that $\sigma = s_{\alpha_1} \cdots s_{\alpha_k}$ is a reduced decomposition, such that the associated pair $(p,q)$ is maximal \wrt{} the total order
    \begin{align*}
        (p,q) \geq (p',q') \quad \Longleftrightarrow \quad p > p' \mkern8mu \text{or} \mkern8mu(p = p' \mkern8mu \text{and} \mkern8mu q-p \leq q'-p')
    \end{align*}
    on $\Z \times \Z$. We now partition the set $J = \set{p, \dots, q}$ by fixing numbers $p = p_0 < p_1 < \dots < p_t \leq q$, such that for all $p+1 \leq j \leq q$ the simple roots $\alpha_{j-1}$ and $\alpha_{j}$ are disconnected in the Dynkin diagram if and only if there exists a non-zero index $i \in \set{1, \dots, t}$ with $j = p_i$. Set $p_{t+1} \coloneqq q+1$. For all $i = 0, \dots, t$ we define the set $J_i = \set{j \in J \mid p_i \leq j < p_{i+1}}$, the associated set $\Delta_i = \set{\alpha_j \mid j \in J_i}$ of simple roots and the subword $\sigma_i = \prod_{j \in J_i} s_{\alpha_j}$, where the product is taken is ascending order. Let $s \in \set{0, \dots, t}$ be the smallest integer, such that the union $\bigcup_{i=s}^t \Delta_i$ is connected in the Dynkin diagram. 
    
    By the assumption~\ref{itm:nice_I_d}, every path connecting $\alpha_p$ and $\alpha_q$ contains a simple root $\beta \notin \Delta_{P_I}$. Therefore the number $s$ is at least $1$ and by the definition of $s$ we have $\sigma_{s-1} \sigma_s \cdots \sigma_t = \sigma_s \cdots \sigma_t \sigma_{s-1}$. Hence, for $s \geq 2$, the number $q-p$ was not minimal and for $s = 1$ the number $p$ was not maximal, which contradicts our choice of the reduced decomposition of $\sigma$.
\end{proof}

\begin{remark}
    For $Q = B$ the condition~\ref{itm:nice_I_d} simplifies as follows: The two parabolic subgroups $Q_I$ and $Q^r$ generate $P_I$ and there is no edge in the Dynkin diagram of $G$ connecting the two subsets $\Delta_{Q_I}$ and $\Delta_{Q^r}$.
\end{remark}

\begin{corollary}
    \label{cor:nice_I}
    Let $\mathcal I$ be a \nice{\tau} poset for $\tau = w_0 W_Q$. Then the map
    \begin{align*}
        \ulWlambda \to \DCP{}, \quad (\theta, I) \mapsto (\operatorname{min}_Q \circ \operatorname{max}_{Q_I}(\theta), I)
    \end{align*}
    is the inverse map of $\rho$. In particular, it is an isomorphism of posets.
\end{corollary}
\begin{proof}
    Let $(\theta, I)$ be any element in $\ulWlambda$. We fix a chain $I = I_r \subsetneq \dots \subsetneq I_m = [m]$ of covering relations in $\mathcal I$. As we remarked in the proof of Theorem~\ref{thm:nice_I}, the lift $\widetilde\theta = \min_Q \circ \max_{Q_I}(\theta)$ of $\theta$ is $Q_I$-minimal and $Q^r$-maximal, where $Q^r = P_{I_r} \cap \dots \cap P_{I_m}$. Hence $(\widetilde\theta, I)$ lies in the defining chain poset by Proposition~\ref{prop:min_max_crit} and thus it is the preimage of $(\theta, I)$ under $\rho$.
\end{proof}

\section{Non \nice{\tau} posets}

We prove in Theorem~\ref{thm:standard_monomial_basis} that \nice{\tau} posets induce a stratification of LS-type. Among other benefits, such a stratification has the advantage that its Newton-Okounkov polytopal complexes consist of products of simplices and are therefore compatible with the multiprojective structure (see~\cite[Section 10]{ownarticle}). In particular, there exists a purely combinatorial formula for the leading term of the multivariate Hilbert polynomial. If the index poset is totally ordered, then one can easily extract the coefficients of all the occurring monomials. Up to multiplying by a product of factorials, these coefficients are equal to the multidegrees of the stratified variety. 

Even though the index posets which are not \nice{\tau} are incompatible with our notion of a Seshadri stratification, there still seems to exist hidden geometry behind the defining chain poset. 

\begin{conjecture}
    Let $\mathcal I$ be any totally-ordered index poset satisfying condition~(\ref{eq:need_for_s2}). By permuting the dominant weights in the sequence $\underline\lambda$ we can assume that $\mathcal I$ consists of the sets $[i]$ for $i \in [m]$. Then the multidegree of the Schubert variety $X_\tau \subseteq \prod_{i=1}^m \PP(V(\lambda_i)_{\tau_i})$ to a tuple $\underline k \in \N_0^m$ with $k_1 + \dots + k_m = \dim X_\tau$ is given by
    \begin{align*}
        \deg_{\underline k}(X_\tau) = \sum_{\mathfrak C} \mkern2mu b_{\mathfrak C} \mkern1mu,
    \end{align*}
    where the sum runs over all maximal chains $\mathfrak C$ in $\DCP{}$, which contain exactly $k_i + 1$ elements from $W/W_{P_i} \subseteq \DCP{}$ for each $i = 1, \dots, m$ and $b_{\mathfrak C}$ denotes the product of all bonds in $\mathfrak C$.
\end{conjecture}

A non-trivial example of the above behaviour is the poset from Figure~\ref{fig:dcp_example}. The reader may check, that all bonds are equal to $1$ and -- even though this poset is clearly non-isomorphic to the poset from Figure~\ref{fig:hasse_a_3} -- the conjecture provides the same multidegrees of $\mathrm{SL}_4(\K)/B$, which the author computed in~\cite[Example 11.12]{ownarticle}. 

If the conjecture is true, it might be possible that there is simply no geometric explanation behind it. But maybe there exists a more general definition of a Seshadri stratification, which allows the use of the same or similar extremal functions for different strata. A possible approach would be to demand additional structure in the definition of a stratification, namely a surjective monotone map $\rho: A \to \underline A$. In case of the defining chain poset $A = \DCP{}$ one chooses $\underline A = \ulWlambda$. Index sets, strata and extremal functions would still be indexed by the elements of $A$ and the condition~\ref{itm:seshadri_strat_b} would be replaced by the following: The function $f_q$ vanishes on $\hat X_p$, if $\rho(q) \nleq \rho(p)$. The quasi-valuation can then be defined as a map $\mathcal V: \K[X] \setminus \set{0} \to \Q^A \twoheadrightarrow \Q^{\underline A}$. The author presumes that this still leads to a quasi-valuation with non-negative entries, such that its image is a union of finitely generated monoids and with $\mathcal V(f_p) = e_{\rho(p)}$ for all $p \in A$. As this endeavour requires re-checking many of the arguments in~\cite{seshstrat}, we only consider \nice{\tau} posets.

\section{Sequences of fundamental weights}
\label{subsec:poset_examples}

In this section we consider the special case where $\tau = w_0 W_Q$ is the unique maximal element in $W/W_Q$ and the sequence $\underline\lambda = (k_1 \omega_1, \dots, k_m \omega_m)$ is defined by pairwise distinct fundamental weights $\omega_1, \dots, \omega_m$ of $G$ and natural numbers $k_1, \dots, k_m \in \N$. For a fixed parabolic subgroup $Q \subseteq G$ one can always choose a sequence consisting of the fundamental weights $\omega$ with $\langle \omega, \alpha^\vee \rangle = 0$ for all $\alpha \in \Delta_Q$. In practice, one would most likely choose $k_1 = \dots = k_m = 1$, so that the LS-tableaux of type $(\underline\lambda, \mathcal I)$ give rise to a character formula for the irreducible representation $V(\mu)$ to every dominant weight $\mu$ of the parabolic subgroup $Q$, \ie for each $\mu \in \N_0 \mkern1mu \omega_1 + \dots + \N_0 \mkern1mu \omega_m$ (see Remark~\ref{rem:shape_to_degree}).

Let $\alpha_i \in \Delta$ denote the simple root with $\langle \omega_i, \alpha_i^\vee \rangle = 1$ for all $i \in [m]$. Then the following criterion is just a reformulation of Theorem~\ref{thm:nice_I}\,\ref{itm:nice_I_d}.

\begin{corollary}
    \label{cor:nice_I_fundamental_weights}
    The poset $\mathcal I$ is \nice{\tau}, if and only if the following two conditions hold for every $I \in \mathcal I$ and each chain $I = I_r \subsetneq \dots \subsetneq I_m = [m]$ of covering relations in $\mathcal I$:
    \begin{enumerate}[label=(\roman{enumi})]
        \item The intersection of the two sets $I$ and $I' = \bigcup_{j=r}^m \underline I_j$ is equal to $\underline I$.
        \item All paths from $\set{\alpha_i \mid i \in I \setminus \underline I}$ to $\set{\alpha_i \mid i \in I' \setminus \underline I}$ contain a vertex $\alpha_i$ for $i \in \underline I$.
    \end{enumerate}
\end{corollary}

It is of course computationally advantageous when the index poset $\mathcal I$ is totally ordered, as the defining chain poset consists of much less elements in these cases. However, there might not exist a poset $\mathcal I$, which is \nice{\tau} and totally ordered at the same time.

\begin{corollary}
    \label{cor:totally_ordered_nice_I}
    There exists a \nice{\tau}, totally ordered index poset $\mathcal I$, if and only if there is a path in the Dynkin diagram of $G$ containing all vertices from the set $\Delta \setminus \Delta_Q = \set{\alpha_1, \dots, \alpha_m}$.
\end{corollary}
\begin{proof}
    First we consider the index poset of the form $\mathcal I = \set{[i] \mid i \in [m]}$. Then the first condition in Corollary~\ref{cor:nice_I_fundamental_weights} is automatically fulfilled for each $I = [i] \in \mathcal I$ as $I' = \set{i, \dots, m}$. Hence the Corollary implies:
    \begin{align}
        \label{eq:equiv_nice_I}
        \text{$\mathcal I$ is \nice{\tau}} \quad \Longleftrightarrow \quad &\text{for each $i \in [m]$ all paths from $\set{\alpha_1, \dots, \alpha_{i-1}}$} \nonumber \\
        & \text{to $\set{\alpha_{i+1}, \dots, \alpha_m}$ contain $\alpha_i$}.
    \end{align}

    Let $\mathcal I = \set{[i] \mid i \in [m]}$ be totally ordered and \nice{\tau}. \WWlog{} it consists of all sets $[i]$ for $i \in [m]$. Then there exists a path containing $\set{\alpha_1, \dots, \alpha_m}$: It is given by the unique path $\pi$ from $\alpha_1$ to $\alpha_m$. Dynkin diagrams of simple algebraic groups are connected and contain no cycles, hence $\pi$ is unique. It also contains the roots $\alpha_2, \dots, \alpha_{m-1}$ by (\ref{eq:equiv_nice_I}). 
    
    Conversely, let us assume there is a path $\pi$ containing $\alpha_1, \dots, \alpha_m$. We can rearrange the indices, such that this is also the order in which $\pi$ visits these simple roots. Then the poset $\mathcal I = \set{[i] \mid i \in [m]}$ is \nice{\tau}, by the equivalence in~(\ref{eq:equiv_nice_I}).
\end{proof}

For every flag variety $G/Q$ in each Dynkin type we now give an example of a \nice{\tau} poset $\mathcal I$, that is as small as possible. We already covered the cases, where the simple roots in $\Delta \setminus \Delta_Q$ are all contained in one path in the Dynkin diagram. Here we can choose a totally ordered poset $\mathcal I$. This is always possible in the types \texttt{A}, \texttt{B}, \texttt{C}, \texttt{F} and \texttt{G}. In particular, for the sequence $\underline\lambda = (\omega_{k_1}, \dots, \omega_{k_m})$ we used in Example~\ref{ex:strat_type_A} the poset $\mathcal I = \set{ \set{i, \dots, m} \mid i \in [m]}$ is \nice{w_0 W_Q}. 

\begin{remark}
    In combination with Proposition~\ref{prop:min_max_crit} this also proves \cite[Remark 11.14]{ownarticle}. It states that in the stratification from Example~\ref{ex:strat_type_A} the set of lifts
    \begin{align*}
        \Set{\mkern2mu \operatorname{min}_Q \circ \operatorname{max}_{Q_i}(\theta) \in W/W_Q \mid (\theta, i) \in \ulW \mkern2mu }
    \end{align*}
    to a fixed index $i \in [m]$ coincides with the set of all elements in $W/W_Q$, which are $Q_i$-minimal and $Q^i$-maximal for the parabolic subgroup $Q^i = \bigcap_{j=1}^i P_{k_i}$.
\end{remark}

Next we look at the remaining cases in type \texttt{D}, where we cannot choose a totally ordered poset $\mathcal I$. The set of the simple roots $\alpha_1, \dots, \alpha_m$ not belonging to $\Delta_Q$ therefore contains both degree one vertices (\wwlog{} they are $\alpha_1$ and $\alpha_2$), which are adjacent to the degree $3$ vertex. We assume that the other roots $\alpha_3, \dots \alpha_m$ are ordered along the line beginning with the degree $3$ vertex. For example, we may order the roots as depicted:
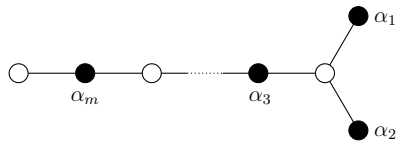
\begin{center}
    \begin{tikzpicture}[baseline]
        \dynkin[scale=2.5,labels={,\alpha_m,,\, \alpha_{3},,\, \alpha_1,\, \alpha_2},label directions={,,,,,,}] D{o*o...*o**}
    \end{tikzpicture}
\end{center}
Then the following poset satisfies the condition~(\ref{eq:need_for_s2}) and is \nice{\tau}:
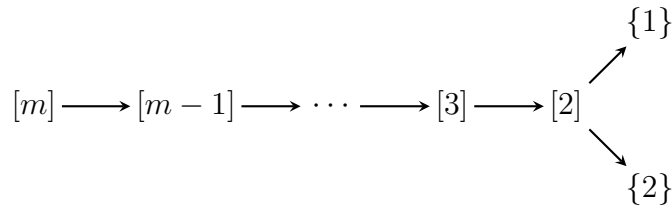
\begin{center}
    \begin{tikzpicture}[baseline]
        \node (1) at (3.5,0) {$[m]$};
        \node (2) at (5.5,0) {$[m-1]$};
        \node (dots) at (7.4,0) {$\,\cdots$};
        \node (3last) at (9,0) {$[3]$};
        \node (2last) at (10.5,0) {$[2]$};
        \node (m-1) at (11.6,1.1) {$\set{1}$};
        \node (m) at (11.6,-1.1) {$\set{2}$};
        
        \draw [thick, shorten <=-2pt, shorten >=-2pt, -stealth] (1) -- (2);
        \draw [thick, shorten <=-2pt, shorten >=-2pt, -stealth] (2) -- (dots);
        \draw [thick, shorten <=-2pt, shorten >=-2pt, -stealth] (dots) -- (3last);
        \draw [thick, shorten <=-2pt, shorten >=-2pt, -stealth] (3last) -- (2last);
        \draw [thick, shorten <=-2pt, shorten >=-2pt, -stealth] (2last) -- (m-1);
        \draw [thick, shorten <=-2pt, shorten >=-2pt, -stealth] (2last) -- (m);
    \end{tikzpicture}
\end{center}
In type \texttt{E} we can take the same index poset (with suitable numbering of the roots $\alpha_1, \dots, \alpha_m$), if in the graph one obtains by erasing the vertex $\beta$ of degree $3$ from the Dynkin diagram there is at most one connected component, which contains two or more simple roots $\alpha_1, \dots, \alpha_m$. 

This only leaves us with the case, where two of these connected components contain two or more of our fixed simple roots. One of the two components then contains exactly two roots. We call them $\alpha_2$ and $\alpha_3$ where the latter is adjacent to the degree $3$ vertex $\beta$. Let $\alpha_1$ be the degree one vertex adjacent to $\beta$. We order the remaining simple roots along the line starting at $\beta$.
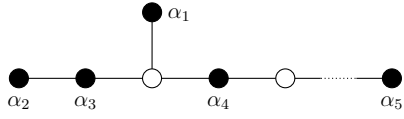
\begin{center}
    \begin{tikzpicture}[baseline]
        \dynkin[scale=2.5,labels={\alpha_{2},\, \alpha_{1},\alpha_{3},,\alpha_{4},,\alpha_5},label directions={,,,,,,}] E{***o*o...*}
    \end{tikzpicture}
\end{center}
To obtain a \nice{\tau} index poset, we again include all sets $[i]$ for $i \geq 3$ as well as $\set{1,2}$ and $\set{2,3}$. However, Corollary~\ref{cor:nice_I_fundamental_weights} and the requirement~(\ref{eq:need_for_s2}) force us to take all possible rank one elements $\set{1}, \set{2}$ and $\set{3}$. In total, the following index poset is \nice{\tau}:
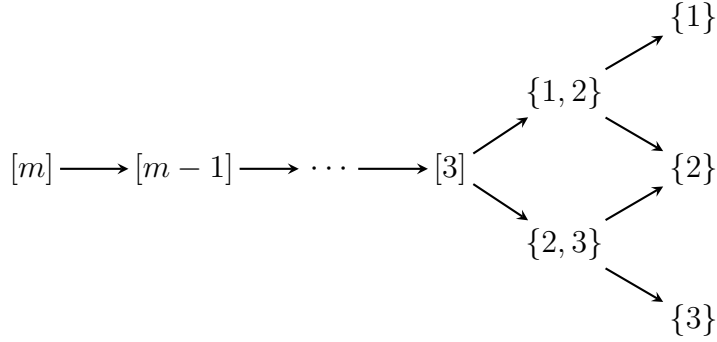
\begin{center}
    \begin{tikzpicture}[baseline]
        \node (m) at (3.5,0) {$[m]$};
        \node (m-1) at (5.5,0) {$[m-1]$};
        \node (dots) at (7.4,0) {$\,\cdots$};
        \node (3last) at (9,0) {$[3]$};
        \node (12) at (10.5,1) {$\set{1,2}$};
        \node (23) at (10.5,-1) {$\set{2,3}$};
        \node (1) at (12.2,2.0) {$\set{1}$};
        \node (2) at (12.2,0) {$\set{2}$};
        \node (3) at (12.2,-2.0) {$\set{3}$};
        
        \draw [thick, shorten <=-2pt, shorten >=-2pt, -stealth] (m) -- (m-1);
        \draw [thick, shorten <=-2pt, shorten >=-2pt, -stealth] (m-1) -- (dots);
        \draw [thick, shorten <=-2pt, shorten >=-2pt, -stealth] (dots) -- (3last);
        \draw [thick, shorten <=-2pt, shorten >=-2pt, -stealth] (3last) -- (12);
        \draw [thick, shorten <=-2pt, shorten >=-2pt, -stealth] (3last) -- (23);
        \draw [thick, shorten <=-2pt, shorten >=-2pt, -stealth] (12) -- (1);
        \draw [thick, shorten <=-2pt, shorten >=-2pt, -stealth] (12) -- (2);
        \draw [thick, shorten <=-2pt, shorten >=-2pt, -stealth] (23) -- (2);
        \draw [thick, shorten <=-2pt, shorten >=-2pt, -stealth] (23) -- (3);
    \end{tikzpicture}
\end{center}
 
\section{The multiprojective stratification on a Schubert variety}

From now on we assume that the index poset $\mathcal I$ we fixed in Section~\ref{sec:choices} is \nice{\tau}. In this section we construct the multiprojective Seshadri stratification associated to our choices of $\underline\lambda$, $\mathcal I$ and $\tau$. Let $R$ denote the multihomogeneous coordinate ring of the Schubert variety $X = X_\tau$ with respect to the embedding
\begin{align*}
    X_\tau \longhookrightarrow \prod_{i=1}^m \PP(V(\lambda_i)_{\tau_i})
\end{align*}
fixed by the sequence $\underline\lambda$. We begin by relating the multihomogeneous elements in $R$ to certain irreducible representations.

\begin{lemma}
    \label{lem:graded_components}
    There exists a graded isomorphism of $B$-modules
    \begin{align*}
    \K[\hat X_\tau] = \bigoplus_{\underline d \in \N_0^m} \K[\hat X_\tau]_{\underline d} \cong \bigoplus_{\underline d \in \N_0^m} V(\underline d \cdot \underline \lambda)_\tau^*.
    \end{align*}
\end{lemma}
\begin{proof}
    Fix a tuple $\underline d \in \N_0^m$ and let $Q' = \bigcap_{i \in I} P_i$ be the parabolic subgroup associated to the set $I = \set{i \in [m] \mid d_i \neq 0}$. It suffices to work with the Schubert variety $X_{\tau'} \subseteq G/Q'$ for $\tau' = \pi_{Q'}(\tau)$, as the surjection $X_\tau \twoheadrightarrow X_{\tau'}$ induces an isomorphism of $B$-modules $\K[\hat X_{\tau'}]_{\underline d_I} \to \K[\hat X_{\tau}]_{\underline d}$ between the graded components of their multihomogeneous coordinate rings of degree $\underline d$ and $\underline d_I = (d_i)_{i \in I}$.

    We define the $G$-equivariant, linear map $\phi: V(\underline d \cdot \underline \lambda) \to \bigotimes_{i=1}^m \, V(\lambda_i)^{\otimes d_i}$ sending a highest weight vector $v_{\underline d \cdot \underline \lambda} \in V(\underline d \cdot \underline \lambda)$ to the product $v_{\lambda_1}^{\otimes d_1} \otimes \cdots \otimes v_{\lambda_m}^{\otimes d_m}$ of highest weight vectors $v_{\lambda_i} \in V(\lambda_i)$. Every weight vector $v_{\tau} \in V(\underline d \cdot \underline \lambda)$ of weight $\tau(\underline d \cdot \underline \lambda)$ is mapped to the tensor product $v_{\tau_1}^{\otimes d_1} \otimes \cdots \otimes v_{\tau_m}^{\otimes d_m}$, where $v_{\tau_i} \in V(\lambda_i)$ is a weight vector of weight $\tau(\lambda_i)$. Therefore $\phi$ induces a morphism
    \begin{align*}
        \phi_\tau: \PP(V(\underline d \cdot \underline \lambda)_{\tau}) \to \PP(\bigotimes_{i \in I} V(\lambda_i)_{\tau_i}^{\otimes d_i}).
    \end{align*}
    It is well known that $\phi$ is injective for $\operatorname{char} \K = 0$, so $\phi_\tau$ is injective as well. Hence we have the following commutative diagram of closed, $B$-equivariant embeddings:
    \begin{equation}
        \label{eq:diagram_graded_iso}
        \begin{tikzcd}
        X_{\tau'} \arrow[d, hook] \arrow[rr, hook] & & \PP(V(\underline d \cdot \underline \lambda)_{\tau}) \arrow[d, "\phi_\tau", hook] \\
        \prod_{i \in I} \PP(V(\lambda_i)_{\tau_i}) \arrow[r, hook, "\delta"] & \prod_{i \in I} \PP(V(\lambda_i)_{\tau_i})^{\times d_i} \arrow[r, hook] & \PP(\bigotimes_{i \in I} V(\lambda_i)_{\tau_i}^{\otimes d_i})
        \end{tikzcd}
    \end{equation}
    Here $\delta$ is the diagonal embedding. This diagram implies that the image of the multicone $\hat X_\tau$ in $\bigotimes_{i=1}^m V(\lambda_i)_{\tau_i}^{\otimes d_i}$ is contained in $V(\underline d \cdot \underline \lambda)_{\tau} \subseteq \bigotimes_{i \in I} V(\lambda_i)_{\tau_i}^{\otimes d_i}$. Hence there is an induced morphism $\iota: \hat X_{\tau'} \to V(\underline d \cdot \underline \lambda)_{\tau}$, such that the following diagram commutes:
    \begin{equation}
    \label{eq:diagram_graded_iso_2}
    \begin{tikzcd}
    \hat X_{\tau'} \arrow[d, hook] \arrow[rr, "\iota", hook] & & V(\underline d \cdot \underline \lambda)_{\tau} \arrow[d, hook] \\
    \prod_{i \in I} V(\lambda_i)_{\tau_i} \arrow[r, hook] & \prod_{i \in I} V(\lambda_i)_{\tau_i}^{\times d_i} \arrow[r] & \bigotimes_{i \in I} V(\lambda_i)_{\tau_i}^{\otimes d_i}
    \end{tikzcd}
    \end{equation}
    The comorphism of the bottom row in this diagram induces a surjection from the dual space $(\bigotimes_{i \in I} V(\lambda_i)_{\tau_i}^{\otimes d_i})^*$ to the space $\K[\prod_{i \in I} V(\lambda_i)_{\tau_i}]_{\underline d_I}$. Wetherefore get the diagram of $B$-invariant surjections
    \begin{equation*}
    \begin{tikzcd}
    \K[\hat X_{\tau'}]_{\underline d_I} & V(\underline d \cdot \underline \lambda)_{\tau}^* \arrow[l, two heads] \\
    \K[\prod_{i \in I} V(\lambda_i)_{\tau_i}]_{\underline d_I} \arrow[u, two heads] & (\bigotimes_{i \in I} V(\lambda_i)_{\tau_i}^{\otimes d_i})^* \arrow[u, two heads] \arrow[l, two heads]
    \end{tikzcd}
    \end{equation*}
    The top row of this diagram is an isomorphism: Since all maps in (\ref{eq:diagram_graded_iso_2}) are $B$-equivariant, the Demazure module $V(\underline d \cdot \underline \lambda)_{\tau}$ is (linearly) spanned by the image of $\iota$.
\end{proof}

For every element $I \in \mathcal I$ the poset
\begin{align*}
    D(\lambda_I, \tau_I) = \set{\theta \in W/W_{P_I} \mid \theta \leq \tau_I}
\end{align*}
is the defining chain poset to the (one-element) sequence $(\lambda_I)$ and $\tau_I = \pi_{P_I}(\tau)$. Since the index poset is \nice{\tau}, the restriction of the monotone map
\begin{align*}
    \DCP{} \to D(\lambda_I, \tau_I), \quad (\theta, I) \mapsto \pi_{P_I}(\theta)
\end{align*}
to the subset $D_I(\underline\lambda, \tau) = \set{(\theta, J) \in \DCP{} \mid J = I}$ is an isomorphism of posets. It is also compatible with covering relations: A relation $(\theta, I) \succeq (\phi, I)$ is a covering relation in $D_I(\underline\lambda, \tau)$, if and only if $\pi_{P_I}(\theta) \geq \pi_{P_I}(\phi)$ is a covering relation in $D(\lambda_I, \tau_I)$.

As described in Section~\ref{sec:choices}, there is a Seshadri stratification on each Schubert variety $X_{\tau_I} \subseteq \PP(V(\lambda_I)_{\tau_I})$ with underlying poset $D_I(\underline\lambda, \tau)$. The strata in the subset $D_I(\underline\lambda, \tau)$ for the multiprojective stratification should be Schubert varieties in $G/Q_I$ and also be compatible with the strata for $X_{\tau_I} \subseteq \PP(V(\lambda_I)_{\tau_I})$ in the sense that $X_{(\theta, I)}$ should project to $X_{\pi_{P_I}(\theta)} \subseteq G/P_I$ via the map $G/Q_I \twoheadrightarrow G/P_I$. Hence we define the stratum $X_{(\theta, I)}$ of an element $(\theta, I) \in \DCP{}$ to be the Schubert variety
\begin{align*}
    X_{(\theta, I)} \coloneqq X_{\pi_{Q_I}(\theta)} \subseteq \prod_{i \in I} \PP(V(\lambda_i)_{\tau_i}).
\end{align*}

For each index $i = 1, \dots, m$ we also have a Seshadri stratification on the Schubert variety $X_{\tau_i} \subseteq \PP(V(\lambda_i)_{\tau_i})$ with the underlying poset
\begin{align*}
    D(\lambda_i, \tau_i) = \set{\theta \in W/W_{P_i} \mid \theta \leq \tau_i}.
\end{align*}
Its extremal functions $f_\theta$ can be pulled back to the multicone $\hat X_\tau$ via the linear projection
\begin{align*}
    \prod_{j=1}^m V(\lambda_j)_{\tau_j} \to V(\lambda_i)_{\tau_i}.
\end{align*}
Hence we can use the same extremal functions for the multiprojective stratification as well: For every $i \in [m]$ and $\phi \in D(\lambda_i, \tau_i)$ we choose a $T$-eigenvector $\ell_\phi$ of weight $-\phi(\lambda_i)$ in the irreducible $G$-representation $V(\lambda_i)^*$. The extremal function $f_{(\theta, I)}$ shall then be defined as the product 
\begin{align*}
    f_{(\theta, I)} \coloneqq \prod_{i \in \underline I} \ell_{\pi_{P_i}(\theta)}.
\end{align*}
This function is of multidegree $e_I$ and by Lemma~\ref{lem:graded_components} it can also be interpreted as a weight vector of weight $-\theta(\lambda_I)$ in the representation $V(\lambda_I)^*$.

Note that this construction is a generalization of the following two stratifications:
\begin{enumerate}[label=(\alph{enumi})]
    \item When choosing the one-element sequence $\underline\lambda = (\lambda)$ one obtains the original stratification on $X_\tau \subseteq \PP(V(\lambda)_\tau)$ from~\cite{seshstratandschubvar}.
    \item Let $Q = P_{k_1} \cap \dots \cap P_{k_m}$ be a parabolic subgroup in Dynkin type \texttt{A} as in Example~\ref{ex:strat_type_A}. By Corollary~\ref{cor:totally_ordered_nice_I}, the index poset $\mathcal I = \set{[1], \dots, [m]}$ is \nice{\tau} for the sequence $\underline\lambda = (\omega_{k_m}, \dots, \omega_{k_1})$ and $\tau = w_0 W_Q$. Hence the defining chain poset $\DCP{}$ is isomorphic to $\ulWlambda \cong \ulW$.
\end{enumerate}

\begin{lemma}
    \label{lem:prop_strata}
    For all $J \subseteq I$ in $\mathcal I$ and $(\theta, I) \in \DCP{}$, we have
    \begin{align*}
    \set{(v_1, \dots, v_m) \in \hat X_{(\theta, I)} \mid v_i = 0 \mkern8mu \forall \mkern1mu i \in I \setminus J} = \hat X_{(\min_Q \circ \pi_{Q_J}(\theta), J)}.
    \end{align*}
\end{lemma}
\begin{proof}
    By Corollary~\ref{cor:dcp_go_down}, the element $(\min_Q \circ \pi_{Q_J}(\theta), J)$ indeed lies in the defining chain poset. Let $v = (v_1, \dots, v_m) \in \hat X_{(\theta, I)}$ and choose $w_i \in V(\lambda_i)_{\tau_i} \setminus \set{0}$ with $v_i \in \K w_i$ for all $i \in I$. As the diagram
    \begin{equation*}
    \begin{tikzcd}
    X_{\pi_{Q_I}(\theta)} \arrow[r, hook] \arrow[d, two heads] & \prod_{i \in I} \PP(V(\lambda_i)_{\tau_i}) \arrow[d, two heads] &  \\
    X_{\pi_{Q_J}(\theta)} \arrow[r, hook] & \prod_{i \in J} \PP(V(\lambda_i)_{\tau_i})
    \end{tikzcd}
    \end{equation*}
    commutes, the tuple $([w_j])_{j \in J} \in \prod_{j \in J} \PP(V(\lambda_j)_{\tau_j})$ can be viewed as an element of $G/Q_J$ and lies in the Schubert variety $X_{\pi_{Q_J}(\theta)}$. When $v_i = 0$ holds for all $i \in I \setminus J$, we have $v \in \hat X_{(\min_Q \circ \pi_{Q_J}(\theta), J)}$. Conversely, if $v \in \hat X_{(\min_Q \circ \pi_{Q_J}(\theta), J)}$, we choose a preimage via the projection $X_{\pi_{Q_I}(\theta)} \twoheadrightarrow X_{\pi_{Q_J}(\theta)}$ and set all coordinates in $I \setminus J$ to zero. Hence $v$ lies in the multicone $\hat X_{(\theta, I)}$.
\end{proof}

\begin{theorem}
    \label{thm:stratification}
    Suppose that the poset $\mathcal I$ we chose in Section~\ref{sec:choices} is $\tau$-standard. Then the varieties $X_{(\theta, I)}$ and extremal functions $f_{(\theta, I)}$ defined at the beginning of this section form a (multiprojective) Seshadri stratification on $X = X_\tau$. The defining chain poset $A = \DCP{}$ is the underlying poset of this stratification and $\mathcal I$ is its index poset.
\end{theorem}
\begin{proof}
    It is well known, that Schubert-varieties are smooth in codimension one (see e.\,g. \cite[Corollary 3.5]{seshstratandschubvar}). Their multicones $\hat X_{(\theta, I)} \subseteq \prod_{i=1}^m V(\lambda_i)_{\tau_i}$ are closed, irreducible subvarieties of $\hat X_\tau$ and smooth in codimension one as well (\cite[Corollary A.11]{ownarticle}).
    
    The relation $(\theta, I) \succeq (\phi, J)$ in $\DCP{}$ is equivalent to the inclusion $\hat X_{(\theta, I)} \supseteq \hat X_{(\phi, J)}$ of their corresponding multicones. Clearly $\hat X_{(\theta, I)} \supseteq \hat X_{(\phi, J)}$ holds for every covering relation $(\theta, I) \succ (\phi, J)$ in $\DCP{}$. Conversely if $\hat X_{(\theta, I)} \supseteq \hat X_{(\phi, J)}$, then $\hat X_{(\min_Q \circ \pi_{Q_J}(\theta), J)}$ lies in between these two multicones. In particular, $\pi_{Q_J}(\theta) \geq \pi_{Q_J}(\phi)$. Because of the $Q_J$-minimality of $\phi$ we can lift this to $\theta \geq \operatorname{min}_Q \circ \pi_{Q_J}(\theta) \geq \operatorname{min}_Q \circ \pi_{Q_J}(\phi) = \phi$, which implies $\rho(\theta, I) \geq \rho(\phi, J)$. Thus $(\theta, I) \succeq (\phi, J)$ since $\rho$ is an isomorphism.
    
    We now check the requirements~\ref{itm:seshadri_strat_a}-\ref{itm:seshadri_strat_c} for a Seshadri stratification: The defining chain poset is a graded poset and the rank of an element $(\theta, I) \in \DCP{}$ is equal to the length of the subposet $\Wlambda_{\preceq (\theta, I)}$. By Theorem~\ref{thm:unique_def_chain} and Remark~\ref{rem:restriction_of_DCP} this length is given by $r(\theta) + |I| - 1$, which is the dimension of the multicone $\hat X_{(\theta, I)}$. Therefore \ref{itm:seshadri_strat_a} is fulfilled.

    Let $(\theta, I), (\phi, J)$ be two elements in $\DCP{}$ and $v = (v_1, \dots, v_m)$ be a point in the multicone $\hat X_{(\theta, I)}$. By its definition, the extremal function $f_{(\phi, J)}$ vanishes on the point $v$, if and only if $\ell_{\pi_{P_j}(\theta)}$ vanishes on $v_j$ for some $j \in \underline J$. The vanishing behaviour of $\ell_{\pi_{P_j}(\phi)}$ can be described via the Seshadri stratification on $X_{\tau_j} \subseteq \PP(V(\lambda_j)_{\tau_j})$ with underlying poset $D(\lambda_j, \tau_j) = \set{\sigma \in W/W_{P_j} \mid \sigma \leq \tau_j}$. Since \ref{itm:seshadri_strat_c} holds for this stratification and $v_j \in \hat X_{\pi_{P_j}(\theta)}$, it follows:
    \begin{align}
        \label{eq:vanishing_criterion}
        f_{(\phi, J)}(v) = 0 \quad \Longleftrightarrow \quad \text{$v_j \in \hat X_\sigma$ for some $j \in \underline J$ and $\sigma < \pi_{P_j}(\theta)$ in $W/W_{P_j}$}.
    \end{align}
    
    For condition~\ref{itm:seshadri_strat_b} we assume $(\phi, J) \npreceq (\theta, I)$. If $J \nsubseteq I$, then $f_{(\phi,J)}$ vanishes on $\hat X_{(\theta,I)}$ by definition of the strata and the requirement~(\ref{eq:need_for_s2}) on the poset $\mathcal I$. Now let $J \subseteq I$. We have $\pi_{P_J}(\phi) \nleq \pi_{P_J}(\theta)$, otherwise $\operatorname{min}_Q \circ \pi_{P_J}(\phi) \leq \operatorname{min}_Q \circ \pi_{P_J}(\theta) \leq \theta \leq \operatorname{max}_Q \circ \pi_{P_I}(\theta)$ would be a contradiction to $\rho(\phi, J) \nleq \rho(\theta, I)$. By the definition of $P_J$ there exists an index $j \in \underline J$, such that $\pi_{P_j}(\phi) \nleq \pi_{P_j}(\theta)$, where $P_j = B W_{\lambda_j} B$ is the parabolic subgroup associated to the dominant weight $\lambda_j$. Using the equivalence~(\ref{eq:vanishing_criterion}) we see that $f_{(\phi, J)}$ vanishes on $\hat X_{(\theta, I)}$.
    
    Lastly we show~\ref{itm:seshadri_strat_c}. The function $f_{(\theta, I)}$ vanishes on every point $v \in \hat X_{(\phi, J)}$ for $(\theta, I) \succ (\phi, J) \in \DCP{}$. This is clearly the case for $J \subsetneq I$. If $J = I$, there exists an index $i \in \underline I$, such that $\pi_{P_i}(\phi) < \pi_{P_i}(\theta)$, since $\pi_{P_I}(\phi) < \pi_{P_I}(\theta)$. Hence $f_{(\theta, I)}$ vanishes on $v$ by~(\ref{eq:vanishing_criterion}).
    
    Conversely we assume, that $f_{(\theta, I)}$ vanishes on $v$. Hence there exists an index $i \in \underline I$ and an element $\phi \in W/W_{P_i}$ with $v_i \in \hat X_\phi$ and $\phi < \pi_{P_i}(\theta)$. First we consider the case $v_i = 0$. If $I$ is minimal in $\mathcal I$, then $v = 0$ is automatically contained in the right hand side of~(\ref{eq:s3}). Else, by Corollary~\ref{cor:dcp_go_down} and Lemma~\ref{lem:prop_strata}, the point $v$ lies in the stratum to $(\min_Q \circ \pi_{Q_J}(\theta), J)$ for $J = I \setminus \set{i}$ and this element is strictly smaller than $(\theta, I)$ in $\DCP{}$. It remains the case $v_i \neq 0$. For each $k \in I$ we choose an element $w_k \in V(\lambda_k)_{\tau_k} \setminus \set{0}$, such that $v_k \in \K w_k$. The tuple $([w_k])_{k \in I}$ lies in a unique Schubert cell $C_\sigma$ for $\sigma \in W/W_{Q_I}$, viewed as a locally closed subvariety of $\prod_{k \in I} \PP(V(\lambda_k)_{\tau_k})$. This element $\sigma$ is strictly smaller than $\pi_{Q_I}(\theta)$, since $v_i \in \hat X_\phi$ and $\phi < \pi_{P_i}(\theta)$. By Lemma~\ref{lem:bruhat_interval} there now exists $\psi \in W/W_{Q_I}$ covered by $\pi_{Q_I}(\theta)$, such that $\pi_{Q_I}(\theta) > \psi \geq \sigma$ and $\pi_{P_I}(\theta) > \pi_{P_I}(\psi) \geq \pi_{P_I}(\sigma)$. Hence $(\psi, I)$ lies in the defining chain poset, is strictly smaller than $(\theta, I)$ and $v \in \hat X_{(\psi, I)}$ is contained in the right hand side of~(\ref{eq:s3}).
\end{proof}

\begin{corollary}
    For $\tau = w_0 W_Q$ the stratification from Theorem~\ref{thm:stratification} is a (multiprojective) Seshadri stratification on $X = G/Q$. One can determine combinatorially which posets $\mathcal I$ are \nice{\tau} via Theorem~\ref{thm:nice_I} and therefore give rise to such a stratification.
\end{corollary}

Similar to the Seshadri stratification on $X_\tau \subseteq V(\lambda)_\tau$ constructed in~\cite{seshstratandschubvar}, the bonds of the multiprojective stratification can be described via the root system combinatorics of $G$. In order to prove the next lemma, we therefore need more notation for root subgroups. For every root $\alpha$ in the root system $\Phi$ let $U_\alpha$ denote the associated root subgroup of $G$ (see \cite[Section 26.3]{humphreys2012linear}). Let $\mathbb G_{\mathrm{a}}$ denote the additive algebraic group $(\K, +)$. Up to a scalar multiple, there exists an unique isomorphism $\varepsilon_\alpha: \mathbb G_{\mathrm{a}} \to U_\alpha$, such that $t \varepsilon_\alpha(x) t^{-1} = \varepsilon_\alpha(\alpha(t) x)$ holds for all $t \in T$ and $x \in \mathbb G_{\mathrm{a}}$. 

\begin{lemma}
    \label{lem:bonds}
    The bond to a covering relation $(\theta, I) \succ (\phi, J)$ in $\DCP{}$ is given by 
    \begin{align*}
    b_{(\theta, I), (\phi, J)} = \begin{cases} \vert \langle \phi(\lambda_I), \beta^\vee \rangle \vert, & \text{if} \ I = J, \\ 1, & \text{if} \ I \neq J, \end{cases}
    \end{align*}
    where $\beta$ is the unique positive root with $s_\beta \cdot \min_B(\phi) = \min_B(\theta) \in W$ in the case $I = J$.
\end{lemma}
\begin{proof}
    The case $I \neq J$ is covered by Lemma~\ref{lem:projective_covering_relation}. Now let $I = J$ and $\underline d \in \N_0^m$ be the sum of all vectors $e_i$ for $i \in I$. We fix a weight vector $v_I$ in the Demazure module $V(\underline d \cdot \underline \lambda)_\tau$ of weight $\phi(\underline d \cdot \underline \lambda)$ and a weight vector $v_i \in V(\lambda_i)_{\tau_i}$ of weight $\phi(\lambda_i)$ for each $i \in I$. Our proof relies on \cite[Lemma 3.3]{seshstratandschubvar}: It states that the map
    \begin{align*}
        \overline f: U_\phi \times U_{-\beta} \to \PP(V(\underline d \cdot \underline \lambda)_\tau), \quad (u, v) \mapsto uv \cdot [v_I],
    \end{align*}
    is an isomorphism onto an open subvariety of $X_\theta \subseteq \PP(V(\underline d \cdot \underline \lambda)_\tau)$, where $U_\phi \subseteq G$ is a direct product of root subgroups. The explicit construction of $U_\phi$ does not matter for this proof. It only matters the fact that it is generated by root subgroups to positive roots. For each $i \in I$ we also have a well defined morphism $f_i: U_\phi \times U_{-\beta} \to \PP(V(\lambda_i)_{\tau_i})$, $(u, v) \mapsto uv \cdot [v_i]$. Together, they induce the morphism
    \begin{align*}
        f: U_\phi \times U_{-\beta} \to \prod_{i \in I} \PP(V(\lambda_i)_{\tau_i}).
    \end{align*}
    Since all maps in (\ref{eq:diagram_graded_iso}) are $B$-equivariant, we have the commuting diagram
    \begin{equation*}
        \begin{tikzcd}
        U_\phi \times U_{-\beta} \arrow[r, "f"] \arrow[d, "\overline{f}", swap, hook] & \prod_{i \in I} \PP(V(\lambda_i)_{\tau_i}) \arrow[d, hook] \\
        \PP(V(\underline d \cdot \underline \lambda)_\tau) \arrow[r, hook] & \PP(\bigotimes_{i \in I} \mkern1mu V(\lambda_i)_{\tau_i})
        \end{tikzcd}
    \end{equation*}
    It thus follows that $f$ is an isomorphism onto an open subvariety of $X_\theta \subseteq \prod_{i \in I} \PP(V(\lambda_i)_{\tau_i})$ as well. In this subvariety the divisor $\hat X_{(\phi, I)}$ of $\hat X_{(\theta, I)}$ is given by the subset $U_\phi \times \set{1}$.
    
    As the extremal function $f_{(\theta, I)}$ can be seen as a linear function on $\bigotimes_{i \in \underline I} V(\lambda_i)_{\tau_i}$, we have to project to the smaller index set $\underline I \subseteq I$. Each weight vector $v_{\underline I} \in V(\lambda_I)_\tau$ of weight $\phi(\lambda_I)$ defines a morphism
    \begin{align*}
        f_{\underline I}\mkern-3mu : \mkern2mu U_\phi \times U_{-\beta} \to \PP(V(\lambda_I)_\tau), \quad (u,v) \mapsto uv \cdot [v_{\underline I}]
    \end{align*}
    and the following diagram commutes:
    \begin{equation*}
        \begin{tikzcd}
        U_\phi \times U_{-\beta} \arrow[r, "f", hook] \arrow[dr, "f_{\underline I}", swap] & \prod_{i \in I} \PP(V(\lambda_i)_{\tau_i}) \arrow[r, two heads] & \prod_{i \in \underline I} \PP(V(\lambda_i)_{\tau_i}) \arrow[d, hook] \\
        & \PP(V(\lambda_I)_\tau) \arrow[r, hook] & \PP(\bigotimes_{i \in \underline I} \mkern1mu V(\lambda_i)_{\tau_i})
        \end{tikzcd}
    \end{equation*}
    We choose a parametrization of the affine spaces $U_\phi$ and $U_{-\beta}$ by parameters $t_\gamma \in \K$ for $\gamma \in \Phi_\phi^+$ and $t_\beta \in \K$ respectively. The action of the root subgroup $U_{-\beta} \cong \mathbb G_{\mathrm{a}}$ on the weight vector $v_{\underline I} \in V(e_I \cdot \underline \lambda)_\tau$ is given by the polynomial
    \begin{align*}
        t_\beta \cdot v_{\underline I} = v_{\underline I} + t_\beta w_1 + t_\beta^2 w_2 + \dots + t_\beta^b w_b,
    \end{align*}
    where $b = \vert \langle \phi(\lambda_I), \beta^\vee \rangle \vert$ and $w_i$ is a weight vector in $V(\lambda_I)_\tau$ of weight $\phi(\lambda_I) - i \beta$ for all $i = 1, \dots, s$ (this was discussed in the proof of Proposition 27.2 in Humphreys' book \cite{humphreys2012linear}). Using the coordinates $t_\gamma$ and $t_\beta$, the elements in the image of $f_{\underline I}$ are of the form
    \begin{align}
        \label{eq:parametrization_bond}
        [t_\beta^b w_b + \text{sum of weight vectors in $V(\lambda_I)_{\tau}$ of greater weights than $\theta(\lambda_I)$}],
    \end{align}
    because $U_\phi$ is unipotent and generated by positive root subgroups. By the construction of the extremal function $f_{(\theta, I)}$, it is a dual vector to the extremal weight space in $V(\lambda_I)_\tau$ of weight $\theta(\lambda_I) = s_\beta(\phi(\lambda_I)) = \phi(\lambda_I) - b \beta$. Applying this to the elements in~(\ref{eq:parametrization_bond}) implies that the vanishing multiplicity of $f_{(\theta, I)}$ at the divisor $\hat X_{(\phi, I)}$ is equal to $b$.
\end{proof}

\section{The LS-fan of monoids}
\label{sec:LS-fan}

From now on we fix a multiprojective Seshadri stratification on a Schubert variety $X_\tau$ as in Theorem~\ref{thm:stratification}. Recall that the defining chain poset $\DCP{}$ plays the role of the underlying poset $A$ of this stratification. In particular, we assume that the index poset $\mathcal I$ chosen in Section~\ref{sec:choices} is \nice{\tau}. We show that the stratification is balanced and of LS-type and the elements in the fan of monoids $\Gamma$ correspond bijectively to $\tau$-standard LS-tableaux of type $(\underline\lambda, \mathcal I)$.

Recall that a Seshadri stratification on a projective variety $X \subseteq \PP(V)$ is of \textit{LS-type} (first defined in~\cite{seshstratnormal}), if all extremal functions are linear and if each monoid $\Gamma_{\mathfrak C}$ to a maximal chain $\mathfrak C = \set{p_r > \dots > p_0}$ coincides with the \textit{LS-monoid} $\mathrm{LS}_{\mathfrak C} = \mathrm{LS}_{\mathfrak C} \cap \Q_{\geq 0}^{\mathfrak C}$ of $\mathfrak C$, \ie the intersection of the \textit{LS-lattice}
\begin{align*}
    \mathrm{LS}_{\mathfrak C} = \Set{ \sum_{i=0}^s a_i e_{p_i} \in \Q^{\mathfrak C} \ \middle\vert \ b_{p_i, p_{i-1}} (a_i + \dots + a_s) \in \Z \ \forall i \in [s], a_0 + \dots + a_s \in \Z}.
    \end{align*}
of $\mathfrak C$ and the positive orthant $\Q_{\geq 0}^{\mathfrak C}$. These stratifications are named after their connection to LS-paths and LS-algebras (cf. \cite{chirivi2000}, \cite{seshstratLS}) as explained in~\cite[Section 3.4]{seshstratnormal}.

Our multiprojective stratification on $X_\tau$ was built by gluing the Seshadri stratifications on the Schubert varieties $X_{\tau_I} \subseteq \PP(V(\lambda_I)_{\tau_I})$ for $I \in \mathcal I$. The disjoint union of their underlying posets
\begin{align*}
    D(\lambda_I, \tau_I) = \set{\theta \in W/W_{P_I} \mid \theta \leq \tau_I}
\end{align*}
forms the defining chain poset $\DCP{}$. It was shown in both~\cite{seshstratandschubvar} and \cite{seshstratgeometric} that the stratification on $X_{\tau_I}$ is of LS-type. We should therefore expect that the multiprojective stratification is of LS-type as well.

\begin{definition}[{\cite[Definition 10.2]{ownarticle}}]
    A Seshadri stratification on a multiprojective variety $X \subseteq \prod_{i=1}^m \PP(V_i)$ is called \textbf{of LS-type}, if the following conditions are fulfilled:
    \begin{enumerate}[label=(\alph{enumi})]
        \item \label{itm:LS_a} Each component of the multidegree $\deg f_p \in \N_0^m$ is at most $1$ for all $p \in A$;
        \item \label{itm:LS_b} if $I_p = I_q$ for any two elements $p, q \in A$, then $\deg f_p = \deg f_q$;
        \item \label{itm:LS_c} the fan of monoids $\Gamma$ is equal to the union $\bigcup_{\mathfrak C} \mathrm{LS}_{\mathfrak C}^+$ over all maximal chains $\mathfrak C$ in $A$.
    \end{enumerate}
\end{definition}

Clearly, the first two conditions are fulfilled for the multiprojective stratification on $X_\tau$. It remains to prove the third statement. For this reason we fix the following notation.
\begin{itemize}
    \item For each index set $I \in \mathcal I$ and maximal chain $\mathfrak C$ in $D(\lambda_I, \tau_I)$ we write $\mathrm{LS}_{\mathfrak C, \lambda_I}$ for the associated LS-lattice and $\mathrm{LS}_{\mathfrak C, \lambda_I}^+ = \mathrm{LS}_{\mathfrak C, \lambda_I} \cap \Q_{\geq 0}^{\mathfrak C}$ for the LS-monoid. The fan of monoids of the stratification on $X_{\tau_I} \subseteq \PP(V(\lambda_I)_{\tau_I})$ shall be denoted by $\mathrm{LS}_{\lambda_I}^+$.
    \item For every maximal chain $\mathfrak C$ in $\DCP{}$ let $\mathrm{LS}_{\mathfrak C, \underline\lambda}$ be the associated LS-lattice and let $\mathrm{LS}_{\mathfrak C, \underline\lambda}^+ = \mathrm{LS}_{\mathfrak C, \underline\lambda} \cap \Q^{\mathfrak C}_{\geq 0}$ denote its LS-monoid. The set-theoretic union
    \begin{align*}
        \mathrm{LS}_{\underline\lambda}^+ = \bigcup_{\mathfrak C} \mkern2mu \mathrm{LS}_{\mathfrak C, \underline\lambda}^+
    \end{align*}
    over all maximal chains $\mathfrak C$ in $\DCP{}$ is called the \textit{Lakshmibai-Seshadri-fan of monoids} corresponding to to $\underline\lambda$, $\tau$ and $\mathcal I$.
\end{itemize}

Even though the LS-fans $\mathrm{LS}_{\underline\lambda}^+$ does not only depend on the sequence $\underline\lambda$, but also on the coset $\tau \in W/W_Q$ and the index poset $\mathcal I$, we usually do not index the LS-fan by $\tau$ and $\mathcal I$ to simplify the notation. Similarly, $\mathrm{LS}_{\lambda_I}^+$ also depends on the coset $\tau_I \in W/W_{P_I}$.

By identifying each poset $D(\lambda_I, \tau_I)$ with $D_I(\underline\lambda, \tau) = \set{(\theta, J) \in \DCP{} \mid J = I}$, we can uniquely decompose the elements $\underline a \in \mathrm{LS}_{\underline\lambda}^+$ into a sum $\underline a = \sum_{I \in \mathcal I} \underline a^{(I)}$ of elements $\underline a^{(I)} \in \mathrm{LS}_{\underline\lambda}^+ \cap \Q^{D(\lambda_I, \tau_I)}$. The definition of the LS-fan $\mathrm{LS}_{\underline\lambda}^+$ suggests that these elements $\underline a^{(I)}$ lie in the LS-fan $\mathrm{LS}_{\lambda_I}^+$.

\begin{lemma}
    \label{lem:LS_fan_decomposition}
    For each $I \in \mathcal I$ the intersection $\mathrm{LS}_{\underline\lambda}^+ \cap \Q^{D(\lambda_I, \tau_I)}$ coincides with the LS-fan $\mathrm{LS}_{\lambda_I}^+$ of the stratification on $X_{\tau_I} \subseteq \PP(V(\lambda_I)_{\tau_I})$.
\end{lemma}
\begin{proof}
    We fix a maximal chain $\mathfrak C$ in $\DCP{}$ and let $I_1 \subseteq \dots \subseteq I_m = [m]$ be the associated maximal chain in $\mathcal I$. This induces a partition of $\mathfrak C$ into the subchains
    \begin{align*}
        \mathfrak C_j = \set{(\theta, I) \in \mathfrak C \mid I = I_j}
    \end{align*}
    for $j = 1, \dots, m$. If $p > q$ is a covering relation in $\mathfrak C$, such that $\set{p, q} \nsubseteq \mathfrak C_j$ for all $j$, then the bond $b_{p,q}$ is equal to $1$ by Lemma~\ref{lem:projective_covering_relation}. It is easy to see from the definition of an LS-lattice, that these lattices decompose along bonds of $1$, \ie we have
    \begin{align*}
        \mathrm{LS}_{\mathfrak C, \underline\lambda} = \prod_{j=1}^m \mathrm{LS}_{\mathfrak C, \underline\lambda} \cap \Q^{\mathfrak C_j}.
    \end{align*}
    
    For fixed $j \in [m]$ let $J = I_j$, $(\theta, J) \succ (\phi, J)$ be a covering relation in $\mathfrak C_j$ and $\beta$ be the unique positive root with $s_\beta \cdot \min_B(\phi) = \min_B(\theta)$. By Lemma~\ref{lem:bonds}, the corresponding covering relation $\pi_{P_J}(\theta) > \pi_{P_J}(\phi)$ in the poset $D(\lambda_J, \tau_J)$ has the bond $\vert \langle \phi(\lambda_J), \beta^\vee \rangle \vert$. Hence the bonds inside $\mathfrak C_j$ agree with the bonds in the corresponding chain in $D(\lambda_J, \tau_J)$. Since LS-lattices only depend on the bonds, the sublattice $\mathrm{LS}_{\mathfrak C, \underline\lambda} \cap \Q^{\mathfrak C_j} \subseteq \mathrm{LS}_{\mathfrak C, \underline\lambda}$ coincides with the LS-lattice $\mathrm{LS}_{\mathfrak C_j, \lambda_J} \subseteq \mathrm{LS}_{\lambda_J}$ to the chain $\mathfrak C_j \subseteq D(\lambda_J, \tau_J)$.

    For each $I \in \mathcal I$ it now follows the equality $\mathrm{LS}_{\underline\lambda}^+ \cap \Q^{D(\lambda_I, \tau_I)} = \mathrm{LS}_{\lambda_I}^+$: For every maximal chain $\mathfrak C$ in $\DCP{}$ we have
    \begin{align*}
        \mathrm{LS}_{\mathfrak C, \underline\lambda}^+ \cap \Q^{D(\lambda_I, \tau_I)} = \mathrm{LS}_{\mathfrak C, \underline\lambda} \cap \Q^{\DCP{}}_{\geq 0} \cap \Q^{D(\lambda_I, \tau_I)} = \mathrm{LS}_{\mathfrak C, \lambda_I} \cap \Q^{\mathfrak C_I}_{\geq 0} \subseteq \mathrm{LS}_{\lambda_I}^+
    \end{align*}
    for the chain $\mathfrak C_I = \mathfrak C \cap D(\lambda_I, \tau_I)$. Conversely, each maximal chain $\mathfrak C$ in $D(\lambda_I, \tau_I)$ is contained in a maximal chain $\mathfrak D$ in $\DCP{}$. Hence $\mathrm{LS}_{\mathfrak C, \lambda_I}^+$ is a subset of $\mathrm{LS}_{\mathfrak D, \underline\lambda}^+ \subseteq \mathrm{LS}_{\underline\lambda}^+$.
\end{proof}

Recall from~\cite{ownarticle} that the \textit{degree map} of a Seshadri stratification with underlying poset $A$ is defined as the $\Q$-linear map
\begin{align*}
    \mathrm{deg}: \Q^A \to \Q^m, \quad e_p \mapsto \deg f_p.
\end{align*}
The multidegree $\deg g$ of a multihomogeneous regular function $g \in R \setminus \set{0}$ always coincides with the degree $\deg \mathcal V(g)$ of its quasi-valuation (cf. Lemma 4.3 in \loccit{}). In particular, the degree of every element in the fan of monoids is an element of $\N_0^m$.

We can conclude that the degree of every element $\underline a \in \mathrm{LS}_{\underline\lambda}^+$ is contained in $\N_0^m$: By the above lemma, we can decompose $\underline a$ into a sum of elements $\underline a^{(I)} \in \mathrm{LS}_{\lambda_I}^+$ over $I \in \mathcal I$. Since $\mathrm{LS}_{\lambda_I}^+$ is the fan of monoids of the Seshadri stratification on $X_{\tau_I} \subseteq \PP(V(\lambda_I)_{\tau_I})$, we also have a degree map $\deg_I: \Q^{D(\lambda_I, \tau_I)} \to \Q$ for each index set $I \in \mathcal I$. The degree $\deg_I \underline a^{(I)}$ is a non-negative integer and $\deg \underline a = \sum_{I \in \mathcal I} \deg_I(\underline a^{(I)}) \mkern2mu e_I$ thus is an element of $\N_0^m$. 

This provides two partitions of the LS-fan of monoids $\mathrm{LS}_{\underline\lambda}^+$ into the subsets
\begin{align*}
    \mathrm{LS}_{\underline \lambda}^+(\underline d) = \set{\underline a \in \mathrm{LS}_{\underline \lambda}^+ \mid \deg \underline a = \underline d} \quad \text{and} \quad \mathrm{LS}_{\underline \lambda}^+(d) = \set{\underline a \in \mathrm{LS}_{\underline \lambda}^+ \mid \vert \mkern-2mu \deg \underline a \mkern1mu \vert = d}
\end{align*}
for $\underline d \in \N_0^m$ or $d \in \N_0$ respectively. Here $\vert \mkern-2mu \deg \underline a \mkern1mu \vert$ denotes the \textit{total degree}, \ie the sum of all entries in the degree vector $\deg \underline a$. Similarly, for each $I \in \mathcal I$ the LS-fan $\mathrm{LS}_{\lambda_I}^+$ is the disjoint union of the subsets
\begin{align*}
     \mathrm{LS}_{\lambda_I}^+(d) = \set{\underline a \in \mathrm{LS}_{\lambda_I}^+ \mid \deg \underline a = d}.
\end{align*}

Lemma~\ref{lem:LS_fan_decomposition} also implies that the LS-fan of monoids $\mathrm{LS}_{\underline\lambda}^+$ is completely determined by the decomposition over the index poset $\mathcal I$ and the defining chain poset:
\begin{align}
    \label{eq:LS_fan_decomposition}
    \mathrm{LS}_{\underline\lambda}^+ = \Set{\underline a = (\underline a^{(I)})_{I \in \mathcal I} \in \prod_{I \in \mathcal I} \mathrm{LS}_{\lambda_I}^+ \ \middle\vert \ \exists \, \text{max.\,chain} \; \mathfrak C \subseteq \DCP{}\mkern-3mu : \supp \underline a \subseteq \mathfrak C} \mkern-2mu .
\end{align}
This decomposition is also compatible with the degrees: For every $d \in \N_0$ and $I \in \mathcal I$ we have the inclusion $\mathrm{LS}_{\lambda_I}^+(d) \subseteq \mathrm{LS}_{\underline\lambda}^+(d e_I)$.

To each element $\underline a$ in the LS-fan $\mathrm{LS}_{\underline\lambda}^+$ one can associate a dominant weight. If $a_{(\theta, I)} \in \K$ is the coefficient in $\underline a$ of the basis vector $e_{(\theta, I)}$, then we define
\begin{align*}
    \wt \underline a = \sum_{(\theta, I) \in \DCP{}} a_{(\theta, I)} \cdot \theta(\lambda_I) \in \Lambda \otimes_\Z \Q.
\end{align*}
By its definition, this element only lies in the rational span of the weight lattice $\Lambda$, but it follows from the bijections in Proposition~\ref{prop:bijection_YT_fan} that $\wt \underline a$ is actually contained in $\Lambda^+$: Let $\underline d$ be the degree of $\underline a$ and $\underline\pi = (\pi_1, \dots, \pi_s) \in \LST$ be the unique LS-tableau with $\Theta_{\underline d}(\underline\pi) = \underline a$. By the construction of the maps $\Theta_{\underline d}$ and $\Theta_d^{(I)}$, the element $\wt \underline a$ is equal to the end point $(\pi_1 \ast \dots \ast \pi_s)(1)$, which is a dominant weight.

\begin{remark}
    \label{rem:LS-fan_induced_strat}
    The LS-fan of monoids is compatible with the induced stratification on $X_{(\theta, I)}$ for a fixed $(\theta, I) \in \DCP{}$ in the following sense: For each $J \in \mathcal I$ with $J \subseteq I$ we have an induced stratification on the Schubert variety $X_{\theta_J} \subseteq \PP(V(\lambda_J)_{\theta_J})$ to $\theta_J = \pi_{P_J}(\theta)$ with the underlying poset 
    \begin{align*}
        D(\lambda_J, \theta_J) = \set{\phi \in W/W_{P_J} \mid \phi \leq \theta_J}.
    \end{align*}
    We include additional indices to differentiate between the LS-fan $\mathrm{LS}_{\lambda_J, \tau_J}^+ \subseteq \Q^{D(\lambda_J, \tau_J)}$ associated to the stratification on $X_{\tau_J}$ and the LS-fan $\mathrm{LS}_{\lambda_J, \theta_J}^+$ to the stratification on $X_{\theta_J}$. Then $\mathrm{LS}_{\lambda_J, \theta_J}^+$ can naturally be seen as a subset of $\mathrm{LS}_{\lambda_J, \tau_J}^+$ via the linear map $\Q^{D(\lambda_J, \theta_J)} \hookrightarrow \Q^{D(\lambda_J, \tau_J)}$. An element $\underline a \in \mathrm{LS}_{\lambda_J, \tau_J}^+$ is contained in this subset, if and only if it is zero or the maximal element in its support is less or equal to $\theta_I$ in $D(\lambda_I, \tau_I)$. Analogously, the LS-fan
    \begin{align*}
        \mathrm{LS}_{\underline\lambda, \theta}^+ = \Set{\underline a \in \prod_{J \in \mathcal I \atop J \subseteq I} \mathrm{LS}_{\lambda_J, \theta_J}^+ \ \middle\vert \ \exists \, \text{max.\,chain} \; \mathfrak C \subseteq \DCP{}: (\theta, I) \in \mathfrak C, \mkern3mu \supp \underline a \subseteq \mathfrak C}
    \end{align*}
    to the stratification on $X_{(\theta, I)}$ can be seen as a subset of the LS-fan $\mathrm{LS}_{\underline\lambda, \tau}^+$ to the stratification on $X_\tau$ (see Remark~\ref{rem:restriction_of_DCP}). This subset contains exactly those elements, which are either zero or the maximal element in their support is less or equal to $(\theta, I) \in \DCP{}$.
\end{remark}

Analogous to the stratification from Example~\ref{ex:strat_type_A}, the fan of monoids $\mathrm{LS}_{\lambda_I}^+$ also has an interpretation in terms of LS-tableaux. For every $d \in \N_0$ let
\begin{align*}
    \B(d \lambda_I)_{\tau_I} = \set{(\sigma_p, \dots, \sigma_1; 0, a_p, \dots, a_1 = 1) \in \B(d \lambda_I) \mid \sigma_p \leq \tau_I}.
\end{align*}
be the set of all LS-paths in $\B(d \lambda)$, such that their initial direction is bounded by $\tau_I$. Using the language of LS-tableaux: This can be viewed as the set of all $\tau_I$-standard LS-tableaux of shape $(d \lambda_I)$. It was proved in~\cite[Proposition A.6]{seshstratandschubvar} that the map
\begin{align}
    \label{eq:bij_LS_paths_fan}
    \Theta^{(I)}_d: \mathbb{B}(d \lambda_I)_{\tau_I} \to \mathrm{LS}_{\lambda_I}^+(d), \quad (\sigma_p, \dots, \sigma_1; 0, a_p, \dots, a_1 = 1) \mapsto \sum_{j=1}^p (a_j - a_{j+1}) d e_{\sigma_j}
\end{align}
is a bijection, where $a_{p+1} \coloneqq 0$. 

It is known that every LS-path $\pi \in \mathbb B(d\lambda_I)_{\tau_I}$ can be uniquely decomposed (up to a reparametrization) into a concatenation $\pi = \pi_1 \ast \dots \ast \pi_d$ of LS-paths $\pi_k \in \mathbb B(\lambda_I)_{\tau_I}$ with $\min \supp \pi_k \geq \max \supp \pi_{k+1}$ for all $k = 1, \dots, d-1$. The support $\supp \pi$ of an LS-path $\pi = (\sigma_p, \dots, \sigma_1; 0, a_p, \dots, a_1 = 1)$ is the set $\supp \pi = \set{\sigma_p, \dots, \sigma_0}$. The bijections $\Theta^{(I)}_d$ translate this decomposition to the fan of monoids (see \cite[Proposition A.5, Lemma A.8]{seshstratandschubvar}): Every element $\underline a \in \mathrm{LS}_{\lambda_I}^+(d)$ can be uniquely decomposed into a sum $\underline a = \underline a^1 + \dots + \underline a^d$ of elements $\underline a^k \in \mathrm{LS}_{\lambda_I}^+(1)$, such that $\min \supp \underline a^{k} \geq \max \supp \underline a^{k+1}$ holds in $D(\lambda_I, \tau_I)$ for all $k = 1, \dots, d-1$. This property is passed to the LS-fan $\mathrm{LS}_{\underline \lambda}^+$ as well.

\begin{lemma}
    \label{lem:unique_decomposition}
    Every element $\underline a \in \mathrm{LS}_{\underline \lambda}^+$ has a unique decomposition $\underline a = \underline a^1 + \dots + \underline a^s$ into elements $\underline a^k \in \bigcup_{I \in \mathcal I} \mathrm{LS}_{\lambda_I}^+(1) \subseteq \mathrm{LS}_{\underline\lambda}^+$, such that $\min \supp \underline a^{k} \geq \max \supp \underline a^{k+1}$ holds for all $j = 1, \dots, s-1$.
\end{lemma}
\begin{proof}
    Let $I_1 \supsetneq \dots \supsetneq I_j$ be the (unique) chain in $\mathcal I$ containing exactly those $I \in \mathcal I$ where the component $\underline a^{(I)}$ of $\underline a$ is non-zero. Then we have $\min \supp \underline a_{I_k} \geq \max \supp \underline a_{I_{k+1}}$ for all $k = 1, \dots, r-1$. Therefore the existence and uniqueness of the claimed decomposition of $\underline a$ follows from the decomposition in each fan $\mathrm{LS}_{\lambda_I}^+$.
\end{proof}

\begin{proposition}
    \label{prop:bijection_YT_fan}
    Let $\LST$ be the set of all $\tau$-standard LS-tableaux of type $(\underline\lambda, \mathcal I)$ and degree $\underline d \in \N_0^m$. Then the map
    \begin{align*}
        \Theta_{\underline d}: \LST \to \mathrm{LS}_{\underline \lambda}^+(\underline d), \quad (\pi_1, \dots, \pi_s) \mapsto \Theta^{(I_1)}_{1}(\pi_1) + \dots + \Theta_1^{(I_s)}(\pi_s)
    \end{align*}
    is a bijection, where $(\pi_1, \dots, \pi_s)$ is of shape $(\lambda_{I_1}, \dots, \lambda_{I_s})$ for a weakly decreasing sequence $I_1 \supseteq \dots \supseteq I_s$ in $\mathcal I$.
\end{proposition}
\begin{proof}
    By the bijections in~(\ref{eq:bij_LS_paths_fan}) and the $\tau$-standardness of the tableaux, the image of $\Theta_{\underline d}$ is indeed contained in $\mathrm{LS}_{\underline \lambda}^+(\underline d)$, so the map $\Theta_{\underline d}$ is well-defined.

    Now let $\underline a = (\underline a_I)_{I \in \mathcal I} \in \mathrm{LS}_{\underline \lambda}^+(\underline d)$ with the unique decomposition $\underline a = \underline a^1 + \dots + \underline a^s$ from Lemma~\ref{lem:unique_decomposition}. Every element $\underline a^k$ corresponds to an LS-path $\pi_k \in \B(\lambda_{I_k})$ for some $I_k \in \mathcal I$. The associated LS-tableau $\underline \pi^{\underline a} = (\pi_1, \dots, \pi_s)$ has degree $\sum_{k=1}^s e_{I_k} = \deg \underline a = \underline d$ and is $\tau$-standard, since the support of $\underline a$ lies in a maximal chain of $\DCP{}$. The resulting map
    \begin{align*}
        \Theta_{\underline d}^{-1}: \mathrm{LS}_{\underline \lambda}^+(\underline d) \to \LST, \quad \underline a \mapsto \underline \pi^{\underline a}
    \end{align*}
    is inverse to $\Theta_{\underline d}$: By construction, $\Theta_{\underline d}^{-1} \circ \Theta_{\underline d}$ is the identity, so $\Theta_{\underline d}$ is injective. Furthermore, every element $\underline a \in \mathrm{LS}_{\underline \lambda}^+(\underline d)$ is contained in its image, as $\Theta_{\underline d}(\underline \pi^{\underline a}) = \underline a$.
\end{proof}

\section{Filtrations of Demazure modules}

In order to show that the fan of monoids of the multiprojective stratification on $X_\tau$ coincides with the LS-fan $\mathrm{LS}_{\underline\lambda}^+$, we use a special set of functions called \textit{path vectors}, which forms a basis of the leaves $R_{\geq \underline a} / R_{> \underline a}$ associated to the quasi-valuation $\mathcal V$. In this section we summarize the definition of path vectors and some of their important properties (see the appendix in~\cite{seshstratandschubvar}). We adapt the notation to our specific situation and only consider dominant weights of the form $d \lambda_I$ for a fixed degree $d \in \N_0$ and index set $I \in \mathcal I$. There exists a canonical filtration on the Demazure module $V(d\lambda_I)_{\tau_I}$ and its dual space $V(d\lambda_I)_{\tau_I}^*$, both indexed by the set $\mathrm{LS}_{\lambda_I}^+(d)$. We refer to \loccit{} for an explicit construction of the vectors $v_{\underline b, \underline \sigma}$ we mention below and the existence of path vectors.

We define a relation $\rhd$ on $\mathrm{LS}_{\lambda_I}^+(d)$ in the following way: Let $\underline a, \underline b$ be two elements in $\mathrm{LS}_{\lambda_I}^+(d)$, $\sigma_1 > \dots > \sigma_p$ be the elements in $\supp \underline a$ and $\kappa_1 > \dots > \kappa_q$ be the elements in $\supp \underline b$. The relation $\rhd$ then is defined by
\begin{align*}
    \underline a \rhd \underline b \quad \Longleftrightarrow \quad &\sigma_1 > \kappa_1 \ \text{or} \ (\sigma_1 = \kappa_1 \ \text{and} \ a_{\sigma_1} > b_{\kappa_1}) \ \text{or} \\
    &(\sigma_1 = \kappa_1 \ \text{and} \ a_{\sigma_1} = b_{\kappa_1} \ \text{and} \ \sigma_2 > \kappa_2) \ \text{or} \\
    &(\sigma_1 = \kappa_1 \ \text{and} \ a_{\sigma_1} = b_{\kappa_1} \ \text{and} \ \sigma_2 > \kappa_2 \ \text{and} \ a_{\sigma_2} = b_{\kappa_2}) \ \text{or} \ \dots.
\end{align*}
We write $\underline a \unrhd \underline b$ if $\underline a = \underline b$ or $\underline a \rhd \underline b$. This relation coincides with the definition from~\cite[Definition 6.1]{seshstratandschubvar}.

Recall that the quasi-valuation $\mathcal V$ of the stratification on $X_\tau$ depends on the choice of a total order $\geq^t$ on $\DCP{}$ linearizing the partial order $\succeq$. We also denote the associated lexicographic order on $\Q^{D(\underline\lambda, \tau)}$ by $\geq^t$. Note that the relation $\unrhd$ has the following property for all $\underline a, \underline b \in \mathrm{LS}_{\lambda_I}^+(d)$: If $\underline a \unrhd \underline b$, then we have $\underline a \geq^t \underline b$ for every possible choice of the total order $\geq^t$ on $\DCP{}$.

To each element $\underline a \in \mathrm{LS}_{\lambda_I}^+(d)$ and a reduced decomposition $\underline\sigma$ of the unique maximal element $\sigma$ in the support of $\underline a$ one can associate a vector $v_{\underline a, \underline \sigma} \in V(d\lambda_I)_{\tau_I}$ of weight $\wt \underline a$. As $\sigma$ is not an element of the Weyl group itself, a reduced decomposition of $\sigma$ is a decomposition $\min_B(\sigma) = s_{\alpha_{i_1}} \cdots s_{\alpha_{i_\ell}}$ into simple reflections with $\ell$ minimal. When fixing a reduced decomposition $\underline\sigma^{\underline a}$ of $\max \supp \underline a$ for each element $\underline a \in \mathrm{LS}_{\lambda_I}^+(d)$, then the set $\set{v_{\underline a, \underline \sigma^{\underline a}} \mid \underline a \in \mathrm{LS}_{\lambda_I}^+(d)}$ is a basis of the Demazure module $V(d\lambda_I)_{\tau_I}$. This basis does depend on the chosen reduced decompositions, but there is a canonical filtration on $V(d\lambda_I)_{\tau_I}$ via the subspaces
\begin{align*}
    V(d\lambda_I)_{\tau_I, \unlhd \underline a} &= \left\langle v_{\underline b, \underline \sigma} \ \middle\vert \ \text{$\underline b \in \mathrm{LS}_{\lambda_I}^+(d)$, $\underline a \unrhd \underline b$, $\underline\sigma$ reduced decomposition of $\operatorname{max} \supp \underline b$} \right\rangle_\K, \\
    V(d\lambda_I)_{\tau_I, \lhd \underline a} &= \left\langle v_{\underline b, \underline \sigma} \ \middle\vert \ \text{$\underline b \in \mathrm{LS}_{\lambda_I}^+(d)$, $\underline a \rhd \underline b$, $\underline\sigma$ reduced decomposition of $\operatorname{max} \supp \underline b$}  \right\rangle_\K.
\end{align*}
For each $\underline a \in \mathrm{LS}_{\lambda_I}^+(d)$ the subquotient $V(d\lambda_I)_{\tau_I, \unlhd \underline a} / V(d\lambda_I)_{\tau_I, \lhd \underline a}$ is one-dimensional.

The language of path vectors gives rise to a similar filtration on the dual space $V(d \lambda_I)_{\tau_I}^*$. 

\begin{definition}[{\cite[Definition 6.4]{seshstratandschubvar}}]
    A \textbf{path vector} to an element $\underline a \in \mathrm{LS}_{\lambda_I}^+(d)$ is a weight vector $p_{\underline a} \in V(d \lambda_I)_{\tau_I}^*$ of weight ($-\wt \underline a$), such that
    \begin{enumerate}[label=(\alph{enumi})]
        \item there exists a reduced decomposition $\underline \sigma$ of $\sigma = \max \supp \underline a$ with $p_{\underline a}(v_{\underline a, \underline \sigma}) = 1$;
        \item for each $\underline a' \in \mathrm{LS}_{\lambda_I}^+(d)$ and all reduced decompositions $\underline \sigma'$ of $\sigma' = \max \supp \underline a'$, $p_{\underline a}(v_{\underline a', \underline \sigma'}) \neq 0$ implies $\underline a' \unrhd \underline a$.
    \end{enumerate}
\end{definition}

In \cite{littelmann1998contracting}, Littelmann used quantum Frobenius splitting to associate a canonical function to every element $\underline a \in \mathrm{LS}_{\lambda_I}^+(d)$, which he called path vector. It satisfies the above conditions, hence the definition we use here is more general and there exists a path vector to each element in $\mathrm{LS}_{\lambda_I}^+(d)$.

Again, one obtains a basis of the dual module $V(d\lambda_I)_{\tau_I}^*$ by fixing a path vector to each element $\underline a \in \mathrm{LS}_{\lambda_I}^+(d)$ and the subspaces
\begin{align*}
    V(d\lambda_I)_{\tau_I, \unrhd \underline a}^* &= \left\langle p_{\underline b} \ \middle\vert \ \text{$p_{\underline b}$ path vector to some $\underline b \in \mathrm{LS}_{\lambda_I}^+(d)$ with $\underline b \unrhd \underline a$} \right\rangle_\K, \\
    V(d\lambda_I)_{\tau_I, \rhd \underline a}^* &= \left\langle p_{\underline b} \ \middle\vert \ \text{$p_{\underline b}$ path vector to some $\underline b \in \mathrm{LS}_{\lambda_I}^+(d)$ with $\underline b \rhd \underline a$} \right\rangle_\K.
\end{align*}
define a filtration on $V(d\lambda_I)_{\tau_I}^*$, such that the subquotient $V(d\lambda_I)_{\tau_I, \unrhd \underline a}^* / V(d\lambda_I)_{\tau_I, \rhd \underline a}^*$ is one-dimensional for each $\underline a \in \mathrm{LS}_{\lambda_I}^+(d)$. Any path vector $p_{\underline a}$ to $\underline a$ defines a non-zero element of this subquotient.

\section{The quasi-valuation of a path vector}

Throughout this section we use the notation from Remark~\ref{rem:LS-fan_induced_strat} for the LS-fans $\mathrm{LS}_{\lambda_I, \theta_I}^+ \subseteq \mathrm{LS}_{\lambda_I, \tau_I}^+$ of an induced stratification. By Lemma~\ref{lem:graded_components}, every path vector to an element $\underline a \in \mathrm{LS}_{\lambda_I}^+(d)$ can be seen as a multihomogeneous function in $R = \K[X_\tau]$ of degree $d e_I$. Their vanishing behaviour can be described combinatorially in the following way.

\begin{lemma}
    \label{lem:path_vectors_prop}
    A path vector $p_{\underline a}$ to an element $\underline a \in \mathrm{LS}_{\lambda_I}^+(1)$ vanishes identically on the multicone $\hat X_{(\phi, I)}$ to $(\phi, I) \in \DCP{}$, if and only if the unique maximal element $\theta$ in $\supp \underline a \subseteq D(\lambda_I, \tau_I)$ is not less or equal to $\phi_I = \pi_{P_I}(\phi)$.
\end{lemma}
\begin{proof}
    It suffices to prove this statement for the affine cone $\hat X_\phi \subseteq V(\lambda_I)_{\tau_I}$ instead of the multicone $\hat X_{(\phi, I)} \subseteq \prod_{i=1}^m V(\lambda_i)_{\tau_i}$. The same equivalence then also follows for the multicone via the diagram~(\ref{eq:diagram_graded_iso_2}) for $\underline d = e_I$.

    The Demazure module $V(\lambda_I)_\phi$ is equal to the linear span of the affine cone $\hat X_\phi$. Since the path vector $p_{\underline a}$ is linear, it vanishes identically on $\hat X_\phi$, if and only if it vanishes on every vector of the form $v_{\underline a', \underline\sigma'} \in V(\lambda_I)_\phi$ for $\underline a' \in \mathrm{LS}_{\lambda_I, \phi_I}^+(1)$ and a reduced decomposition $\underline\sigma'$ of $\sigma' = \max \supp \underline a'$.

    If $\theta \leq \phi_I$, then the vector $v_{\underline a, \underline\theta}$ is contained in $\mathrm{LS}_{\lambda_I, \phi_I}^+(1)$ for every reduced decomposition $\underline \theta$ of $\theta$. Hence $p_{\underline a}$ does not vanish on $\hat X_\phi$. Conversely, if $p_{\underline a}(v_{\underline a', \underline\sigma'}) \neq 0$ for some element $\underline a' \in \mathrm{LS}_{\lambda_I, \phi_I}^+(1)$, then we have $\underline a' \unrhd \underline a$ by the definition of a path vector, \ie the maximal element $\sigma'$ in $\supp \underline a'$ is larger or equal to $\theta$. But as $\underline a' \in \mathrm{LS}_{\lambda_I, \phi_I}^+$, it follows $\phi_I \geq \sigma' \geq \theta$.
\end{proof}

For each total order $\geq^t$ on $\DCP{}$ and every fixed element $(\theta, I) \in \DCP{}$ we have an induced total order on the underlying poset $\DCP{}_{\preceq (\theta, I)}$ of the stratification on $\hat X_{(\theta, I)}$. Let $\mathcal V_{(\theta, I)}$ denote the associated quasi-valuation on $\K[X_{(\theta, I)}]$. Path vectors are compatible with the induced stratification, which allows the use of inductive arguments.

\begin{lemma}
    \label{lem:path_vector_induced_strat}
    Let $p_{\underline a}$ be a path vector to an element $\underline a \in \mathrm{LS}_{\lambda_I}^+(1) \subseteq \mathrm{LS}_{\underline\lambda}^+(e_I)$ and $(\theta, I) = \max \supp \underline a \in \DCP{}$. Then the restriction of $p_{\underline a}$ to the multicone $\hat X_{(\theta, I)}$ is a path vector to the element $\underline a \in \mathrm{LS}_{\lambda_I, \theta_I}^+(1)$ and it holds 
    \begin{align*}
        \mathcal V(p_{\underline a}) = \mathcal V_{(\theta, I)}(p_{\underline a}\big|_{\hat X_{(\theta, I)}}).
    \end{align*}
\end{lemma}
\begin{proof}
    By Lemma 6.6 in \cite{seshstratandschubvar}, the restriction of $p_{\underline a} \in V(\lambda_I)^*_{\tau}$ to $V(\lambda_I)_\theta$ is a path vector associated to $\underline a \in \mathrm{LS}_{\lambda_I, \theta_I}^+(1)$. The corresponding function in $\K[\hat X_{(\theta, I)}]$ coincides with the restriction of the function $p_{\underline a} \in \K[\hat X_\tau]$ to the subvariety $\hat X_{(\theta, I)} \subseteq \hat X_\tau$. This can be shown via the diagram~(\ref{eq:diagram_graded_iso_2}). 

    If $\underline \theta_I$ is any reduced decomposition of $\theta_I \in W/W_{P_I}$, then $v_{\underline a, \underline\theta_I}$ lies in the Dema\-zure module $V(\lambda_I)_{\theta_I}$. Hence $p_{\underline a}$ restricts to a non-zero element in $\K[\hat X_{(\theta, I)}]$ and we have  
    \begin{align*}
        \mathcal V_{(\theta, I)}(p_{\underline a}\big|_{\hat X_{(\theta, I)}}) = \operatorname{min} \set{\mathcal V_{\mathfrak C}(p_{\underline a}) \mid \text{$\mathfrak C$ max.\;chain in $\DCP{}$}, (\theta, I) \in \mathfrak C}
    \end{align*}
    by the definition of the quasi-valuation. 

    Let $\mathfrak C$ be any maximal chain in $\DCP{}$ and $(\phi, J)$ be the minimal element in $\mathfrak C$, such that the path vector $p_{\underline a}$ does not vanish identically on $X_{(\phi, J)}$. This means that the coefficient of the basis vector $e_{(\phi, J)}$ in $\mathcal V_{\mathfrak C}(p_{\underline a})$ is positive. We now show $(\phi, J) \succeq (\theta, I)$, since this implies $\mathcal V_{\mathfrak C}(p_{\underline a}) \geq^t \mathcal V_{(\theta, I)}(p_{\underline a}\big|_{\hat X_{(\theta, I)}})$ for every choice of the total order $\geq^t$. Therefore the quasi-valuation is given by $\mathcal V(p_{\underline a}) = \mathcal V_{(\theta, I)}(p_{\underline a}\big|_{\hat X_{(\theta, I)}})$.
    
    As $p_{\underline a}$ has degree $e_I$ in $\K[\hat X_\tau]$, this implies $\underline I \subseteq J$, hence we have $I \subseteq J$ by the requirement~(\ref{eq:need_for_s2}) on the poset $\mathcal I$. By Corollary~\ref{cor:dcp_go_down}, the element $(\phi^\triangledown, I)$ with $\phi^\triangledown = \min_Q \circ \pi_{Q_I}(\phi)$ is less or equal to $(\phi, J)$ in $\DCP{}$. The elements $\phi$ and $\phi^\triangledown$ are equal in $W/W_{P_I}$, hence the images of the multicones $\hat X_{(\phi, J)}$ and $\hat X_{(\phi^\triangledown, I)}$ under the projection map
    \begin{align*}
        \prod_{i \in J} V(\lambda_i)_{\tau_i} \twoheadrightarrow \prod_{i \in \underline I} V(\lambda_i)_{\tau_i}
    \end{align*}
    coincide. We assumed that $p_{\underline a}$ does not vanish identically on $\hat X_{(\phi, J)}$, so it does not vanish identically on $\hat X_{(\phi^\triangledown, I)}$ as well. It now follows $(\theta, I) \preceq (\phi^\triangledown, I)$ from Lemma~\ref{lem:path_vectors_prop}. This completes the proof.
\end{proof}

\begin{proposition}
    \label{prop:quasi_valuation_path_vector}
    Let $p_{\underline a} \in R$ be a path vector to an element $\underline a \in \mathrm{LS}_{\lambda_I}^+(1)$ for some $I \in \mathcal I$. Then $\mathcal V(p_{\underline a}) = \underline a$ holds independent of the chosen total order $\geq^t$ on $\DCP{}$.
\end{proposition}
\begin{proof}
    We prove the statement by induction over the rank $r$ of $(\tau, [m])$ in the defining chain poset. The case $r = 0$ is trivial. If $r \geq 1$, we consider two cases. First, when the maximal element $(\theta, I)$ in the support of $\underline a$ is strictly smaller than $(\tau, [m])$, we use the induced stratification on $\hat X_{(\theta, I)}$ and Lemma~\ref{lem:path_vector_induced_strat} to conclude
    \begin{align*}
        \mathcal V(p_{\underline a}) = \mathcal V_{(\theta, I)}(p_{\underline a}\big|_{\hat X_{(\theta, I)}}) = \underline a
    \end{align*}
    by induction. Notice, that this holds independent of the chosen total order $\geq^t$.

    Now suppose that $(\tau, [m])$ is the maximal element in $\supp \underline a$. Here we have $I = [m]$. Fix a positive integer $N \in \N$ with $Nu \in \N$, where $u \in \Q_{\geq 0}$ is the coefficient of the basis vector $e_{(\tau, [m])}$ in $\underline a$. We use Corollary C.13 in \cite{seshstratandschubvar} on the element $\underline a \in \mathrm{LS}_{\lambda_I}^+$: It states that $p_{\underline a}^N$ is -- up to multiplying by a non-zero scalar -- equal to $p_\tau^{Nu} p_{\underline b} \in V(N \lambda_I)_{\tau_I}^*$, where $p_\tau$ is a weight vector in $V(\lambda_I)^*$ of weight $-\tau(\lambda_I)$ and $p_{\underline b}$ is a path vector associated to the element
    \begin{align*}
        \underline b = N \underline a - N u \, e_{\tau_I} \in \mathrm{LS}_{\lambda_I}^+(s)
    \end{align*}
    of degree $s = Nd - Ndu$. The pullback of $p_\tau$ to the multicone $\hat X_\tau$ is the extremal function $f_{(\tau, [m])}$ (up to a non-zero scalar), so we have $\mathcal V(f_{(\tau, [m])}^{Nu}) = Nu \, e_{(\tau, [m])}$, which holds independent of the choice of $\geq^t$. 

    For $s = 0$ the path vector $p_{\underline b}$ is constant and it follows
    \begin{align*}
        \mathcal V(p_{\underline a}) &= \frac{1}{N} \mathcal V(p_{\underline a}^N) = \frac{1}{N} \mathcal V(f_{(\tau, [m])}^{Nu}) = u \, e_{(\tau, [m])} = \underline a.
    \end{align*}
    Now we assume $s \geq 1$. As the element $\underline b$ might not be of degree one, we have to write it in terms of path vectors of degree one to use the induction. Therefore we fix a path vector $\overline p_{\underline c}$ to each $\underline c \in \mathrm{LS}_{\lambda_I}^+(1)$. This defines a function $g_{\underline c}$ for every element $\underline c \in \mathrm{LS}_{\lambda_I}^+(s)$: To each $\underline c^k$ in the unique decomposition $\underline c = \underline c^1 + \dots + \underline c^s$ from Lemma~\ref{lem:unique_decomposition} we have the path vector $\overline p_{\underline c^k}$. 
    It was shown in~\cite[Proposition C.10]{seshstratandschubvar} that the product $g_{\underline c} \coloneqq \overline p_{\underline c^1} \cdots \overline p_{\underline c^s}$ in $\bigotimes_{i=1}^s V(\lambda_I)_{\tau_I}$ restricts to a path vector associated to $\underline c$, up to multiplying by a root of unity.
    
    Using the filtration of $V(s\lambda_I)_{\tau_I}^*$ via the subspaces $V(s\lambda_I)_{\tau_I, \unrhd \underline c}^*$, we can write the path vector $p_{\underline b}$ as a linear combination $p_{\underline b} = g_{\underline b} + \sum_{\underline c \rhd \underline b} d_{\underline c} \, g_{\underline c}$ over elements $\underline c \in \mathrm{LS}_{\lambda_I}^+(s)$ with $\underline c \unrhd \underline b$. We now show
    \begin{equation}
        \label{eq:quasi_val_strictly_greater}
        \mathcal V(g_{\underline c}) > \mathcal V(g_{\underline b}) \quad \text{for all} \ \underline c \rhd \underline b.
    \end{equation}
    Let $\sigma' = \max \supp \underline c$ and $\sigma = \max \supp \underline b$. We need to distinguish between two cases. 

    If $\sigma' = \sigma$ we have $\max \supp \underline c^k \preceq (\sigma, I) \prec (\tau, [m])$ for each $k \in [s]$ and thus $\mathcal V(\overline p_{\underline c^k}) = \underline c^k$ by induction. Because the union of all $\supp \underline c^k$ for $k \in [s]$ lies in a maximal chain of $\DCP{}$, the quasi-valuation is additive:
    \begin{align*}
        \mathcal V(g_{\underline c}) = \mathcal V(\overline p_{\underline c^1} \cdots \overline p_{\underline c^s}) = \underline c^1 + \dots + \underline c^s = \underline c.
    \end{align*}
    This is independent of the choice of $\geq^t$. Analogously, we see $\mathcal V(g_{\underline b}) = \underline b$. Since $\underline c \rhd \underline b$, it holds $\mathcal V(g_{\underline c}) = \underline c \geq^t \underline b = \mathcal V(g_{\underline b})$ for every choice of the total order $\geq^t$ on $\DCP{}$.
    
    In the remaining case $\sigma' \neq \sigma$ we have $\sigma' > \sigma$, as $\underline c \rhd \underline b$. It follows from Lemma~\ref{lem:path_vectors_prop} that $g_{\underline c}$ does not vanish identically on $\hat X_{(\sigma', I)}$, but it restricts to the zero function on the multicone $\hat X_{(\phi, I)}$ for each $\phi < \sigma'$ in $D(\lambda_I, \tau_I)$. The function $g_{\underline c}$ also vanishes identically on every multicone $\hat X_{(\phi, J)}$ for $(\phi, J) \prec (\sigma', I)$ in $\DCP{}$ and $J \subsetneq I$, because $g_{\underline c}$ is homogeneous of degree $e_I$ and $\underline I \nsubseteq J$. Hence $(\sigma', I)$ lies in the support of $\mathcal V(g_{\underline c})$. Analogously, one can show that $\sigma$ is the maximal element in $\supp \mathcal V(g_{\underline b})$. Therefore $\mathcal V(g_{\underline c}) > \mathcal V(g_{\underline b})$.
    
    Using the inequality (\ref{eq:quasi_val_strictly_greater}) we can now conclude $\mathcal V(p_{\underline b}) = \mathcal V(g_{\underline b}) = \underline b$. Hence the set $\set{(\tau, [m])} \cup \supp \mathcal V(p_{\underline b})$ lies in a maximal chain of $\DCP{}$ and it follows
    \begin{align*}
        \mathcal V(p_{\underline a}) &= \frac{1}{N} \mathcal V(p_{\underline a}^N) = \frac{1}{N} \mathcal V(f_{(\tau, [m])}^{Nu} p_{\underline b}) = \frac{1}{N} \big( \mathcal V(f_{(\tau, [m])}^{Nu}) + \mathcal V(p_{\underline b}) \big) \\
        &= u \, e_{(\tau, [m])} + (\underline a - u \, e_{(\tau, [m])}) = \underline a.
    \end{align*}
    This is independent of the choice of the total order $\geq^t$.
\end{proof}

\section{Standard monomial theory}
\label{sec:smt}

Recall that we still consider a fixed multiprojective Seshadri stratification on a Schubert variety $X_\tau$ as in Theorem~\ref{thm:stratification}. For every $I \in \mathcal I$ the indecomposable elements in the fan $\mathrm{LS}_{\lambda_I}^+$ are exactly the elements of degree one. In case of our generalized stratification, however, not every indecomposable element in $\Gamma$ needs to be of total degree one. Instead, it follows from Lemma~\ref{lem:unique_decomposition} that the set of indecomposables in the LS-fan $\mathrm{LS}_{\underline\lambda}^+$ is equal to
\begin{align*}
    \mathbb G = \bigcup_{I \in \mathcal I} \mathrm{LS}_{\lambda_I}^+(1).
\end{align*}
We therefore fix a path vector $p_{\underline a} \in R$ to each element $\underline a \in \mathbb G$ and let $\mathbb G_R = \set{p_{\underline a} \mid \underline a \in \mathbb G}$ be the set of all these functions.

\begin{definition}
    \label{def:standard_monomial}
    A monomial $p_{\underline a^1} \cdots p_{\underline a^s} \in R$ of path vectors in $\mathbb G_R$ is called \textbf{standard}, if $\min \supp \underline a^{k} \geq \max \supp \underline a^{k+1}$ holds for all $k = 1, \dots, s-1$.
\end{definition}

Another way to characterize standardness comes from the language of LS-tableaux: Let $p_{\underline a^1} \cdots p_{\underline a^s}$ be a monomial of path vectors in $\mathbb G_R$ with $\underline a^k \in \mathrm{LS}_{\lambda_{I_k}}^+(1)$. Each element $\underline a^k$ corresponds to an LS-path $\pi_k$ of shape $\lambda_{I_k}$. The monomial $p_{\underline a^1} \cdots p_{\underline a^s}$ is standard, if and only if the LS-tableau $\underline\pi = (\pi_1, \dots, \pi_s)$ is of type $(\underline\lambda, \mathcal I)$ and $\tau$-standard.

By Lemma~\ref{lem:unique_decomposition} we have an associated standard monomial
\begin{align*}
    p_{\underline a} \coloneqq p_{\underline a^1} \cdots p_{\underline a^s}
\end{align*}
to every element $\underline a \in \mathrm{LS}_{\underline\lambda}^+$, where $\underline a = \underline a^1 + \dots + \underline a^s$ is the unique decomposition into elements $\underline a \in \mathbb G$ with $\min \supp \underline a^{k} \geq \max \supp \underline a^{k+1}$ holds for all $k = 1, \dots, s-1$. Conversely, every standard monomial in $\mathbb G_R$ is of the form $p_{\underline a}$ for some $\underline a \in \mathrm{LS}_{\underline\lambda}^+$. The monomial $p_{\underline a}$ is a multihomogeneous function in $R$ of degree $\sum_{k=1}^s \deg \underline a^k = \deg \underline a$.

\begin{theorem}
    \label{thm:standard_monomial_basis}
    $ $
    \begin{enumerate}[label=(\alph{enumi})]
        \item For each $\underline a \in \mathrm{LS}_{\underline\lambda}^+$ holds $\mathcal V(p_{\underline a}) = \underline a$. Additionally, the set of all standard monomials in $\mathbb G_R$ forms a vector space basis of $R = \K[\hat X_\tau]$.
        \item The fan of monoids to this Seshadri stratification coincides with the LS-fan $\mathrm{LS}_{\underline\lambda}^+$. The stratification is balanced and of LS-type. In particular, it is normal.
    \end{enumerate}
\end{theorem}
\begin{proof}
    $ $
    \begin{enumerate}[label=(\alph{enumi})]
        \item For each standard monomial $p_{\underline a} = p_{\underline a^1} \cdots p_{\underline a^s}$ the set $\supp \underline a^1 \cup \dots \cup \supp \underline a^s$ is contained in a maximal chain of $\DCP{}$, so the quasi-valuation is additive and Proposition~\ref{prop:quasi_valuation_path_vector} implies 
        \begin{align*}
            \mathcal V(p_{\underline a}) = \mathcal V(p_{\underline a^1} \cdots p_{\underline a^s}) = \underline a^1 + \dots + \underline a^s = \underline a.
        \end{align*}
        Now fix a tuple $\underline d \in \N_0^m$. The set $\set{p_{\underline a} \mid \underline a \in \mathrm{LS}_{\underline \lambda}^+(\underline d)}$ is linearly independent in the graded component $\K[\hat X_\tau]_{\underline d}$, since the quasi-valuation $\mathcal V$ is injective on it. By Lemma~\ref{lem:graded_components}, the cardinality of this set is therefore bounded by the dimension of the Demazure module $V(\underline d \cdot \underline \lambda)_\tau$. On the other hand, we have seen in Proposition~\ref{prop:bijection_YT_fan}, that there is a bijection between $\mathrm{LS}_{\underline \lambda}^+(\underline d)$ and the set $\LST$ of all $\tau$-standard LS-tableaux of type $(\underline\lambda, \mathcal I)$ with degree $\underline d$. The degree of an LS-tableau is determined by its shape and there always exists at least one shape to each degree (see Remark~\ref{rem:shape_to_degree}). For each subset of LS-tableaux in $\LST$ of a fixed shape $(\lambda_{I_1}, \dots, \lambda_{I_s})$, we have the Demazure-type character formula from equation~(\ref{eq:comb_demazure_character_formula_multi}), so the size of this subset is equal to the dimension of $V(\underline d \cdot \underline \lambda)_\tau$. In total, we get the following inequalities:
        \begin{align}
            \label{eq:surprise}
            \dim V(\underline d \cdot \underline \lambda)_\tau \leq \vert \mkern1mu \LST \vert = \vert \mkern1mu \mathrm{LS}_{\underline \lambda}^+(\underline d) \vert \leq \dim V(\underline d \cdot \underline \lambda)_\tau.
        \end{align}
        Therefore the standard monomials of degree $\underline d$ form a basis of $\K[\hat X_\tau]_{\underline d}$.
        \item As the standard monomials in $\mathbb G_R$ form a basis of $R$, their image under $\mathcal V$ agrees with the fan of monoids $\Gamma$, hence $\Gamma = \mathrm{LS}_{\underline\lambda}^+$. This also implies that the stratification is of LS-type, since the other two requirements are fulfilled by construction. Furthermore, we have seen that the quasi-valuation $\mathcal V(p_{\underline a}) = \underline a$ of each standard monomial $p_{\underline a}$ does not depend on the choice of the total order $\geq^t$ on $\DCP{}$, so the stratification is also balanced.\hfill\qedhere
    \end{enumerate}
\end{proof}

\begin{corollary}
    \label{cor:unique_shape}
    $ $
    \begin{enumerate}[label=(\alph{enumi})]
        \item \label{itm:unique_shape_a} To each $\underline d \in \N_0^m$ there exists exactly one weakly decreasing sequence $I_1 \supseteq \dots \supseteq I_s$ in $\mathcal I$ with $\sum_{k=1}^s e_{I_k} = \underline d$.
        \item For every $I \in \mathcal I$ and $d \in \N_0$ it holds $\mathrm{LS}_{\lambda_I}^+(d) = \mathrm{LS}_{\underline\lambda}^+(d e_I)$.
    \end{enumerate}
\end{corollary}
\begin{proof}
    Statement~\ref{itm:unique_shape_a} follows immediately from the inequalities in (\ref{eq:surprise}) and we have already seen the inclusion $\mathrm{LS}_{\lambda_I}^+(d) \subseteq \mathrm{LS}_{\underline\lambda}^+(d e_I)$. For the reverse inclusion let $\underline a$ be an element in $\mathrm{LS}_{\underline\lambda}^+(d e_I)$ and write it in the form $\underline a = \underline a^{(I_1)} + \dots + \underline a^{(I_m)}$, where $I_1 \subsetneq \dots \subsetneq I_m$ is a maximal chain in $\mathcal I$ and $\underline a^{(I_j)} \in \mathrm{LS}_{\lambda_{I_j}}^+$ for all $j \in [m]$. There exist non-negative integers $k_1, \dots, k_m$ such that $\underline a^{(I_j)}$ is of degree $k_j$ in $\mathrm{LS}_{\lambda_{I_j}}^+$. Since $\sum_{j=1}^m k_j e_{I_j} = d e_I$ it now follows from part~\ref{itm:unique_shape_a} that $I = I_j$ for some $j \in [m]$ and $k_j = d$. Hence $\underline a = \underline a^{(I_j)} \in \mathrm{LS}_{\lambda_{I}}^+(d)$.
\end{proof}

This completes the goal of this \whatisthis{}. We have seen that there exists a normal and balanced Seshadri stratification on each multiprojective Schubert variety $X_\tau$. The elements in its fan of monoids correspond to $\tau$-standard LS-tableaux of type $(\underline\lambda, \mathcal I)$. All tableaux of a fixed degree $\underline d \in \N_0^m$ have the same shape $(\lambda_{I_1}, \dots, \lambda_{I_s})$. The decomposition of $\underline a \in \mathrm{LS}_{\underline\lambda}^+(\underline d)$ into indecomposable elements corresponds exactly to the columns $\pi_k \in \mathbb B(\lambda_{I_k})$, $k \in [s]$, in the corresponding LS-tableau to $\underline a$. We have a standard monomial theory on $\K[X_\tau]$ determined by the indecomposable elements $\mathbb G$. Each non-standard monomial in $\mathbb G_R$ can be written as a linear combination of standard monomials via a straightening relation as in Proposition~\ref{prop:smt}. This SMT is also compatible with all Schubert varieties, which appear as strata (cf. Corollary~\ref{cor:smt_induced_strat}).

\begin{corollary}
    \label{cor:smt_LS}
    Let $\LST$ be the set of all $\tau$-standard LS-tableaux $\underline\pi$ of type $(\underline\lambda, \mathcal I)$ and degree $\underline d \in \N_0^m$.
    \begin{enumerate}[label=(\alph{enumi})]
        \item \label{itm:smt_LS_b} The standard monomials in $\mathbb G_R$ of degree $\underline d \in \N_0^m$ form a basis of the module $V(\underline d \cdot \underline\lambda)^*$, indexed by $\LST$.
        \item Suppose that the fixed path vectors $p_{\underline a}$ for $\underline a \in \mathbb G$ are constructed as in~\cite{littelmann1998contracting} via quantum Frobenius splitting. Then the standard monomial basis from part~\ref{itm:smt_LS_b} coincides with the standard monomial basis from Section~6 in \loccit{}.
    \end{enumerate}
\end{corollary}
\begin{proof}
    The first statement is a consequence of Theorem~\ref{thm:standard_monomial_basis} and Proposition~\ref{prop:bijection_YT_fan}. To show the second statement, we fix the unique weakly decreasing sequence $I_1 \supseteq \dots \supseteq I_s$ in $\mathcal I$ with $e_{I_1} + \dots + e_{I_s} = \underline d$. Then every LS-tableau in $\LST$ is of shape $(\lambda_{I_1}, \dots, \lambda_{I_s})$. We abbreviate these dominant weights by $\lambda_k \coloneqq \lambda_{I_k}$. For each LS-path $\pi \in \mathbb B(\lambda_k)$, $k \in [s]$, let $p_{\pi} \in V(\lambda_k)_{\tau}^*$ denote the path vector constructed via quantum Frobenius splitting (see \cite[Section 3]{littelmann1998contracting}). By Proposition~\ref{prop:tau_standard_vs_path_model}, a monomial $p_{\pi_1} \cdots p_{\pi_s}$ in these path vectors is standard in the sense of \loccit{}, if and only if it is standard in the sense of Definition~\ref{def:standard_monomial}.
\end{proof}

\vskip 25pt

\appendix

\section{Weyl groups}
\label{sec:weyl_groups}

This section contains a brief summary of the notation in this~\whatisthis{} and a loose collection of lemmata. All statements which we do not prove here are well known and can be found in any classical text book about Coxeter groups, for example in~\cite{bjorner2006combinatorics}.

We fix a semisimple algebraic group $G$ with Weyl group $W$, a maximal torus $T$ of $G$ and a Borel subgroup $B$ containing $T$. For each parabolic subgroup $Q \subseteq G$ -- by which we mean a closed subgroup containing $B$ -- with Weyl subgroup $W_Q \subseteq W$ and $\sigma \in W$, there is a (unique) smallest element $\sigma^Q \in W$ in the coset $\sigma W_Q$. It has the property $\ell(\sigma^Q \tau) = \ell(\sigma^Q) + \ell(\tau)$ for all $\tau \in W_Q$. We denote the set of these smallest elements by
\begin{align*}
    W^Q = \set{\sigma^Q \in W \mid \sigma \in W}.
\end{align*}
Thus the product map $W^Q \times W_Q \to W$ is a length-preserving bijection. More general: For any two parabolic subgroups $Q \subseteq Q'$ the product map 
\begin{align}
    \label{eq:product_map_weyl}
    W^{Q'} \times (W_{Q'} \cap W^Q) \to W^Q, \quad (\sigma, \tau) \mapsto \sigma \tau
\end{align}
is a length-preserving bijection. The quotient $W/W_Q$ is a graded poset, \ie all maximal chains have the same length. The rank function $r: W/W_Q \to \N_0$ maps a coset $\theta \in W/W_Q$ to the length $\ell(\theta^Q)$ of its unique representative $\theta^Q \in W^Q$, \ie the smallest number $\ell \in \N_0$ such that there exists a decomposition $\theta^Q = s_1 \cdots s_\ell$ into simple reflections. Such a minimal decomposition is usually called a \textit{reduced decomposition}.

To every inclusion $Q \subseteq Q'$ of two parabolic subgroups there is the monotone surjection\label{eq:def_pi_Q}
\begin{align*}
    \pi_{Q,Q'} \!: \: W/W_Q \twoheadrightarrow W/W_{Q'}, \ \sigma W_Q \mapsto \sigma W_{Q'},
\end{align*}
where we typically write $\pi_{Q'}$ instead, if the source is clear. Every element $\theta \in W/W_{Q'}$ has a unique minimal preimage $\min_Q(\theta)$ and a unique maximal preimage $\max_Q(\theta)$ in $W/W_Q$ under $\pi_{Q'}$. The corresponding two maps\label{eq:def_min_Q_max_Q}
\begin{align*}
    \mathrm{min}_Q \!: W/W_{Q'} &\hookrightarrow W/W_{Q}, \ \theta \mapsto \mathrm{min}_Q(\theta) \quad \text{and}\\
    \mathrm{max}_Q \!: W/W_{Q'} &\hookrightarrow W/W_{Q}, \ \theta \mapsto \mathrm{max}_Q(\theta),
\end{align*}
are isomorphisms of posets onto their image. 

\begin{definition}
    We say an element $\sigma W_Q \in W/W_Q$ is \textbf{\boldmath{}$Q'$-minimal}/\textbf{\boldmath{}$Q'$-maximal}, if it lies in the image of $\min_Q$ or $\max_Q$ respectively. 
\end{definition}

In general, it is enough to understand the Weyl group quotients $W/W_P$ for every maximal parabolic subgroup $P$: If $(Q_i)_{i \in I}$ is a finite family of parabolic subgroups containing $B$ and $Q = \bigcap_{i \in I} Q_i$ is their intersection, then the following map is an isomorphism of posets onto its image:
\begin{align*}
    W/W_Q \hookrightarrow \prod_{i \in I} W/W_{Q_i}, \quad \sigma \mapsto (\pi_{Q_i}(\sigma))_{i \in I}.
\end{align*}

\subsection{Dynkin type $\texttt{A}$}
\label{subsec:weyl_type_A}

We briefly fix the notation we use for the simple algebraic group $G = \mathrm{SL}_n(\K)$. Let $T$ and $B$ be the conventional choices for a maximal torus and Borel subgroup in $G$, \ie the torus $T$ of diagonal matrices in $G$ and the Borel subgroup $B$ of all upper triangular matrices with determinant $1$. The Weyl group $W = N_G(T)/C_G(T)$ can be identified with the symmetric group $S_n$. We usually denote a permutation $\sigma: [n] \to [n]$ in the one-line notation $\sigma = \sigma(1) \cdots \sigma(n)$. 

Let $\varepsilon_i: T \to \K^\times$ be the character of $T$, where $\varepsilon_i(t)$ is equal to the $i$-th entry on the diagonal of $t \in T$. To each $i \in [n-1]$ there is the associated fundamental weight $\omega_i = \varepsilon_1 + \dots + \varepsilon_i$ and the maximal parabolic subgroup $P_i = B W_{P_i} B$, where $W_{P_i} \subseteq W$ is the stabilizer of $\omega_i$. It is generated by the simple reflections $s_j = (j, j+1 \in S_n$ with $j \neq i$. As the elements of $W/W_{P_i} \cong S_n / (S_i \times S_{n-i})$ correspond to $i$-element subsets of $[n]$, we also use the one-line notation for the cosets $\sigma W_{P_i} \in W/W_{P_i}$ and write them in the form $\sigma(1) \cdots \sigma(i)$.

The Bruhat order on $W/W_{P_i}$ can be characterized via the one-line notation: For a tuple $\underline j = (j_1, \dots, j_i)$ of natural numbers, let $\underline j^\leq$ be the permuted tuple with weakly increasing entries (from left to right). For all $\phi = \phi(1) \cdots \phi(i), \theta = \theta(1) \cdots \theta(i) \in W/W_{P_i}$ we then have
\begin{align*}
    \phi \leq \theta \quad \Longleftrightarrow \quad (\phi(1), \dots, \phi(i))^\leq \leq (\theta(1), \dots, \theta(i))^\leq,
\end{align*}
where the tuples on the right hand side are compared component-wise. The one-line notation can also be used to describe the Bruhat order on $W$, as $\sigma \leq \tau \in W$ is equivalent to $\pi_{P_i}(\sigma) \leq \pi_{P_i}(\tau)$ for all $i = 1, \dots, n-1$.

The one-line notation is very advantageous in conjunction with the lifting maps $\min_Q$ and $\max_Q$. For example, let $Q = P_1 \cap P_3$ for $n = 4$ and $\theta = 134 \in W/W_{P_3}$ (we omit the brackets of the tuple notation), then the unique maximal lift $\max_Q(\theta)$ in $W/W_Q$ is given by the projection of $\max_B(\theta) \in W$ to $W/W_Q$ and the element $\max_B(\theta)$ is equal to $4312$. More precisely, the image of all $P_3$-maximal elements in $W/W_Q$ under the map $\min_B: W/W_Q \to W$ are all permutations $\sigma = \sigma(1) \cdots \sigma(4)$ with $\sigma(1) > \sigma(2)$, $\sigma(1) > \sigma(3)$ and $\sigma(2) < \sigma(3)$.

\subsection{Some useful lemmas}

The following lemma by Deodhar is used many times throughout this \whatisthis{}. Our version of this statement follows directly from {\cite[Lemma 11.1, Lemma 11.1']{lakshmibaiGP4}} or \cite[Lemma 12.8.9]{lakshmibai2001flag} by projecting the unique minimal/maximal lift in $W$ to $W/W_Q$.

\begin{lemma}[Deodhar's Lemma]
    \label{lem:deodhar}
    Let $Q \subseteq Q'$ be two parabolic subgroups containing $B$ and $\theta \geq \phi$ be two elements of $W/W_{Q'}$.
    \begin{enumerate}[label=(\alph{enumi})]
        \item If $\overline\theta \in W/W_Q$ is a lift of $\theta$, then there is a unique maximal lift $\overline\phi \in W/W_Q$ of $\phi$ such that $\overline\phi \geq \overline\theta$.
        \item If $\overline\phi \in W/W_Q$ is a lift of $\phi$, then there is a unique minimal lift $\overline\theta \in W/W_Q$ of $\theta$ such that $\overline\theta \geq \overline\phi$.
    \end{enumerate}
\end{lemma}

\begin{lemma}
    \label{lem:red_decomp}
    If $s_1, \dots, s_r \in W$ are pairwise distinct simple reflections, then $s_1 \cdots s_r$ is in reduced decomposition.
\end{lemma}
\begin{proof}
    If $s_1 \cdots s_r$ was not reduced, then there are indices $1 \leq i < j \leq r$ with $s_1 \cdots s_r = s_1 \cdots \hat s_i \cdots \hat s_j \cdots s_r$ (where the hat indicates, that $s_i$ and $s_j$ are omitted). By induction over $r$, the prefix $s_1 \cdots s_{r-1}$ is reduced, hence $j = r$. It follows $s_1 \cdots s_{r-1} = s_1 \cdots \hat s_i \cdots s_r$. Both sides of this equation are in reduced decomposition. As $s_1, \dots, s_r$ are pairwise distinct and the set of simple reflections appearing in a reduced decomposition is unique, we conclude $s_i = s_r$, which is impossible.
\end{proof}

\begin{lemma}
    \label{lem:bruhat_interval}
    Let $Q \subseteq P$ be two parabolic subgroups of $G$ and $\theta > \phi \in W/W_Q$, such that $\pi_P(\theta) > \pi_P(\phi)$. Then there exists a covering relation $\theta > \psi$ in $W/W_Q$, such that $\psi \geq \phi$ and $\pi_P(\theta) > \pi_P(\psi)$.
\end{lemma}
\begin{proof} 
    We show the statement by induction over difference $d = r(\theta) - r(\phi)$ of ranks in $W/W_Q$. For $d = 1$ there is nothing to prove. Now let $d \geq 2$ and $\overline\phi$ be the unique maximal lift of $\pi_{P}(\phi)$ in $W/W_{Q}$, that is less or equal to $\pi_{Q}(\theta)$. If $\pi_{Q}(\theta) > \overline\phi$ is already a covering relation, then we can take $\psi = \overline\phi$. Otherwise we look at the Bruhat interval $[\theta, \overline\phi] = \set{\sigma \in W/W_Q \mid \theta \geq \sigma \geq \overline\phi}$. Suppose that every element in this interval except $\overline\phi$ projects to $\pi_P(\theta)$. By Deodhar's Lemma (\ref{lem:deodhar}) there exists a unique minimal lift $\overline\theta \in W/W_Q$ of $\pi_P(\theta)$ with $\overline\theta \geq \overline\phi$. Hence there is exactly one element covering $\overline\phi$ in $[\theta, \overline\phi]$. But this is false for Bruhat intervals of two elements, which lengths differ by more than $1$. A proof of this statement can be found in~\cite[Lemma 2.7.3]{bjorner2006combinatorics}. 
    
    Therefore there exists an element $\phi' \in W/W_Q$, such that $\pi_{Q}(\theta) > \phi' > \overline\phi$ and $\pi_{P}(\theta) > \pi_{P}(\phi') > \pi_{P}(\overline\phi)$. Using the induction on $\phi'$ instead of $\phi$ we get an element $\psi \in W/W_Q$ covered by $\theta$ with $\pi_P(\theta) > \pi_P(\psi)$ and $\psi \geq \phi' \geq \overline\phi \geq \phi$.
\end{proof}

\section{Young-tableaux and other tableau models}
\label{sec:tableaux}

The LS-tableaux from Section~\ref{sec:LS-tableaux} are a generalization of more well known tableau models, like the ones of Hodge-Young in type \texttt{A} and of Lakshmibai-Musili-Seshadri in the types \texttt{B}, \texttt{C} and \texttt{D}. We fix a connected, simply-connected, simple algebraic group $G$, a maximal torus $T$ and a Borel subgroup $B$ containing $T$. Let $\Delta$ be the set of all simple roots corresponding to the choice of $B$. For each Dynkin type we order the fundamental weights $\omega_1, \dots, \omega_n$ of $G$ (equivalently, the simple roots) in the same way as in \cite[Plates I to IX]{bourbaki1968groupes}. Each fundamental weight $\omega_i$ corresponds to the maximal parabolic subgroup $P_i$ stabilizing the highest weight space in $V(\omega_i)$. Furthermore, we fix a dominant weight $\mu = a_1 \omega_1 + \dots + a_{n} \omega_{n}$ for $a_1, \dots, a_n \in \N_0$. There exists a unique sequence $\underline\mu = (\omega_{i_1}, \dots, \omega_{i_s})$ of fundamental weights, such that $\omega_{i_1} + \dots + \omega_{i_s} = \mu$ and $1 \leq i_1 \leq \dots \leq i_s \leq n$. 

\vskip 10pt
{\noindent\textbf{Type \boldmath{}$\texttt{A}_{n}$:}} Let $\omega_i$ and $s_i$ be defined as in Appendix~\ref{subsec:weyl_type_A}. Since the fundamental representations in type \texttt{A} are minuscule, each LS-path model $\mathbb{B}(\omega_i)$, for $i \in [n]$, can be set-theoretically identified with $W/W_{P_i}$ and thus with the set $\mathrm{SSYT}(\omega_i)$ of all semistandard Young-tableaux consisting of one column with exactly $i$ boxes. Therefore, the set of all LS-tableaux of shape $\underline\mu$ can be interpreted as the set $\mathrm{YT}(\mu)$ of all Young-tableaux of shape $\mu$, \ie for all $i \in [n]$ they contain exactly $a_i$ many columns of length $i$. Notice that the order of the columns is reverted under this bijection. For example, the LS-tableau $(\pi_1, \pi_2, \pi_3, \pi_4)$ with the columns
\begin{align*}
    \pi_1 = (s_2 s_1 W_{P_1} \mkern1mu ; \mkern1mu 0, 1), \ \pi_2 = (s_3 s_2 W_{P_2} \mkern1mu ; \mkern1mu 0, 1), \ \pi_3 = (s_1 s_2 W_{P_2} \mkern1mu ;\mkern1mu  0, 1), \ \pi_4 = (s_3 W_{P_3} \mkern1mu ; \mkern1mu 0, 1)
\end{align*}
corresponds to the Young-tableau
\begin{align*}
    \vcenter{\hbox{\begin{ytableau} 1 & 2 & 1 & 3 \\ 2 & 3 & 4 \\ 4 \end{ytableau}}}
\end{align*}
A Young-tableau is called \textit{semistandard}, if its entries are weakly increasing along each row (from left to right) and strictly increasing along each column (from top to bottom). 

\begin{proposition}
    Under the above bijection, a Young-tableau $T \in \mathrm{YT}(\mu)$ is semistandard, if and only if its corresponding LS-tableau is standard.
\end{proposition}
\begin{proof}
    Let $1 \leq k_1 < \dots < k_m \leq n$ be the unique integers with $\set{k_1, \dots, k_m} = \set{i_1, \dots, i_s}$ and fix integers $1 \leq j_1 \leq \dots \leq j_s \leq m$ with $k_{j_r} = i_r$ for all $r \in [s]$. Let $Q$ be the intersection of the maximal parabolic subgroups $P_{k_1}, \dots, P_{k_m}$. Consider the underlying poset $\underline W = \coprod_{i=1}^m W/W_{P_{k_i}} \times \set{i}$ of the Seshadri stratification from Example~\ref{ex:strat_type_A}. From right to left the columns of $T$ can be seen as elements $(\theta_1, j_1), \dots, (\theta_s, j_s) \in \underline W$. The partial order $\geq$ on $\underline W$ is completely determined by the following property: The tableau $T$ is semistandard, if and only if $(\theta_1, j_1) \geq \dots \geq (\theta_s, j_s)$.

    In \cite[Lemma 11.8]{ownarticle} the author proved how the partial order on $\underline W$ can be lifted to $W/W_Q$: To each element $(\theta_r, j_r)$, $r \in [s]$, we can associate a lift $\widetilde\theta_r \coloneqq \min_Q \circ \max_{Q_{j_r}}(\theta_j)$ in $W/W_Q$. Then we have $(\theta_1, j_1) \geq \dots \geq (\theta_s, j_s)$, if and only if $\widetilde\theta_1 \geq \dots \geq \widetilde\theta_s$ and $j_1 \leq \dots \leq j_s$. The semistandardness of $T$ thus is equivalent to the existence of a defining chain for the corresponding LS-tableau.
\end{proof}

We remark that by reversing the order of the weights in the sequence $\underline\mu$ one obtains similar tableaux, which we call \textit{Anti-Young-tableaux}. The only difference to ordinary Young-tableaux is that the boxes align at the bottom and right instead of top and left. Again, an LS-tableau is standard, if and only if the corresponding Anti-Young-tableaux is semistandard, \ie its entries increase weakly along the rows from left to right and strictly along the columns from top to bottom (see. Figure~\ref{fig:anti-tableaux}).

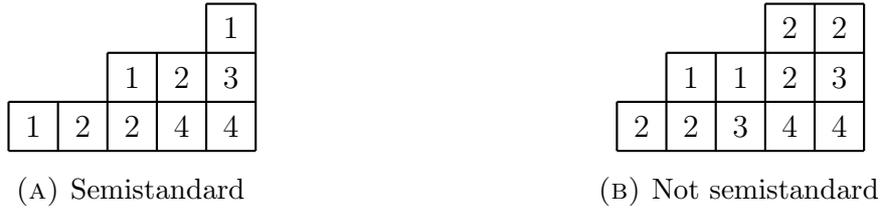
\begin{figure}
\centering
\begin{subfigure}{.5\textwidth}
    \begin{center}
        \begin{tikzpicture}[scale=0.65]
            \node at (-4.5,0.5) {$1$};
            \node at (-3.5,0.5) {$2$};
            \node at (-2.5,0.5) {$2$};
            \node at (-1.5,0.5) {$4$};
            \node at (-0.5,0.5) {$4$};
            \node at (-2.5,1.5) {$1$};
            \node at (-1.5,1.5) {$2$};
            \node at (-0.5,1.5) {$3$};
            \node at (-0.5,2.5) {$1$};
    
            \draw [thick] (0,0) -- (-5,0);
            \draw [thick] (0,0) -- (0,3);
            \draw [thick] (0,1) -- (-5,1);
            \draw [thick] (0,2) -- (-3,2);
            \draw [thick] (0,3) -- (-1,3);
            \draw [thick] (-1,0) -- (-1,3);
            \draw [thick] (-2,0) -- (-2,2);
            \draw [thick] (-3,0) -- (-3,2);
            \draw [thick] (-4,0) -- (-4,1);
            \draw [thick] (-5,0) -- (-5,1);
        \end{tikzpicture}
    \end{center}
    \caption{Semistandard}
    \label{fig:anti-ss}
\end{subfigure}%
\begin{subfigure}{.5\textwidth}
    \begin{center}
        \begin{tikzpicture}[scale=0.65]
            \node at (-4.5,0.5) {$2$};
            \node at (-3.5,0.5) {$2$};
            \node at (-2.5,0.5) {$3$};
            \node at (-1.5,0.5) {$4$};
            \node at (-0.5,0.5) {$4$};
            \node at (-3.5,1.5) {$1$};
            \node at (-2.5,1.5) {$1$};
            \node at (-1.5,1.5) {$2$};
            \node at (-0.5,1.5) {$3$};
            \node at (-1.5,2.5) {$2$};
            \node at (-0.5,2.5) {$2$};
    
            \draw [thick] (0,0) -- (-5,0);
            \draw [thick] (0,0) -- (0,3);
            \draw [thick] (0,1) -- (-5,1);
            \draw [thick] (0,2) -- (-4,2);
            \draw [thick] (0,3) -- (-2,3);
            \draw [thick] (-1,0) -- (-1,3);
            \draw [thick] (-2,0) -- (-2,3);
            \draw [thick] (-3,0) -- (-3,2);
            \draw [thick] (-4,0) -- (-4,2);
            \draw [thick] (-5,0) -- (-5,1);
        \end{tikzpicture}
    \end{center}
    \caption{Not semistandard}
    \label{fig:anti-no-ss}
\end{subfigure}
\caption{Anti-Young-tableaux for $n = 4$}
\label{fig:anti-tableaux}
\end{figure}

\vskip 10pt
{\noindent\textbf{Types \boldmath{}$\texttt{B}_{n}$ and \boldmath{}$\texttt{C}_{n}$:}} Instead of Young-tableaux, we obtain the tableau model developed by Lakshmibai, Musili and Seshadri (see~\cite{lakshmibaiGP4} and \cite{lakshmibaiGP5}). For a maximal parabolic subgroup $P_i$ they defined certain pairs of elements in $W/W_{P_i}$ called \textit{admissable pairs}.

For each $\theta \in W/W_{P_i}$ let $[X_\theta]$ denote the element in the Chow ring of $G/P_i$ induced by the Schubert variety $X_\theta$. Let $H$ be the unique Schubert variety of codimension one in $G/P_i$. By a formula of Chevalley from~\cite{demazure1974desingularisation} (see also \cite[Section 4.5]{seshadri2016introduction}) it holds
\begin{align*}
    [X_{\theta}] \cdot [H] = \sum_{\phi} d_\phi [X_\phi]
\end{align*}
in the Chow ring, where the sum is taken over all elements $\phi \in W/W_{P_i}$ covered by $\theta$. The number $d_\phi$ is given by $\vert \langle \phi(\omega_i), \beta^\vee \rangle \vert$, where $\beta$ is the unique positive root with $s_\beta \phi^{P_i} = \theta^{P_i}$ and $\phi^{P_i}$ (respectively $\theta^{P_i}$) is the unique minimal representative of $\phi$ (respectively $\theta$) in the Weyl group $W$. This number $d_\phi$ is called the \textit{(intersection) multiplicity} of $X_\phi$ in $X_\theta$ (sometimes also Chevalley multiplicity).

A pair $(\theta, \phi)$ of cosets $\theta, \phi \in W/W_{P_i}$ is called an \textit{admissable pair}, if either $\theta = \phi$ or there exists a chain $\theta = \phi_1 > \dots > \phi_k = \phi$ covering relations in $W/W_{P_i}$, such that for every $j = 2, \dots, k$ the Schubert variety $X_{\phi_j} \subseteq G/P_i$ is a divisor of $X_{\phi_{j-1}}$ with intersection multiplicity $2$. Note that these chains are a special case (for $a = \tfrac12)$ of $a$-chains defined by Littelmann in \cite[Section 2.2]{littelmann1994littlewood}, which play an important role in the definition of LS-paths. An admissable pair $(\theta, \phi)$ with $\theta > \phi$ thus corresponds to the LS-path $(\theta > \phi \mkern1mu ; \mkern1mu 0, \tfrac12, 1) \in \mathbb B(\omega_i)$ and an admissable pair $(\theta, \theta)$ corresponds to the LS-path $(\theta \mkern1mu ; \mkern1mu 0, 1) \in \mathbb B(\omega_i)$. Every LS-path in $\mathbb B(\omega_i)$ is of one of these two forms, since the fundamental weights $\omega_i$ in the types \texttt{B} and \texttt{C} are \textit{classical}, \ie $\vert \langle \omega_i, \beta^\vee \rangle \vert \leq 2$ holds for all roots $\beta$ in the root system of $G$. Equivalently, the intersection multiplicity of $X_\phi \subseteq X_\theta$ is at least $2$ for each covering relation $\theta > \phi$ in $W/W_{P_i}$.

A \textit{Young diagram} of type $(a_1, \dots, a_n)$ in the sense of~\cite{lakshmibaiGP5} can be seen as a sequence of admissable pairs $(\theta_j, \phi_j)$ with $j = 1, \dots, s$ and $\theta_j, \phi_j \in W/W_{P_{i_j}}$. Hence these Young diagrams correspond to LS-tableaux of shape $\underline\mu$. Under this correspondence the notions of standard Young diagrams and standard LS-tableaux agree, as both are given via the existence of a defining chain. Lakshmibai and Seshadri even allowed other orderings of the fundamental weights, but for the explicit ordering we defined above, Littelmann showed in the Appendix of~\cite{littelmann1990generalization} that one can interpret their Young diagrams via certain classical Young-tableaux with entries in $\set{1, \dots, 2n}$.

In type $\texttt{C}_n$ the Young diagrams from~\cite{lakshmibaiGP5} can also be interpreted as the (right canonical) symplectic tableaux introduced by De Concini~\cite{deconcini}. A symplectic tableau is essentially a Young-tableau of the group $\texttt{A}_{2n-1}$, so it consists of at most $2n-1$ rows and its entries lie in $[2n]$. Such a Young-tableau is symplectic, if every column $C$ is \textit{admissable}, \ie it can be split into a Young-tableau $(C_L, C_R)$ of two columns of the same length via the combinatorial method from \cite[Definition 2.1]{deconcini}. This ensures that $C_L$ and $C_R$ can be viewed as an admissable pair (as of~\cite{lakshmibaiGP5}) using the natural embedding $\texttt{C}_n \hookrightarrow \texttt{A}_{2n-1}$ from Dynkin diagram folding. Every symplectic tableau $(C_1, \dots, C_r)$ corresponds to a Young diagram of Lakshmibai, Musili and Seshadri via its \textit{split column form} $(C_{1, L}, C_{1, R} \dots, C_{r, L}, C_{r, R})$. In addition, a symplectic tableau is called \textit{standard}, if its split column form is semistandard, which is equivalent to the existence of a defining chain. We refer to~\cite{sheats} for a more formal and detailed explanation on the connection between these two tableau models.

\vskip 10pt
{\noindent\textbf{Type \boldmath{}$\texttt{D}_{n}$:}} Since the fundamental weights in type \texttt{D} are classical as well (see above), the Young diagrams from the types \texttt{B} and \texttt{C} can also be used in type \texttt{D} and these diagrams still correspond to LS-tableaux. For the ordering of the fundamental weights we chose above, these tableaux can again be identified with certain Young-tableaux, but their explicit combinatorial description is noticeably more difficult than in the types \texttt{B} and \texttt{C}. It can be found in~\cite[Appendix A.3]{littelmann1990generalization}. 

The main difference in type \texttt{D} is the fact that there exists no ordering of the fundamental weights, such that the notions of weakly standard LS-tableaux and standard LS-tableaux coincide (see Proposition~\ref{prop:nice_I_vs_weakly_standard} and Corollary~\ref{cor:totally_ordered_nice_I}). Therefore standardness cannot be verified locally by just considering consecutive columns.

\vskip 10pt
{\noindent\textbf{Other Types:}} In the exceptional types not every fundamental weight is classical. There is a list of all classical fundamental weights in \cite[Section A.2.3]{lakshmibai2007standard}. Since higher intersection multiplicities can occur, one needs to replace admissable pairs by admissable quadruples for fundamental weights $\omega_i$ with $\vert \langle \omega_i, \beta^\vee \rangle \vert \leq 3$ for all roots $\beta$. The resulting tableau model was described in \cite[Section 3]{littelmann1990generalization}. Again, the admissable quadruples correspond to LS-paths $(\sigma_p, \dots,  \sigma_1; 0, d_p, \dots, d_1 = 1)$ with $p \leq 4$ different directions. This shows the power of LS-paths, as they provide a language suited for all intersection multiplicities.

\clearpage

\section*{List of notations}
\pagestyle{empty}

\vskip 4pt

\begin{abbrv}
\item[{$\B(\lambda)$}] path model of LS-paths to $\lambda \in \Lambda^+$
\item[{$b_{p,q}$}] bond of a covering relation $p > q$
\item[{$\Delta$}] set of simple roots of $G$
\item[{$\hat X$}] multicone of an embedded projective variety $X$
\item[{$e_I$}] sum of all unit vectors $e_i$ for $i \in \underline I$
\item[{$\Gamma$}] fan of monoids
\item[{$\Gamma_C$}] monoid associated to a chain $C \subseteq A$
\item[{$\underline I$}] subset of $I$ defined by covering relations (p. \pageref{txt:def_underline_I})
\item[{$\Lambda^+$}] monoid of dominant weights of $G$
\item[{$\Lambda$}] weight lattice of $G$
\item[{$\lambda_I$}] sum of all weights $\lambda_i$ for $i \in \underline I$
\item[{$[k]$}] set $\set{1, \dots, k}$ for $k \in \N$
\item[{$\max_Q$}] maximal lift (p. \pageref{eq:def_min_Q_max_Q})
\item[{$\min_Q$}] minimal lift (p. \pageref{eq:def_min_Q_max_Q})
\item[{$\pi_Q$}] projection map (p. \pageref{eq:def_pi_Q})
\item[{$P_i$}] parabolic subgroup to the weight $\lambda_i$
\item[{$P_I$}] parabolic subgroup to $I \in \mathcal I$
\item[{$Q_I$}] lower parabolic subgroup to $I$ (p. \pageref{eq:def_Q_I_and_Q^I})
\item[{$Q^I$}] upper parabolic subgroup to $I$ (p. \pageref{eq:def_Q_I_and_Q^I})
\item[{$Q_\tau$}] largest parabolic over $Q$, where $\tau$ is $Q_\tau$-maximal (p. \pageref{txt:def_Q_tau})
\item[{$r(-)$}] rank function in a graded poset
\item[$\supp \underline a$] support of an element in some $\Q^A$ (p. \pageref{txt:def_support})
\item[{$\tau_i$}] projection of $\tau \in W/W_Q$ to $W/W_{P_i}$
\item[{$\tau_I$}] projection of $\tau \in W/W_Q$ to $W/W_{P_I}$
\item[{$W_\lambda$}] stabilizer of $\lambda \in \Lambda$ in the Weyl group
\item[{$W^Q$}] set of all minimal representatives of $W/W_Q$ in $W$
\item[{$W_Q$}] Weyl subgroup of $Q$
\item[{$X_\theta$}] Schubert variety in $G/Q$ to $\theta \in W/W_Q$
\end{abbrv}

\vskip 25pt

\printbibliography


\end{document}